\documentclass[reqno, 11pt]{amsart}
\usepackage{pdfsync}
\usepackage{url}
\usepackage[colorlinks, linkcolor=blue, citecolor=blue, urlcolor=blue]{hyperref}
\usepackage{times}

\newtheorem{theorem}{Theorem}[section]
\newtheorem{lemma}[theorem]{Lemma}
\newtheorem{proposition}[theorem]{Proposition}
\newtheorem{corollary}[theorem]{Corollary}

\theoremstyle{definition}

\newtheorem{ex}[theorem]{Example}

\newtheorem{remark}[theorem]{Remark}

\numberwithin{equation}{section}

\usepackage{bbm}
\usepackage{url}
\usepackage{euscript}
\usepackage{pb-diagram}
\usepackage{lamsarrow}
\usepackage{pb-lams}
\usepackage{amsmath}
\usepackage{amsthm}

\usepackage{amsxtra}
\usepackage{amssymb}
\usepackage{pifont}
\usepackage{amsbsy}

\usepackage{graphicx}
\usepackage{epstopdf}

\newskip\aline \newskip\halfaline
\def\skipaline{\vskip\aline}

\def\qedbox{$\rlap{$\sqcap$}\sqcup$}
\def\qed{\nobreak\hfill\penalty250 \hbox{}\nobreak\hfill\qedbox\skipaline}

\def\proofend{\eqno{\mbox{\qedbox}}}

\newcommand{\one}{\mathbbm{1}}

\newcommand\bC{{\mathbb C}}

\newcommand\bR{{\mathbb R}}

\newcommand{\bT}{{\mathbb T}}

\newcommand\bZ{{\mathbb Z}}

\DeclareMathOperator{\re}{\mathbf{Re}}
\DeclareMathOperator{\im}{\mathbf{Im}}

\DeclareMathOperator{\tr}{{\rm tr}}

\DeclareMathOperator{\Gr}{\mathbf{Gr}}
 \DeclareMathOperator{\Hom}{Hom}

\DeclareMathOperator{\spa}{span}

\DeclareMathOperator{\ev}{\mathbf{ev}}

\DeclareMathOperator{\Sym}{\mathbf{Sym}}

\DeclareMathOperator{\Hess}{Hess}

\DeclareMathOperator{\var}{\boldsymbol{var}}

\newcommand{\ba}{\boldsymbol{a}}

\newcommand{\be}{{\boldsymbol{e}}}
\newcommand{\bsf}{\boldsymbol{f}}

\newcommand{\ii}{\boldsymbol{i}}
\newcommand{\jj}{\boldsymbol{j}}
\newcommand{\kk}{\boldsymbol{k}}

\newcommand{\bp}{\boldsymbol{p}}
\newcommand{\br}{\boldsymbol{r}}
\newcommand{\bs}{\boldsymbol{s}}
\newcommand{\bt}{\boldsymbol{t}}
\newcommand{\bu}{{\boldsymbol{u}}}
\newcommand{\bv}{{\boldsymbol{v}}}
\newcommand{\bw}{{\boldsymbol{w}}}
\newcommand{\bx}{{\boldsymbol{x}}}
\newcommand{\by}{{\boldsymbol{y}}}

\newcommand{\bsA}{\boldsymbol{A}}
\newcommand{\bsB}{\boldsymbol{B}}
\newcommand{\bsC}{\boldsymbol{C}}

\newcommand{\bsE}{\boldsymbol{E}}

\newcommand{\bsI}{\boldsymbol{I}}
\newcommand{\bsJ}{\boldsymbol{J}}
\newcommand{\bsK}{{\boldsymbol{K}}}
\newcommand{\bsL}{\boldsymbol{L}}
\newcommand{\bsM}{\boldsymbol{M}}

\newcommand{\bsT}{\boldsymbol{T}}

\newcommand{\bsV}{\boldsymbol{V}}
\newcommand{\bsW}{\boldsymbol{W}}

\newcommand{\bgamma}{\boldsymbol{\gamma}}

\newcommand{\bmu}{\boldsymbol{\mu}}

\newcommand{\bom}{\boldsymbol{\omega}}
\newcommand{\bsi}{\boldsymbol{\sigma}}

\newcommand{\bPsi}{\boldsymbol{\Psi}}
\newcommand{\bXi}{\boldsymbol{\Xi}}

\newcommand{\si}{{\sigma}}

\newcommand{\ve}{{\varepsilon}}

\newcommand{\vfi}{{\varphi}}

\newcommand{\eA}{\EuScript{A}}
\newcommand{\eB}{\EuScript{B}}
\newcommand{\eC}{\EuScript{C}}
\newcommand{\eD}{\EuScript{D}}

\newcommand{\eF}{\EuScript{F}}
\newcommand{\eG}{\EuScript{G}}
\newcommand{\h}{\EuScript H}
\newcommand{\eH}{\EuScript H}

\newcommand{\eM}{\EuScript{M}}
\newcommand{\eN}{\EuScript{N}}

\newcommand{\eR}{\EuScript{R}}

\newcommand{\eT}{\EuScript{T}}
\newcommand{\eU}{\EuScript{U}}
\newcommand{\eV}{\EuScript{V}}

\newcommand{\eX}{\EuScript{X}}
\newcommand{\eY}{\EuScript{Y}}
\newcommand{\eZ}{\EuScript{Z}}

\newcommand{\ra}{\rightarrow}

\newcommand{\Lra}{{\longrightarrow}}

\newcommand{\Llra}{{\Longleftrightarrow}}

\newcommand{\lan}{\langle}
\newcommand{\ran}{\rangle}

\def\inpr{\mathbin{\hbox to 6pt{\vrule height0.4pt width5pt depth0pt \kern-.4pt \vrule height6pt width0.4pt depth0pt\hss}}}

\newcommand{\dt}{\frac{d}{dt}}

\newcommand{\pa}{\partial}

\newcommand{\dual}{{\spcheck{}}}

\newcommand{\hP}{\widehat{P}}
\newcommand{\ha}{\widehat{\boldsymbol{A}}}
\newcommand{\hb}{\widehat{\boldsymbol{B}}}
\newcommand{\bla}{\bar{\lambda}}
\newcommand{\Ri}{R_\infty}

\begin{document}

\title{Critical sets of random  smooth functions on products of spheres}

\date{Started  June 8, 2010. Completed  on December 4, 2010.
Last modified on {\today}. }

\author{Liviu I. Nicolaescu}

\address{Department of Mathematics, University of Notre Dame, Notre Dame, IN 46556-4618.}
\email{nicolaescu.1@nd.edu}
\urladdr{\url{http://www.nd.edu/~lnicolae/}}
\subjclass[2000]{Primary     15B52, 42C10, 53C65, 58K05, 60D05, 60G15, 60G60 }
\keywords{Morse functions, critical points,   Chern-Lashof formula, Kac-Price formula,  spherical harmonics, random matrices, stationary gaussian processes}

\begin{abstract}  We prove  a Chern-Lashof type formula   computing the        expected number of critical points of   smooth function  on a  smooth manifold $M$ randomly chosen from a  finite dimensional subspace $\bsV\subset C^\infty(M)$ equipped with a Gaussian  probability measure.  We then use this formula  to  find the asymptotics  of  the expected number of critical points of a random linear combination of a large  number eigenfunctions of the Laplacian on the  round sphere, tori,  or  a products of  two round spheres.  In the case  $M=S^1$ we show that the number of critical points of a trigonometric  polynomial of degree $\leq \nu$ is a random variable $Z_\nu$ with expectation $\bsE(Z_\nu)\sim 2\sqrt{0.6}\,\nu$ and  variance $\var(Z_\nu)\sim c\nu$  as $\nu\ra \infty$, $c\approx 0.35$.

\end{abstract}

\maketitle

\tableofcontents

\section*{Introduction}
\setcounter{equation}{0}

Suppose that $M$ is a compact, connected smooth manifold of dimension $m$. Given a finite dimensional vector space   $\bsV\subset C^\infty(M)$ of dimension $N$ we would like  to know  the average (expected)  size   of the critical set     of a function $\bv\in \bsV$.   For the  applications we have in mind  $N\gg m$.  We will refer to $\bsV$ as the \emph{sample space} and we will denote by $\mu(\bv)$  the number of critical points of the function $\bv\in \bsV$.

 More explicitly,  we  fix  a Euclidean inner product  $h$  on $\bsV$ and we denote by  $S(\bsV)$   the unit sphere in $\bsV$.   We define  \emph{expected number of critical points}   of a random function in $\bsV$ to be the  quantity 
  \begin{equation*}
  \begin{split}
 \mu(M, \bsV, h)& :=\frac{1}{{\rm area}\, (\,S(\bsV)\,)}\int_{S(\bsV)} \mu (\bv) |dS_h(\bv)|\\
 &\stackrel{ (\ref{eq: gauss})}{=}\frac{1}{(2\pi\si^2)^{\frac{N}{2}}}\int_{\bsV} e^{-\frac{|\bv|^2}{2\si^2} }\mu(\bv)\,|dV_h(\bv)|,\;\;\forall \si>0.
 \end{split}
 \tag{$\boldsymbol{\mu}_0$}
 \label{tag: mu0}
 \end{equation*}
 In other words,  $\mu(M,\bsV, h)$ is the expectation of the random variable $Z_{\bsV,h}$
\begin{equation*}
S(\bsV)\ni\bv\mapsto Z_{\bsV,h}:=\mu(\bv),
\tag{$\boldsymbol{Z}$}
\label{tag: Z}
\end{equation*}
where $S(\bsV)$ is equipped with the probability measure determined by the suitably rescaled  area  density determined by the metric $h$. 

 Let us observe that we can cast  the above setup in the framework of Gaussian  random fields, \cite{AT, CL}.  Fix an orhonormal  basis $(\Psi_\alpha)_{1\leq \alpha\leq N}$ of $(\bsV, h)$,   and a Gaussian probability  measure on $\bsV$,
 \[
 \bgamma= \frac{1}{(2\pi)^{\frac{N}{2}}}e ^{-\frac{|\bv|^2}{2}} |d\bv|.
 \]
 Then the functions $\xi_\alpha:\bsV\ra \bR$, $\bv\mapsto \xi_\alpha(\bv)=(\bv,\Psi_\alpha)$, are  independent,  normally distributed random variables  with  mean $0$ and variance $1$,  and   the equality
 \[
 \bv(\bx)=\sum_\alpha \xi_\alpha(\bv)\Psi_\alpha(\bx)
 \]
 defines  an $\bR$-valued centered Gaussian random field on $M$ with covariance kernel
 \[
 K_{\bsV}(\bx,\by)=\sum_\alpha \Psi_\alpha(\bx)\Psi_\alpha(\by),\;\;\forall \bx,\by\in M.
 \]
Thus, we are seeking the expectation of the number of critical points of a sample function of this field. However, in this paper, most of the time, this  point of view,   will only stay in the background.  A notable exception is Theorem \ref{th: var} whose proof   relies in an essential way on results from the theory of stationary gaussian processes.
 
 It  is possible that all the functions in $\bsV$ have infinite  critical sets, in which case the  integrals in (\ref{tag: mu0}) are infinite.  To avoid this problem  we impose an \emph{ampleness}  condition on $\bsV$.  More precisely, we require that for any point $\bx\in M$, and any covector $\xi\in T^*_\bx M$ there exists a function $\bv\in\bsV$ whose differential at $\bx$ is $\xi$.  As explained in \cite[\S 1.2]{N1}, this condition implies  that     almost all functions $\bv\in \bsV$ are  Morse functions and thus   have finite critical sets.

 The above ampleness condition can be given a different interpretation by introducing the evaluation map 
 \[
 \ev=\ev^{\bsV}: M\ra \bsV\dual:=\Hom(\bsV,\bR), \;\;\bx\mapsto\ev_\bx, 
 \]
 where for any  $\bx\in M$ the linear map $\ev_\bx:\bsV\ra \bR$ is given by
 \[
 \ev_\bx(\bv)=\bv(x),\;\;\forall \bv\in \bsV.
 \]
 The ampleness condition is equivalent with the requirement that the evaluation map be   an immersion.

 This places us    in the setup considered by J. Milnor \cite{Mil}  and Chern-Lashof \cite{CL}.  These authors investigated immersions of a compact manifold $M$ in an Euclidean space $\bsE$, and they computed    the average number of critical points    of the pullback to $M$ of a  random linear function on $\bsE$.   That is precisely our problem with  $\bsE=\bsV\dual$ since a function $\bv\in \bsV$ can be viewed  canonically as a linear   function on $\bsV\dual$.

The papers \cite{CL, Mil} contain a philosophically satisfactory answer to our initial question.  The expected number of critical points   is, up to a universal factor,  the integral over $M$   of a certain scalar called the \emph{total curvature} of the immersion and  canonically determined by the second fundamental form of the immersion.

Our interests are  a bit more pedestrian since we are literally interested in estimating the expected number of critical points when $\dim \bsV\ra \infty$. In our applications, unlike the situation   analyzed in \cite{CL, Mil}, the metric on $M$ is not induced by a metric on $\bsV$, but the other way around. We typically have a natural  Riemann  metric on $M$ and then we use it to induce a metric on $\bsV$,  namely, the restriction of the $L^2$-metric on $C^\infty(M)$ defined by our  Riemann metric on $M$.  The first theoretical   goal of this paper is to  rewrite the results  in \cite{CL, Mil} in  a computationally friendlier form.

More precisely, we would like to   describe a density  $|d\mu|$ on $M$ such that
\[
\mu(M,\bsV,h)=\frac{1}{{\rm area}\,(S(\bsV))} \,\int_M |d\mu|=\frac{1}{\bsi_{N-1}} \int_M |d\mu| ,
\]
where  $|d\mu|$ captures  the infinitesimal behavior of the family of functions   $\bsV$.

In  Corollary \ref{cor: av}  we    describe such a density on $M$  by relying on a standard trick in  integral geometry.   Our approach is different  from the probabilistic method   used in the proof of the closely related  result,  \cite[Thm. 4.2]{BSZ2}, which is a higher dimensional version of  a technique pioneered by   M. Kac-and S. Rice,  \cite{AT, Kac, Rice}. 

It is easier to explain Corollary \ref{cor: av}  if we  fix a metric $g$ on $M$. The density $|d\mu|$ can be written as $|d\mu|=\rho_g|dV_g|$, for some smooth function $\rho_g:M\ra [0,\infty)$. For $\bx\in M$, the number  $\rho_g(\bx)$ captures the average infinitesimal  behavior of the family $\bsV$ at $\bx$.    Here is the explicit description of $\rho_g(\bx)$

Denote by $\bsK_\bx$ the subspace  of $\bsV$ consisting of  the functions that admit $\bx$ as a critical point.   Let $S(\bsK_\bx)$ denote the unit sphere in $\bsK_\bx$ defined by the metric $h$ on $\bsV$.  Any function $\bv\in \bsK_\bx$ has a well-defined Hessian at $\bx$, $\Hess_\bx(\bv)$,  that can be identified via the metric $g$  with a symmetric linear operator
\[
\Hess_\bx(\bv, g)=\Hess_\bx(\bv, g) : T_\bx M\ra T_\bx M.
\]
We set
\begin{equation*}
\Delta_\bx(\bsV):=\int_{S(\bsK_\bx)} |\det\Hess_\bx(\bv)|\,|dS(\bv)|\stackrel{(\ref{eq: gauss})}{=}\frac{2}{\Gamma(\frac{N}{2})}\int_{\bsK_\bx} |\det \Hess_\bx(\bv)|  e^{-|\bv|^2}\,|d\bv|,
\tag{$\boldsymbol{\Delta}$}
\label{tag: bdel}
\end{equation*}
where $\Gamma$  denotes the Gamma function. The differential  of the evaluation map at $\bx$ is a linear map $\eA_\bx^\dag: T_\bx M\ra \bsV\dual$, and we denote by $J_g(\eA^\dag_\bx)$ its Jacobian, i.e., the norm of the induced linear map $\Lambda^m\eA^\dag_\bx: \Lambda^m T_\bx M\ra \Lambda^m\bsV\dual$. Then
\[
\rho_g(\bx)=\frac{\Delta_\bx(\bsV,g)}{J(\eA^\dag_\bx)},
\]
and thus
\begin{equation*}
\begin{split}
\mu(M,\bsV,h) & =\frac{1}{\bsi_{N-1}} \int_M  \frac{\Delta_\bx(\bsV)}{J_g(\eA^\dag_\bx)}\,|dV_g(\bx)|\\
& =\pi^{-\frac{m}{2}}\int_M\frac{1}{J_g(\eA^\dag_\bx)}\left(\int_{\bsK_\bx}|\det \Hess_\bx(\bv)| \, \frac{e^{-|\bv|^2}}{\pi^{\frac{\dim\bsK_\bx}{2}}}   |dV_{\bsK_\bx}(\bv)|\,\right)\,|dV_g(\bx)|\\
&=(2\pi)^{-\frac{m}{2}}\int_M\frac{1}{J_g(\eA^\dag_\bx)}\left(\int_{\bsK_\bx}|\det \Hess_\bx(\bv)| \, \frac{e^{-\frac{|\bv|^2}{2}}}{(2\pi)^{\frac{\dim\bsK_\bx}{2}}}   |dV_{\bsK_\bx}(\bv)|\,\right)\,|dV_g(\bx)|.
\end{split}
\tag{$\boldsymbol{\mu}_1$}
\label{tag: bmu}
\end{equation*}
We want to emphasize that the density $\rho_g|dV_g|$ is \emph{independent} of the metric $g$, but \emph{it does depend on the metric $h$} on $\bsV$. In particular, the expectation $\mu(M,\bsV, h)$ does depend on the choice of metric  $h$.  For all the applications we have in mind,  the metric $h$ on $\bsV$ is obtained from a metric $g$ on $M$   in the fashion explained above.   In this case we will use the notation $\mu(M,\bsV, g)$. Remark \ref{rem: depend}  contains  a rather dramatic  illustration  of what happens when  $h$ is induced by a Sobolev metric other than $L^2$.

 We can simplify (\ref{tag: bmu}) even more. Define a new  metric $\tilde{g}$ on $M$, the pullback  via the evaluation map $\ev: M\ra \bsV$ of the   metric on $\bsV$. Tautologically,
 \[
 J_{\tilde{g}}(\eA^\dag_\bx)=1,\;\;\forall \bx\in M.
 \]
 We denote $\tilde{\nabla}$ the Levi-Civita connection of $\tilde{g}$. For every $\bv\in \bsV$ and $\bx\in M$ we define the Hessian of $\bv$ at $\bx$ with respect to the  metric $\tilde{g}$ to be the  symmetric bilinear form
\[
\widetilde{\Hess}_{\bx}(\bv) : T_\bx M\times T_\bx M\ra \bR,
\]
\[
 \widetilde{\Hess}_{\bx}(\bv)(X,Y)= (XY\bv)(x) -(\tilde{\nabla}_X Y\bv) (x),\;\;\forall X,Y\in T_\bx M.
 \]
 As explained  in \cite[\S 12.2.2]{AT}, for every $\bx\in M$ the random vectors $\bv\mapsto \widetilde{\Hess}_\bx(\bv)$ and $\bv\mapsto d\bv(\bx)$ are  independent. The term
\[
\int_{\bsK_\bx}|{\det}_{\tilde{g}} \widetilde{\Hess}_\bx(\bv)| \, \frac{e^{-\frac{|\bv|^2}{2}}}{(2\pi)^{\frac{\dim\bsK_\bx}{2}}} \, |dV_{\bsK_\bx}(\bv)|
\]
can be given a probabilistic interpretation: it is the conditional expectation of the random variable $\bsV\ni\bv \mapsto |\det_{\tilde{g}} \widetilde{\Hess}_\bx(\bv)|$ with respect to the  conditional density  determined by the condition $d\bv(\bx)=0$. From this point of view, (\ref{tag: bmu}) resembles the Metatheorem \cite[Thm. 11.2.1]{AT}.
 Hence
 \[
\int_{\bsK_\bx}|\det \Hess_\bx(\bv)| \, \frac{e^{-\frac{|\bv|^2}{2}}}{(2\pi)^{\frac{\dim\bsK_\bx}{2}}}   |dV_{\bsK_\bx}(\bv)|=\int_{\bsV}|{\det}_{\tilde{g}} \widetilde{\Hess}_\bx(\bv)| \, \frac{e^{-\frac{|\bv|^2}{2}}}{(2\pi)^{\frac{\dim\bsV}{2}}}   |d\bv|, 
 \]
 and
\[
\mu(M,\bsV, h) =(2\pi)^{-\frac{\dim M}{2}}\int_M\left(\int_{\bsV}|{\det}_{\tilde{g}} \widetilde{\Hess}_\bx(\bv)| \, \frac{e^{-\frac{|\bv|^2}{2}}}{(2\pi)^{\frac{\dim\bsV}{2}}}   |d\bv|\,\right) dV_{\tilde{g}}(\bx).
\tag{$\bmu_2$}
\]
The equality  (\ref{tag: bmu}) can  be used  in some instances  to compute the variance of the number of critical points of a random function in $\bsV$. More precisely, to a sample space  $\bsV\subset C^\infty(M)$ we associate  a sample space $\bsV_\Delta\subset C^\infty(M\times M)$, the image  of $\bsV$ via the  diagonal injection 
\[
\Delta: C^\infty(M)\ra C^\infty(M\times M), \;\; \bv\mapsto \bv^\Delta,\;\; \bv^\Delta(\bx,\by)=\bv(x)+\bv(\by),\;\;\forall \bv\in\bsV,\;\;\bx,\by\in M.
\]
We equip $\bsV^\Delta$ with the metric $h_\Delta$ so that   the map $\Delta:(\bsV,h)\ra (\bsV^\Delta,h_\Delta)$ is an isometry.   There is a  slight problem. The sample space $\bsV^\Delta$ is \emph{never}  ample, no matter how large we choose $\bsV$.  More precisely, the ampleness is always violated along the diagonal $\Delta_M\subset M\times M$.  We denote  by $\bsV^\Delta_*$ the space of  restrictions  to $M^2_*:M\times M\setminus \Delta_M$ of the functions in $\bsV^\Delta$.   Note that for any $\bv\in \bsV$ we have
\[
\mu(\bv^\Delta|_{M^2_*})=\mu(\bv)^2-\mu(\bv).
\]
If $\bsV^\Delta$ is sufficiently large, the sample space $\bsV^\Delta_*$ is ample and we deduce that 
\begin{equation*}
\bsE\bigl(\,Z_{\bsV,h}^2-Z_{\bsV,h}\,\bigr)=\bsE\bigl(\, Z_{\bsV^\Delta_*,h_\Delta}\,\bigr),
\tag{$\bsM_2$}
\label{tag: mom}
\end{equation*}
where $\bsE$ denotes the expectation of a random variable, and   $Z_{-,-}$ are defined as in (\ref{tag: Z}).  The  quantity   in  the lehft-hand side of (\ref{tag: mom}) is the so called  \emph{second combinatorial momentum} of the random variable $Z_{\bsV,h}$.

We also want to point out  that if  we remove the absolute value from the integrand $\Hess(\bv)$ in (\ref{tag: bmu}),  then    we obtain a Gauss-Bonnet type theorem
\[
\chi(M)=\frac{1}{(2\pi)^{\frac{N}{2}}}\int_M\frac{1}{J_g(\eA^\dag_\bx)}\left(\int_{\bsK_\bx}\det \Hess_\bx(\bv) \, e^{-\frac{|\bv|^2}{2}}\,|d\bv|\right)\,|dV_g(x)|,
\]
where $\chi(M)$ denotes the Euler characteristic of $M$. 

Most of our  applications involve  sequences of subspaces $ \bsV_n\subset  C^\infty(M)$  such that $\dim \bsV_n\ra \infty$,  and   we investigate the asymptotic  behavior of $\mu (M,\bsV_n,g)$ as $n\ra \infty$, where $g$ is a   metric on $M$.     One difficulty in applying  (\ref{tag: bmu}) comes from the definition (\ref{tag: bdel})  which involves integrals over  spheres of arbitrarily large dimensions.      There is a simple way of dealing with this issue when     $\bsV$ is \emph{$2$-jet  ample}, that is, for any $\bx\in M$, and any $2$-jet  $j_\bx$ at $\bx$, there exists  $\bv\in \bsV$ whose $2$-jet at $\bx$ is $j_\bx$.

Denote $\Sym^M_\bx$,  the space of   selfadjoint linear operators $(T_\bx M,g)\ra (T_\bx M,g)$.  In this case, the  linear map $\Hess: \bsK_\bx\ra \Sym^M_\bx$
is onto.  The pushforward  by $\Hess_\bx$  of the Gaussian probability measure  $\bgamma_\bx$ on $\bsK_\bx$ 
\[
\bgamma_\bx = \frac{e^{-\frac{|\bv|^2}{2}}}{(2\pi)^{\frac{\dim\bsK_\bx}{2}}}|dV_{\bsK_\bx}(\bv)|,
\]
is a (centered) Gaussian probability measure $\widehat{\bgamma}^{\bsV}_\bx$ on $\Sym^M_\bx$; see \cite[\S 16]{Shi}. In particular, $\widehat{\bgamma}_\bx^{\bsV}$ is uniquely determined by it covariance matrix.  This  is a symmetric, positive  define  linear operator 
\[
\eC^{\bsV}_\bx:\Sym^M_\bx\ra \Sym_\bx^M.
\]
 We can then rewrite (\ref{tag: bmu}) as 
\begin{equation*}
\mu(M,\bsV, h)= (2\pi)^{-\frac{m}{2}}\int_M\frac{1}{J_g(\eA^\dag_\bx)}\left(\int_{\Sym^M_\bx}|\,\det H\,|\,  |d\widehat{\bgamma}_\bx^{\bsV}(H)|\right)\,|dV_g(\bx)|.
\tag{$\bgamma$}
\label{tag: bmup0}
\end{equation*}
This is very similar to the integral formula employed by Douglas-Shiffman-Zelditch, \cite{DSZ1, DSZ2}, in their  investigation of   critical sets  of random holomorphic sections of (ample) holomorphic line bundles.

In concrete situations a more ad-hoc method may  be more suitable. 
Suppose that  for every $\bx$ we can find a subspace $\bsL_\bx\subset\bsK_\bx$ of dimension $\ell(\bx)$, such that for any $\bv\in \bsK_\bx$, $\bv\perp\bsL_\bx$   we have $\Hess_\bx(\bv)=0$.  Noting that $\dim \bsK_\bx= \dim\bsV-\dim M =N-m$  and $\dim \bsL_\bx^\perp = N-m-\ell(\bx)$  we obtain
\[
\int_{\bsK_\bx} e^{-|\bv|^2}|\det \Hess_\bx(\bv)|\, |dV(\bv)|= \left(\int_{\bsL_\bx^\perp} e^{-|\bu|^2}\,|dV(\bu)|\right)\times
\]
\[
\times\left(\int_{\bsL_\bx} e^{-|\bw|^2} |\det \Hess_\bx(\bw)|\,|dV(\bw)|\right)
\]
\[
=\pi^{\frac{N-m-\ell(\bx)}{2}} \underbrace{\left(\int_{\bsL_\bx}  |\det \Hess_\bx(\bw)|\,e^{-|\bw|^2}\,|dV(\bw)|\right)}_{=:\Delta(\bsL_\bx)}.
\]
Using (\ref{tag: si})   we can now rewrite (\ref{tag: bmu}) as
\begin{equation*}
\mu(M,\bsV, h)= \pi^{-\frac{m}{2}}\int_M\pi^{-\frac{\ell(\bx)}{2}} \frac{\Delta(\bsL_\bx)}{J_g(\eA_\bx^\dag)}\, |dV_g(\bx)|.
\tag{$\boldsymbol{\mu}_3$}
\label{tag: bmup}
\end{equation*}

Our  first application of  formula (\ref{tag: bmup}) is in the case $M=S^{d-1}$ and $g$ is the round metric $g_d$ of radius $1$ on the $(d-1)$ sphere. The eigenvalues of the Laplacian $\Delta_d$ on $S^{d-1}$ are
\[
\lambda_n(d)=n(n+d-2),\;\;n=0,1,2,\dotsc.
\]
For any nonnegative integer $n$, and any positive real number $\nu$, we set
\[
\eY_{n,d}:=\ker (\Delta_d-\lambda_n(d))\;\; \bsV_\nu(d):=\bigoplus_{n\leq \nu}eY_{n,d}.
\]
In  Theorem  \ref{th: sph-harm}  and Corollary \ref{cor: av-circle} we show that for any $d\geq 2$ there exists a universal constant $K_d>0$ such that
\begin{equation*}
\mu(S^{d-1}, \bsV_\nu(d), g_d)\sim K_d\dim \bsV_\nu(d)\sim\frac{2K_d}{(d-1)!}\nu^{d-1}\;\;\mbox{as}\;\;\nu\ra \infty.
\tag{$\bsA$}
\label{tag: A}
\end{equation*}
The proportionality constant $K_d$   has an explicit description  as an integral over the Gaussian ensemble of real, symmetric $(d-1)\times(d-1)$-matrices.

In Theorem \ref{th: harm2}  we    concentrate on the  space $\eY_{n,2}$ of spherical harmonics of degree $n$ on the $2$-sphere and we  show that as $n\ra \infty$
\[
\mu(S^2, \eY_{n,2}) \sim \frac{2}{\sqrt{3}} n^2.
\tag{$\bsB$}
\label{tag: c}
\]
If we denote by  $\zeta_n$  the  expected number $\zeta_n$ of nodal domains of a random spherical harmonic of degree $n$ then, according to the recent work of Nazarov and Sodin, \cite{NS}, there exists a positive constant   $a$ such that
\[
\zeta_n\sim an^2\;\;\mbox{as}\;\;n\ra \infty.
\]
The estimate (\ref{tag: c}) implies that $a\leq \frac{1}{\sqrt{3}}\approx 0.5773$. The classical estimates of Pleijel, \cite{Plei}, and Peetre, \cite{Pee}, imply that  $a\leq  \frac{4}{j_0^2}\approx 0.6916$, where $j_0$ denotes the first positive zero of the Bessel function $J_0$.

We next consider various spaces of trigonometric polynomials on an $L$-dimensional torus $\bT^L$.  To a  finite subset $\eM\subset \bZ^L$ we associate the space $\bsV(\eM)$ of trigonometric polynomials on $\bT^L$ spanned by the ``monomials''  
\[
\cos(m_1\theta_1+\cdots + m_L\theta_L), \;\;\sin(m_1\theta_1+\cdots + m_L\theta_L),  \;\;(m_1,\dotsc, m_L)\in\eM,
\]
 and   in Theorem  \ref{th: trig-newton} we give a formula for the expected number $\mu(\eM)$ of critical points  of a    trigonometric polynomial in  $\bsV(\eM)$.   We consider  the special case when
\[
\eM=\eM_\nu^L:=\bigl\{\,(m_1,\dotsc, m_L)\in \bZ^L;\;\;|m_i|\leq \nu,\;\;\forall i=1,\dotsc, L\,\bigr\}
\]
 and in Theorem \ref{th: tor-complexity} we show   that as $\nu\ra \infty$ we have
\[
\mu(\eM_\nu^L)\sim \left(\frac{\pi}{6}\right)^{\frac{L}{2}}\lan |\det X|\ran_\infty \times \dim\bsV(\eM_\nu^L).
\]
 Above, $\lan |\det X|\ran_\infty$ denotes the expected value of  the absolute value of  random symmetric $L\times L$ matrix, where the space $\Sym_L$ of such matrices is equipped with a certain gaussian probability measure that we describe explicitly. In particular, when $L=1$, we have
\begin{equation*}
\mu(\eM_\nu^1)\sim \sqrt{\frac{3}{5}}\dim\bsV(\eM_\nu^1)=2\nu \sqrt{\frac{3}{5}},
\tag{$\bsE$}
\label{tag: E}
\end{equation*}
while  for $L=2$ we have
\[
\mu(\eM_\nu^2)\sim z_2\dim\bsV(\eM_\nu^2). 
\]
The proportionality  constant $z_2$   can be given an explicit, albeit complicated description in terms of elliptic functions.  In particular,
\[
z_2\approx 0.4717....
\]
In the case $L=1$ we were able to prove a bit more.   We denote  by $Z_\nu$ the number  of  critical  points of a random trigonometric  polynomial  in $\eM_\nu^1$.  Then $Z_\nu$  is a random variable with expectation  $\bsE(Z_\nu)$ satisfying the asymptotic behavior (\ref{tag: E}). In Theorem \ref{th: var}  we prove that its variance  satisfies the asymptotic behavior
\begin{equation*}
\var(Z_\nu)\sim \delta_\infty  \nu,
\tag{$\bsV$}
\label{tag: V}
\end{equation*}
where $\delta_\infty$   is a positive constant ($\delta_\infty\approx 0.35$)  described explicitly  by an integral formula, (\ref{eq: varinfi}).

We also  compute the average number of critical points  of a  real trigonometric polynomial in two variables of the form
\[
\bigl\{a\cos x+b\sin x+ c\cos y+d\sin y+ p\cos(x+y)+q\sin(x+y)\,\bigr\}.
\]
This family of trigonometric  polynomials was investigated by  V.I. Arnold in \cite{Ar06} where he proves that a typical polynomial of this form   has   at most $8$ critical points. In Theorem \ref{th: arn} we prove that the average  number of critical points of a trigonometric polynomial in this family is $\frac{4\pi}{3}\approx 4.188$.    Note that the minimum number of  critical points  of  Morse function on the $2$-torus is $4$, and the above average is very close to this  minimal number.

 We then consider  products of  spheres  $S^{d_1-1}\times S^{d_2-1}$ equipped with the product of the round metrics $g_{d_1}\times g_{d_2}$. In Theorems \ref{th: prod-sph} and  \ref{th: circle-sph} we show that, for any $d_1,d_2\geq 2$, there exists a constant $K_{d_1,d_2}>0$  such that, for any $r\geq 1$,  as $\nu\ra \infty$, we have
\begin{equation*}
\mu\bigl(\, S^{d_1-1}\times S^{d_2-1}, \bsV_{\nu^r}(d_1)\otimes \bsV_\nu(d_2)\,\bigr)\sim K_{d_1,d_2} \bigl(\dim \bsV_{\nu^r}(d_1)\otimes \bsV_\nu(d_2)\,\bigr)^{\varpi(d_1,d_2, r)},
\tag{$\bsC$}
\label{tag: B}
\end{equation*}
where
\begin{equation*}
\varpi(d_1,d_2, r)=\begin{cases}
1, & (d_1-2)(d_2-2)=0.\\
&\\
\frac{(d_1-3)r +d_2+1}{(d_1-1)r+ d_2-1}, & (d_1-2)(d_2-2)\neq 0.
\end{cases}
\tag{$\boldsymbol{\varpi}$}
\label{tag: vpi}
\end{equation*}
Let us point out that for $d_1,d_2>2$, the function $r\mapsto \varpi(d_1,d_2,r)$, $r\geq 1$,  is decreasing, nonnegative,
\[
\lim_{r\ra \infty}\varpi(d_1,d_2,r)=\frac{d_1-3}{d_2-1}=\kappa(d_1,d_2)\;\;\mbox{and}\;\;\varpi(d_1,d_2,r=1)=1.
\]
In particular,
\[
\varpi(d_1,d_2,r) <1,\;\;\forall r>1.
\]
More surprisingly,
\[
\varpi(d_1,d_2,r)=\varpi(d_2,d_1,r)=1\;\;\mbox{if}\;\; (d_1-2)(d_2-2)=0,
\]
but this symmetry is lost if $(d_1-2)(d_2-2)\neq 0$.

We  find the asymmetry  displayed in  (\ref{tag: B}) $+$  (\ref{tag: vpi})   very surprising  and   we would like to comment a bit on this aspect.

Observe  that the union of the increasing family of subspaces $\bsW_{\nu,r}=  \bsV_{\nu^r}(d_1)\otimes \bsV_\nu(d_2)$  is dense  in the  Fr\`{e}chet topology of $C^\infty(M)$, $M=S^{d_1-1}\times S^{d_2-1}$.   The space $C^\infty(M)$   carries a natural stratification, where the various strata encode various types of degeneracies of the critical sets of functions on $M$.  The top strata are filled by   (stable) Morse function.  This stratification traces stratifications on each of the subspaces  $\bsW_{\nu,r}$ and,   as $\nu\ra \infty$,  the combinatorics of the  induced stratification on $\bsW_{\nu,r}$  captures more and more of the combinatorics  of the stratification of $C^\infty(M)$.

The equality (\ref{tag: B}) shows that if $r'>r\geq 1$, the functions in $\bsW_{\nu,r}$ have, on average, relatively more critical points than the functions in $\bsW_{\nu,r'}$. This suggest that the subspace  $(\bsW_{\nu,r})$    captures more  of the stratification of $C^\infty(M)$ than $\bsW_{\nu,r'}$, and in this sense it is a more efficient approximation. The best approximation  would be  when $r=1$, i.e., when  the two factors  $S^{d_i-1}$  participate    in the process as equal spectral partners.  Note that this asymmetric behavior is not present when one of the factors is $S^1$.

This heuristic discussion suggests the following  concepts. Suppose  that $M$ is a compact, connected  Riemann manifold of dimension $m$.  Define an \emph{approximation regime} on $M$ to be a  sequence   of finite dimensional subspaces $\bsW_\bullet=(\bsW_\nu)_{\nu\geq 1}$ of  $C^\infty(M)$  such that
\[
\bsW_1\subset \bsW_2\subset\cdots
\]
and their union is  dense in the Fr\`{e}chet topology of $C^\infty(M)$.    For any Riemann metric $g$ on $M$,  we  define the \emph{upper/lower complexities} of such a regime to be the quantities
\[
\kappa^*(\bsW_\bullet,g):=\limsup_{\nu\ra \infty}\frac{\log\mu(M,\bsW_\nu,g)}{\log \dim \bsW_\nu},\;\; \kappa_*(\bsW_\bullet,g):=\liminf_{\nu\ra \infty}\frac{\log\mu(M,\bsW_\nu, g)}{\log \dim \bsW_\nu}.
\]
Intuitively,  the approximation regimes with high  upper complexity offer better approximations of $C^\infty(M)$.  Finally, set
\[
\kappa^*(M):=\sup_{\bsW_\bullet, g}\kappa^*(\bsW_\bullet, g),\;\; \kappa_*(M):=\inf_{\bsW_\bullet, g}\kappa_*(\bsW_\bullet, g).
\]
The above results imply that
\[
\kappa^*(S^{d-1}),\;\;\kappa^*(S^{d_1-1}\times S^{d_2-1})\geq 1,\;\;\forall d_1,d_2\geq 2,
\]
\[
\kappa_*(S^{d_1-1}\times S^{d_2-1})\leq \frac{d_1-3}{d_2-1},\;\;\forall d_1, d_2\geq 3.
\]
In particular, this shows that for any $d\geq 3$, we have
\[
\kappa_*(S^2\times S^{d-1})=0.
\]
In Example \ref{ex: claim}  we\footnote{The construction of the approximation regime  in Example \ref{ex: claim}  was worked  out during a very lively conversation with my colleague Richard Hind who was  confident of its existence.}  construct an approximation regime $(\bsW_n)_{n\geq 1}$on $S^1$ such that
\[
\lim_{n\ra \infty} \frac{\log \mu(S^1,\bsW_n)}{\log \dim \bsW_n}=\infty,
\]
so that $\kappa^*(S^1)=\infty$.

\medskip

\noindent {\bf Acknowledgements.}  I would like to thank   Jesse Johnson for     his  careful proofreading of an earlier  version the manuscript.

\section*{Notations}

\begin{enumerate}

\item $\ii:=\sqrt{-1}$.

\item We will  denote by $\bsi_n$ the ``area'' of the  round  $n$-dimensional sphere $S^n$ of radius $1$, and by $\bom_n$  the  ``volume'' of the  unit ball in $\bR^n$.  These quantities are uniquely determined by the equalities (see \cite[Ex. 9.1.11]{N0})
\begin{equation*}
\bsi_{n-1}=n\bom_n=2\frac{\pi^{\frac{n}{2}}}{\Gamma(\frac{n}{2})},\;\;\Gamma\left(\frac{1}{2}\right)=\sqrt{\pi},
\tag{$\si$}
\label{tag: si}
\end{equation*}
where $\Gamma$  is Euler's Gamma function.

\item For any Euclidean space $\bsV$, we denote by $S(\bsV)$ the unit sphere in $\bsV$ centered at the origin and by $B(\bsV)$ the unit ball in $\bsV$ centered at the origin.

\item If $\bsV_0$ and $\bsV_1$ are two  Euclidean spaces of dimensions $n_0,n_1<\infty$ and $A:\bsV_0\ra \bsV_1$ is a linear map, then the \emph{Jacobian} of $A$ is the nonnegative scalar $J(A)$   defined as the norm of the linear map
\[
\Lambda^k A: \Lambda^k\bsV_0\ra \Lambda^k\bsV_1,\;\;k:=\min(n_0,n_1).
\]
More concretely, if  $n_0\leq n_1$,  and $\{\be_1,\dotsc,\be_{n_0}\}$ is an orthonormal basis of $\bsV_0$, then
\begin{equation*}
J(A)= \bigl(\,\det G(A)\,\bigr)^{1/2},
\tag{$J_-$}
\label{tag: j-}
\end{equation*}
where $G(A)$ is the $n_0\times n_0$ Gramm matrix with entries
\[
G_{ij}=\bigl(\, A\be_i, A\be_j\,\bigr)_{\bsV_1}.
\]
If $n_1\geq n_0$ then
\begin{equation*}
J(A)=J(A^\dag)=\bigl(\,\det G(A^\dag)\,\bigr)^{1/2},
\tag{$J_+$}
\label{tag: j+}
\end{equation*}
where $A^\dag$ denotes the adjoint (transpose) of $A$.   Equivalently,  if $d{\rm Vol}_i\in \Lambda^{n_i}\bsV_i^*$ denotes the  metric volume form  on $\bsV_1$, and $d{\rm Vol}_A$ denotes the metric volume  form on $\ker A$, then  $J(A)$ is the positive number such that
\begin{equation*}
d{\rm Vol}_0= \pm d{\rm Vol}_A\wedge  A^*d{\rm Vol}_1.
\tag{$J_+'$}
\label{tag: j+p}
\end{equation*}

\item For any nonnegative integer $d$, we denote by $[x]_d$ the degree $d$ polynomial
\[
[x]_d:= x(x-1)\cdots (x-d+1),
\]
and by $B_d(x)$ the degree $d$  Bernoulli polynomial    defined by the generating series
\[
\frac{te^{tx}}{e^t-1}=\sum_{d\geq 0} B_d(x)\frac{t^d}{d!}.
\]
The $d$-th Bernoulli number is $B_d:=B_d(0)$, while the leading coefficient of $B_d(x)$ is equal to $1$, and,
\[
\frac{B_{d+1}(\nu+1)-B_{d+1}}{d+1}=\sum_{n=1}^\nu n^d,\;\;\forall \nu\in\bZ_{>0}.
\tag{$S$}
\label{tag: ber}
\]
More generally,  for any smooth function $f: (0,\infty)\ra \bR$ and any  positive integers $\nu, m$, we have the \emph{Euler-Maclaurin summation formula},  (see \cite[Thm D.2.1]{AAR} or \cite[\S 7.21]{WW}),
\begin{align*}
&&\sum_{n=1}^{\nu-1} f(n)=\int_1^{\nu} f(x) dx+ \sum_{k=1}^m\frac{b_k}{k!}\Bigl(f^{(k-1)}(\nu)- f^{(k-1)}(1)\Bigr) \\
&& +\frac{(-1)^{m-1}}{m!}\int_1^\nu \bar{B}_m(x) f^{(m)}(x) dx,
\tag{$EM$}
\label{tag: S}
\end{align*}
where $b_k$ denotes the $k$-the Bernoulli number, $b_k:=B_k(0)$, and $\bar{B}_m$ denotes the associated periodic function
\[
\bar{B}_m(x) :=B_m(x-\lfloor x\rfloor),\;\;\forall x\in \bR.
\]
We will use one simple consequence of the Euler-Maclaurin summation formula.  Suppose that $f(x)$ is a rational  function of the form
\[
f(x)=\frac{P_0(x)}{P_1(x)},
\]
where $P_0(x)$ and $P_1(x)$ are polynomials with leading coefficients $1$ and of degrees $d_0>d_1$. We further assume that $f$ has no poles at nonnegative integers. Then
\begin{equation*}
\sum_{n=1}^\nu f(n)\sim\frac{1}{d_0-d_1+1}\nu^{d_0-d_1+1}\;\;\mbox{as}\;\;\nu\ra \infty.
\tag{$S_\infty$}
\label{tag: sasy}
\end{equation*}

\end{enumerate}

\section{An   abstract result}
\label{s: 1}
\setcounter{equation}{0}

Suppose that $(M, g)$  is compact, connected Riemann manifold  of dimension $m$.  We denote by $|dV_g|$ the induced volume density.

Let $\bsV\subset C^\infty(M)$ be a   vector subspace of finite dimension $N$.  We set $\bsV\dual:=\Hom(\bsV,\bR)$, and we fix a Euclidean metric   $h=(-,-)$ on $\bsV$. We denote by $S(\bsV)$ the unit sphere in $\bsV$ with respect to this metric and by $|dS|$ the area   density on $S(\bsV)$.   The goal of this section is to give an integral geometric      description  of the quantity
\[
\mu(M,\bsV)=\mu(M,g,\bsV,h):=\frac{1}{{\rm area}\,(S(\bsV)\,)}\int_{S(\bsV)} \mu_M(\bv)\, |dS(\bv)|.
\]
The significance of $\mu(M,\bsV)$ is clear: it is the expected number of critical points of a random function $\bv\in S(\bsV)$.

To formulate our main result we need to introduce some notation.  We  form the trivial vector bundle $\underline{\bsV}_M:=\bsV\times M$. Observe that the dual bundle $\underline{\bsV}_M^\dual=\bsV\dual\times M$ is equipped with  a canonical  section
\[
\ev: M\ra \bsV\dual,\;\;M\ni \bx\mapsto \ev_\bx\in\bsV\dual,\;\;\ev_{\bx}(\bv)=\bv(x),\;\;\forall \bv\in \bsV.
\]
Using the metric identification $\bsV\dual\ra \bsV$ we can regard $\ev$ as a map $M\ra \bsV$. More explicitly,  if $(\Psi_\alpha)_{1\leq \alpha\leq N}$ is an orthonormal basis of $\bsV$, then
\[
\ev_\bx=\sum_\alpha \Psi_\alpha(\bx)\cdot \Psi_\alpha\in\bsV.
\]
We have an  adjunction morphism
\[
\eA:\bsV\times M\ra T^*M,\;\; \bsV\times M\ni  (\bv, \bx)\mapsto \eA_\bx\bv:=d_{\bx} \bv\in T^*_\bx M,
\]
where $d_\bx$ denotes the  differential of the function $\bv$ at the point $\bx\in M$.    We  will assume that the  vector space $\bsV$ satisfies the ampleness condition
\begin{equation}
\forall \bx\in M\;\;\mbox{the linear map $\bsV\ni\bv\stackrel{\eA_\bx}{\longmapsto} d_\bx \bv\in T^*_\bx M$ is surjective}.
\label{eq: nondeg}
\end{equation}
The  assumption (\ref{eq: nondeg})  is equivalent   to  the condition:
\begin{equation}
\mbox{the  evaluation map $\ev:M\ra \bsV\dual$ is an immersion.}
\label{eq: ev-imm}
\end{equation}
As explained in \cite[\S 1.2]{N1}, the  condition (\ref{eq: nondeg}) implies that for generic $\bv\in\bsV$, the restriction of the function $\bv$ to   $K$ is a Morse function. We denote  by $\mu_M(\bv)$  its number of critical points.

For every $\bx\in M$, we denote by $\bsK_\bx$ the kernel of the  map $\eA_\bx$. The ampleness condition  (\ref{eq: nondeg}) implies that $\bsK_\bx$ is a subspace of $\bsV$ of codimension $m$.  Observe that the collection of spaces $(\bsK_\bx)_{\bx}$ is naturally organized as a codimension $m$-subbundle $\bsK\ra M$ of $\underline{\bsV}_M$, namely the kernel bundle of $\eA$.

Consider the dual bundle  morphism $\eA^\dag: TM\ra \bsV\dual\times M$. Using the metric identification $\bsV\dual\ra \bsV$ we can regard $\eA^\dag$ as a bundle morphism $\eA^\dag: TM\ra \underline{\bsV}_M$. Its range is $\bsK^\perp$, the orthogonal complement to the kernel of $\eA$. Note that if $\{\Psi_\alpha\}_{1\leq \alpha\leq N}$ is an orthonormal  frame of $\bsV$, $\bx_0\in M$, and $X\in T_{\bx _0}M$, then
\[
\eA^\dag_{\bx_0} X=\sum_{\alpha=1}^N \bigl(\,(X\cdot \Psi_\alpha)(\bx_0)\,\bigr)\cdot \Psi_\alpha\in \bsV.
\]
The    trivial bundle $\underline{\bsV}_M$ is equipped with a trivial connection  $D$.   More precisely, we regard a section of $\bu$ of  $\underline{\bsV}_M$ as a smooth map $\bu: M\ra \bsV$. Then, for any vector field $X$ on $M$,    we  define $D_X\bu$ as the smooth function $M\ra \bsV$ obtained by derivating $\bu$ along $X$. Note  that $\eA^\dag=D\ev$.

We have  an orthogonal direct sum decomposition $\underline{\bsV}_M= \bsK^\perp\times \bsK$. For any section $\bu$ of $\underline{\bsV}_M$, we denote by $\bu^\perp$ the component of $\bu$  along $\bsK^\perp$, and by $\bu^0$ its component along  $\bsK$.       The \emph{shape operator} of the  subbundle $\bsK^\perp$ is the bundle morphism $\bXi: TM\otimes \bsK^\perp\ra \bsK$ defined by the equality
\[
\bXi(X, \bu):= (D_Xu)^0,\;\;\forall X\in C^\infty(TM),\;\;\bu\in C^\infty(\bsK^\perp).
\]
For every $\bx\in M$, we denote by $\bXi_\bx$ the induced linear map $\bXi_\bx: T_\bx M\otimes \bsK_\bx^\perp \ra \bsK_\bx$. If we denote by $\Gr_m(\bsV)$ the  Grassmannian of $m$-dimensional subspaces  of $\bsV$, then we have a  Gauss  map
\[
M\ni\bx \stackrel{\eG}{\longmapsto} \eG(\bx):=\bsK_\bx^\perp\in \Gr_m(\bsV).
\]
For $\bx\in M$, the  shape operator $\bXi_\bx$ can be viewed as a linear map
\[
\bXi_\bx : T_\bx M\ra \Hom(\bsK_\bx^\perp,\bsK_\bx)= T_{\bsK_\bx^\perp}\Gr_m(\bsV),
\]
  and, as such, it  can be identified with the differential of $\eG$ at $\bx$, \cite[\S 9.1.2]{N0}. Any  $\bv\in\bsK_\bx$ determines  a bilinear map
 \[
 \bXi_\bx\cdot\bv:  T_\bx M\otimes \bsK_\bx^\perp \ra\bR,\;\;\bXi_\bx\cdot\bv(\be, \bu)= \bXi_\bx(\be,\bu)\cdot\bv,
 \]
 where, for simplicity, we have denoted by $\cdot$ the inner product in $\bsV$.  By choosing   orthonormal bases  $(\be_i)$ in $T_\bx M$ and $(\bu_j)$ of $\bsK_\bx$ we can identify this bilinear form with an  $m\times m$-matrix.   This matrix depends on the choices of bases, but   the absolute value of its determinant is independent of these bases.   It is thus an invariant of the pair $(\bXi_\bx, \bv)$ that we will  denote by $|\det\bXi_\bx\cdot \bv|$.

 \begin{theorem}
\begin{equation}
\mu(M,\bsV)= \frac{1}{\bsi_{N-1}}\int_M{\left(\,\int_{S(\bsK_\bx)}|\det\bXi_\bx\cdot\bv|\, |dS(\bv)|\,\right)}\, |dV_g(\bx)|.
 \label{eq: av5}
\end{equation}

\label{th: av}
\end{theorem}

\begin{proof} We denote by $E_\bx$ the intersection of $\bsK_\bx$ with the sphere $S(\bsV)$ so that $E_x$ is a geodesic sphere in $S(\bsV)$ of dimension $(N-m-1)$. Now consider the incidence set
\[
E_M :=\bigl\{ (\bx,\bv)\in M\times S(\bsV);\;\; \eA_\bx\bv=0\,\bigr\}=\bigl\{ (\bx,\bv)\in M\times S(\bsV);\;\; \bv\in E_\bx\,\bigr\}.
\]
We have  natural (left/right) smooth projections
\[
M\stackrel{\lambda}{\longleftarrow} E_M\stackrel{\rho}{\longrightarrow} S(\bsV).
\]
The left projection $\lambda:E_M\ra M$ describes $E_M$ as the unit sphere bundle associated to the  metric vector bundle $\bsK_M$. In particular, this shows that $E_M$ is a compact, smooth manifold of dimension $(N-1)$.  For generic $\bv\in S(\bsV)$ the  fiber $\rho^{-1}(\bv) $ is  finite  and can be  identified with  the set of critical points of $\bv:M\ra \bR$. We deduce
\begin{equation}
\mu(M,\bsV)=\frac{1}{{\rm area}\,(\,S(\bsV)\,)}\int_{S(\bsV)} \#\rho^{-1}(\bv)\, |dS(\bv)|.
\label{eq: av1}
\end{equation}
Denote by $g_E$ the metric on $E_M$ induced by the metric on $M\times S(\bsV)$ and by $|dV_E|$ the induced volume density.    The area   formula  (see \cite[\S 3.2]{Feder} or \cite[\S 5.1]{KP}) implies  that
\begin{equation}
\int_{S(\bsV)} \# \rho^{-1}(\bw) |dS(\bv)|=\int_E J_\rho(\bx,\bv) |dV_E(\bx,\bv)|,
\label{eq: av}
\end{equation}
where   the nonnegative  function $J_\rho$ is the Jacobian of $\rho$ defined by the equality
\[
\rho^*|dS|=J_\rho \cdot |dV_E|.
\]
To compute the integral in the right-hand side of (\ref{eq: av}) we need a  more explicit description of the geometry of $E_M$.

Fix a local orthonormal frame $(\be_1,\dotsc, \be_m)$ of $TM$  defined  in a neighborhood $\eN$ in $M$ of a given point $\bx_0\in M$. We denote by $(\be^1,\dotsc, \be^m)$ the dual co-frame of $T^*M$.  Set
\[
\bsf_i(\bx):=\eA_\bx^\dag\be_i(\bx)\in\bsV,\;\;i=1,\dotsc, m,\;\;\bx\in\eN.
\]
More explicitly, $\bsf_i(\bx)$ is defined by the equality
\begin{equation}
\bigl(\, \bsf_i(\bx),\bv\,\bigr)_{\bsV}=\pa_{\be_i}\bv(\bx) ,\;\;\forall\bv\in\bsV.
\label{eq: fi}
\end{equation}
Fix a neighborhood $\eU\subset \lambda^{-1}(\eN)$ in $M\times S(\bsV)$  of the point  $(\bx_0,\bv_0)$,      and  a   local orthonormal frame $\bu_1 (\bx,\bv),\dotsc, \bu_{N-1}(\bx,\bv)$ over $\eU$ of the bundle $\rho^*TS(\bsV)\ra M\times S(\bsV)$ such that the following hold.

\begin{itemize}

\item The vectors $\bu_1(\bx,\bv),\dotsc,\bu_m(\bx,\bv)$ are independent  of the variable $\bv$ and form an orthonormal basis of $K_\bx^\perp$. (E.g.,   we can obtain such vectors  from the vectors $\bsf_1(\bx),\dotsc, \bsf_m(\bx)$ via the  Gramm-Schmidt  process.)

\item  For $(\bx,\bv)\in \eU$, the space   $T_{\bv}E_{\bx}$ is spanned by  the vectors $\bu_{m+1}(\bx, \bv),\dotsc, \bu_{N-1}(\bx,\bv)$.

\end{itemize}

The   collection $\bu_1(\bx),\dotsc,\bu_m(\bx)$ is a collection of smooth sections of $\underline{\bsV}_M$ over $\eN$.  For any $\bx\in \eN$ and any $\be\in T_\bx M$,  we obtain    the vectors (functions).
\[
D_{\be}\bu_1(\bx),\dotsc, D_{\be}\bu_m(\bx)\in \bsV.
\]
Observe that
\begin{equation}
E_M\cap\eU=\bigl\{ (\bx,\bv)\in \eU;\;\; U_i(\bx,\bv)=0,\;\;\forall i=1,\dotsc, m\,\bigr\},
\label{eq: inc-eq}
\end{equation}
where $U_i$ is the function $U_i:\eN\times\bsV\ra \bR$ given by
\[
U_i(\bx,\bv):=\bigl(\, \bu_i(\bx),\bv\,\bigr)_{\bsV}.
\]
Thus, the tangent  space  of  $E_M$ at $(\bx,\bv)$  consists of tangent vectors $\dot{\bx}\oplus \dot{\bv}\in  T_\bx M\oplus T_{\bv} S(\bsV)$ such that
\[
dU_i(\dot{\bx},\dot{\bv})=0,\;\;\forall i=1,\dotsc, m.
\]
We let $\omega_U$ denote the $m$-form
\[
\omega_U:= dU_1\wedge \cdots \wedge dU_m\in \Omega^m(\eU),
\]
and we denote  by  $\|\omega_U\|$   its norm  with respect to the product metric on $M\times S(\bsV)$.  Denote by $|\widehat{dV}|$ the volume density on $M\times S(\bsV)$ induced by the product metric. The equality (\ref{eq: inc-eq}) implies that
\[
|\widehat{dV}|=\frac{1}{\|\omega_U\|} \left|\omega_U\wedge dV_E\,\right|.
\]
Hence
\[
J_\rho|\widehat{dV}|= \frac{1}{\|\omega_U\|}|\omega_U\wedge \rho^* dS|.
\]
 We deduce
\[
J_\rho(\bx_0,\bv_0)=J_\rho(\bx_0,\bv_0)|\widehat{dV}|(\be_1,\dotsc,\be_m, \bu_1, \dotsc,\bu_{N-1})
\]
\[
= \frac{1}{\|\omega_U\|}| \omega_U\wedge \rho^* dS| (\be_1,\dotsc,\be_m, \bu_1, \dotsc,\bu_{N-1})=\frac{1}{\|\omega_U\|}\underbrace{\left| \omega_U\bigl(\, \be_1,\dotsc,\be_m\,\bigr)\right|_{(\bx_0,\bv_0)}}_{=:\Delta_U(\bx_0,\bv_0)}.
\]
Hence,
\begin{equation}
\int_{S(\bsV)} \# \rho^{-1}(\bw) |dS(\bv)|=\int_E \frac{\Delta_U}{\|\omega_U\|}\, |dV_E(\bx,\bv)|.
\label{eq: av3}
\end{equation}
\begin{lemma}  We have the equality $J_\lambda=\frac{1}{\|\omega_U\|}$, where $J_\lambda$ denotes the  Jacobian of the projection $\lambda: E_M\ra M$.
\label{lemma: prod-dens}
\end{lemma}

\begin{proof}  Along $\eU$  we have\[
|\widehat{dV}|= \frac{1}{\|\omega_U\|} \left|\omega_U\wedge dV_E\,\right|
\]
while (\ref{tag: j+p}) implies that
\[
|dV_E| = \frac{1}{J_\lambda}|dV_g\wedge dS_{E_\bx}|.
\]
 Therefore, suffices to show that   along $\eU$ we have
\[
|\widehat{dV}|= |\omega_U\wedge dV_g\wedge dS_{E_\bx}|, 
\]
i.e.,
\[
\left|\omega_U\wedge  dV_g\wedge dS_{E_\bx}(\be_1,\dotsc, \be_m,\bu_1,\dotsc,\bu_{N-1})\,\right|=1.
\]
Since $dU_i(\bu_k)=0$, $\forall k\geq m+1$ we deduce  that
\[
\left|\omega_U\wedge  dV_g\wedge dS_{E_\bx} (\be_1,\dotsc, \be_m,\bu_1,\dotsc,\bu_{N-1})\,\right|=|\omega_U(\bu_1,\dotsc,\bu_m)|.
\]
Thus, it suffices to show that
\[
|\omega_U(\bu_1,\dotsc,\bu_m)|=1.
\]
This follows from the elementary identities
\[
dU_i(\bu_j)=(\bu_i,\bu_j)_{\bsV}=\delta_{ij},\;\;\forall 1\leq i,j\leq m,
\]
where  $\delta_{ij}$ is the Kronecker  symbol.\end{proof}

Using Lemma \ref{lemma: prod-dens} in (\ref{eq: av3})  and the co-area formula we deduce
\begin{equation}
\int_{S(\bsV)} \# \rho^{-1}(\bw) |dS(\bv)|=\int_M\underbrace{\left(\,\int_{E_\bx}\Delta_U(\bx,\bv)\, |dS_{E_\bx}(\bv)|\,\right)}_{=:J(\bx)}\, |dV_g(\bx)|.
\label{eq: av4}
\end{equation}
Observe   that at a point $(\bx,\bv)\in \lambda^{-1}(\eN)\subset E_M$ we have
\[
dU_i(\be_j)=  \bigl(\, D_{\be_j}\bu_i(x),\bv\,\bigr)_{\bsV}.
\]
We can rewrite this in terms of the shape operator $\bXi_\bx: T_\bx M\otimes \bsK_\bx^\perp\ra \bsK_\bx$.   More precisely,
\[
dU_i(\be_j)= (\bXi_\bx(\be_j,\bu_i),\bv)_{\bsV}.
\]
Hence,
\[
\Delta_U(\bx, \bv)=\left| \det \bXi_\bx\cdot\bv\,\right|,
\]
We conclude that
\[
\int_{S(\bsV)} \# \rho^{-1}(\bv) |dS(\bv)|=\int_M{\left(\,\int_{E_\bx}|\det\bXi_\bx\cdot\bv|\, |dS_{E_\bx}(\bv)|\,\right)}\, |dV_M(\bx)|.
\]
This proves (\ref{eq: av5})
\end{proof}

The story is not yet over. We want to  rewrite the right-hand side of (\ref{eq: av5}) in a more computationally friendly form,  preferably in terms of differential-integral invariants of the evaluation map.   The starting point is the observation that the left-hand side of  (\ref{eq: av5})  is   plainly independent of   the metric $g$ on $M$.     This raises the hope that if we judiciously  choose  the metric on $M$ we can  obtain a more manageable   expression for $\mu(M,\bsV)$. One choice presents itself.  Namely, we choose the  metric $\bar{g}$ on $M$  uniquely determined by  requiring  that the bundle morphism
\[
\eA^\dag:(TM, \bar{g})\ra \bsV\times M
\]
is an \emph{isometric embedding}. Equivalently, $\bar{g}$ is  the pullback to $M$  of the metric on $\bsV$ via the immersion $\ev:M\ra \bsV\dual\cong\bsV$.  More concretely,  for any $\bx\in M$ and any $X,Y\in T_\bx M$, we have
\[
\bar{g}_\bx(X,Y)=\bigl(\,\eA_\bx^\dag X, \eA_\bx^\dag Y\,\bigr)_{\bsV}.
\]
With this choice of metric, Theorem \ref{th: av}  is precisely the  main theorem of Chern and Lashof, \cite{CL}.

Fix  $\bx\in M$ and a $\bar{g}$-orthonormal  frame$ (\be_i)_{1\leq i\leq m}$ of $TM$ defined in a neighborhood $\eN$ of $\bx$.    Then  the collection $\bu_j=\eA^\dag \be_j$, $1\leq j$, is a local orthonormal frame of $\bsK^\perp$ on $\eN$.  The  shape operator has the simple description
\[
\bXi_\bx(\be_i,\bu_j)= \bigl(\,D_{\be_i}\eA^\dag\be_j\,\bigr)^0.
\]
Fix an orthonormal basis $(\Psi_\alpha)_{1\leq\alpha\leq N}$ of $\bsV$ so that every $\bv\in \bsV$ has a decomposition
\[
\bv=\sum_\alpha v_\alpha \Psi_\alpha,\;\;v_\alpha\in\bR.
\]
Then, for any $\by\in \eN$, we have
\[
\eA^\dag\be_j(\by)=\sum_\alpha (\pa_{\be_j} \Psi_\alpha)_\by\Psi_\alpha,\;\; D_{\be_i}\eA^\dag\be_j(\by)=\sum_\alpha (\pa^2_{\be_i \be_j}\Psi_\alpha)_\by \Psi_\alpha,
\]
and
\[
\bigl(\, (D_{\be_i}\eA^\dag\be_j)_\by,\bv\,\bigr)_{\bsV}= \sum_\alpha v_\alpha (\pa^2_{\be_i \be_j}\Psi_\alpha)_\by  = \pa^2_{\be_i\be_j}\bv(\by).
\]
If $\bv\in \bsK_\bx$, then the Hessian of $\bv$ at $\bx$ is a well-defined,    symmetric bilinear form $\Hess_\bx(\bv): T_\bx M\times  T_\bx M\ra \bR$, i.e., an element of $T^*_\bx M\otimes  T^*_\bx M$. Using the metric  $\bar{g}$ we can identify it with a linear operator
\[
\Hess_\bx(\bv,\bar{g}) : T_\bx M\ra T_\bx M.
\]
If we fix a $\bar{g}$-orthonormal frame $(\be_i)$ of $T_\bx M$, then  the operator $\Hess_\bx(\bv,\bar{g})$  is described by the symmetric $m\times m$ matrix with entries $\pa^2_{\be_i\be_j}\bv(\bx)$. We deduce that
\[
\left|\det\bXi_\bx \cdot \bv\right|=\left| \det \Hess_\bx(\bv,{\bar{g}})\,\right|,\;\;\forall\bv\in E_\bx.
\]
In particular, we deduce that
\begin{equation}
\mu(M,\bsV)=\frac{1}{\bsi_{N-1}}\int_M\,\left(\,\int_{E_\bx} |\det \Hess_\bx(\bv,{\bar{g}})|\,|dS_\bx(\bv)\,\right)|dV_{\bar{g}}(\bx)|.
\label{eq: av6}
\end{equation}
Finally, we  want to express (\ref{eq: av6}) entirely  in terms of  the  adjunction map $\eA$.   For any $\bx\in M$ and  any  $\bv\in \bsK_\bx$,  we define    the   density
\[
\rho_{\bx,\bv}:  \Lambda^m T_\bx M\ra \bR,
\]
\[
 \rho_{\bx,\bv}(X_1\wedge \cdots \wedge X_m)=    \left| \det\bigl( \, \pa^2_{X_iX_j}\bv (\bx)\,\bigr)_{1\leq i,j\leq m}\,\right|\,\cdot\,\bigl(\,\det \bigl(\,(\eA^\dag X_i,\eA^\dag X_j)_{\bsV}\,\bigr)_{1\leq i,j\leq m}\,\bigr)^{-1/2}
\]
\[
=\left|\,\det\bigl(\,\Hess_\bx(\bv)(X_i,X_j)\,\bigr)_{1\leq i,j\leq m}\,\right| \cdot \bigl(\,\det\bigl(\, \bar{g}(X_i,X_j)\,\bigr)_{1\leq i,j\leq m}\,\bigr)^{-1/2}.
\]
Observe that for any $\bar{g}$-orthonormal frame of $T_\bx M$ we have
\[
\rho_{\bx,\bv}(\be_1\wedge\cdots\wedge\be_m)= |\det \Hess_\bx(\bv, {\bar{g}})\,|.
\]
If we integrate $\rho_{\bx,\bv}$ over $\bv\in S(\bsK_\bx)$, we obtain a density
\[
|d\mu_{\bsV}(\bx)|:  \Lambda^m T_\bx M\ra \bR,
\]
\[
 |d\mu_{\bsV}(\bx)|(X_1\wedge\cdots \wedge X_m)=\int_{S(\bsK_\bx)}  \rho_{\bx,\bv}(X_1\wedge \cdots \wedge X_m)\,|dS_h(\bv)|,\;\;\forall X_1,\dotsc, X_m\in T_\bx M.
\]
Clearly  $|d\mu(\bx,\bsV)|$ varies smoothly with $\bx$, and thus it defines   a density $|d\mu_{\bsV}(-)|$ on $M$.  We want to emphasize that   this density depends on the metric  on $\bsV$ but \emph{it is independent} on any metric on $M$. We will refer to it as \emph{the   density of $\bsV$}.

If we  fix a different metric $g$ on $M$, then we can express    $|d\mu_{\bsV}(-)|$  as a product 
\[
|d\mu_{\bsV}(\bx)|=\rho_g(\bx)\cdot |dV_g(\bx)|,
\]
where $\rho_g=\rho_{g,\bsV}:M\ra \bR$ is a smooth nonnegative function. 

To find a  more useful description  of $\rho_g$, we choose  local coordinates $(x^1,\dotsc, x^m)$  near $\bx$ such that   $(\pa_{x^i})$ is a $g$-\emph{orthonormal} basis of $T_{\bx} M$. Then
\[
\rho_{\bx,\bv}(\pa_{x_1}\wedge \cdots\wedge\pa_{x_m}) =   \left| \det\bigl( \, \pa^2_{x_ix_j}\bv (\bx)\,\bigr)_{1\leq i,j\leq m}\,\right|\,\cdot\,\bigl(\,\det \bigl(\,(\eA^\dag \pa_{x_i},\eA^\dag \pa_{x_j})_{\bsV}\,\bigr)_{1\leq i,j\leq m}\,\bigr)^{-1/2}.
\]
Observe  that the matrix $( \, \pa^2_{x_ix_j}\bv (\bx)\,\bigr)_{1\leq i,j\leq m}$   describes the Hessian operator
\[
\Hess_\bx(\bv,g): T_\bx M\ra T_\bx M
\]
induced by the Hessian of $\bv$ at $\bx$ and the metric $g$.

The scalar  $\bigl(\,\det \bigl(\,(\eA^\dag \pa_{x_i},\eA^\dag \pa_{x_j})_{\bsV}\,\bigr)_{1\leq i,j\leq m}\,\bigr)^{1/2}$ is precisely the Jacobian of the  dual adjunction map $\eA^\dag_\bx: T_\bx M\ra \bsV$ defined in terms of the metric $g$ on $T_\bx M$ and the metric on $\bsV$. We denote it by  $J(\eA^\dag_\bx,g)$. We set
\[
\Delta_\bx(\bsV,g):=  \int_{S(\bsK_\bx)} |\det \Hess_\bx(\bv, {{g}})|\,|dS_\bx(\bv)|.
\]
Since
\[
|dV_g(\bx)|(\pa_{x_1}\wedge\cdots \wedge \pa_{x_m})=1,
\]
we deduce
\begin{equation}
\rho_{g,\bsV}(\bx)= \Delta_\bx(\bsV,g)\cdot J(\eA^\dag_\bx,g)^{-1}.
\label{eq: rho-mu}
\end{equation}
We have thus proved the following result.
\begin{corollary} Suppose $(M,g)$ is a   compact, connected Riemann manifold and $\bsV\subset C^\infty(M)$ is a    vector subspace of dimension $N$. Fix an Euclidean inner product $h$ on $\bsV$ with norm $|-|_h$. Then
\begin{equation}
\begin{split}
\mu(M,g,\bsV,h)&=\frac{1}{{\rm area}\,(S(\bsV)\,)}\int_{|\bv|_h=1}\# \{d\bv=0\} \,|dS_h(\bv)|\\
&=\frac{1}{\bsi_{N-1}}\int_M \frac{\Delta_\bx(\bsV,g)}{J(\eA^\dag_\bx,g, h)}\, \,|dV_g(\bx)|,
\end{split}
\label{eq: av7}
\end{equation}
where $|dS_h|$ denotes the   area density on the  unit sphere $\{|\bv|_h=1\}$, and  $J(\eA^\dag_\bx,g, h)$ denotes the Jacobian of the dual  adjunction map $\eA^\dag_\bx: T_\bx M\ra \bsV$ computed in terms of the metrics $g$ on $T_\bx M$ and $h$ on $\bsV$.\qed
\label{cor: av}
\end{corollary}

We will refer to the quantity $\mu(M,g,\bsV,h)$ as the \emph{expectation} of the quadruple $(M,g,\bsV,h)$.

\begin{remark}  Let us observe  that  we have proved a little bit more. To every    Morse function $\bv\in \bsV$ we associate  the measure
\[
\mu_\bv=\sum_{d\bv(\bx)=0} \delta_\bx,
\]
where $\delta_\bx$ Denotes the Dirac measure concentrated at $\bx$.   For every continuous function $f:M\ra \bR$ we set
\[
\mu_\bv(f):=\int_M f d\mu_\bv=\sum_{d\bv(x)=0} f(\bx),
\]
and we denote by  $\bsE(\mu_\bv(f)\,)$ the expection of the random   variable $S(\bsV)\ni \bv\mapsto \mu_\bv(f)$,
\[
\bsE(\,\mu_\bv(f)\,):=\frac{1}{{\rm area}\,(S(\bsV)}\int_{S(\bsV)} \mu_\bv)(f)\,|dS(\bv)|.
\]
Arguing exactly as in the proof of Corollary \ref{cor: av} we deduce that, for any  Riemann metric $g$ on $M$ we have
\begin{equation}
\bsE(\,\mu_\bv(f)\,)=\frac{1}{\bsi_{N-1}}\int_M \frac{\Delta_\bx(\bsV,g)}{J(\eA^\dag_\bx,g, h)}\, \,f(x) |dV_g(\bx)|.
\label{eq: av8}
\end{equation}
The resulting density on $M$
\[
\frac{1}{\pi^{\frac{N-m}{2}}J(\eA^\dag_\bx,g, h)}\left(\int_{\bsK_\bx} e^{-|\bu|^2} |\det \Hess_x(\bu)|\,|dV_{\bsK_\bx}(\bu)\right) |dV_g(x)|
\]
is called the  expected  density of critical points of a function in $\bsV$.  As explained in the introduction, if $\bsV$ is $2$-jet ample, then  the above Gaussian  integral  over $\bsK_\bx$ can be   reduced to a Gaussian  integral over $\Sym(T_\bx M)$.     In this case,  the resulting formula is a   special case of \cite[Thm.4.2]{ BSZ2}  that was obtained by a different approach, more probabilistic in nature.  \qed
\label{rem: comp}
\end{remark}

\begin{remark}[\textbf{\textit{A Gauss-Bonnet type formula}}]  With a little care, the above  arguments lead to a   Gauss-Bonnet type theorem.   More precisely, if we assume that $M$ is oriented, then,  under appropriate  orientation conventions,  the Morse inequalities imply that the degree of the map $\rho: E_M\ra S(\bsV)$ is equal to the Euler characteristic of $M$. If  instead of working with densities, we work with forms,   then we conclude that
\[
\chi(M)=\frac{1}{\bsi_{N-1}}\int_M \frac{\chi_\bx(\bsV,g)}{J(\eA^\dag_\bx,g, h)}\, \,dV_g(\bx),
\]
where
\[
\chi_\bx(\bsV, g):= \int_{S(\bsK_\bx)} \det \Hess_\bx(\bv, {{g}})\,dS_\bx(\bv).
\]
When $M$ is a submanifold  of the Euclidean space $\bsV$, and we identify  $\bsV$ with $\bsV\dual\subset C^\infty(\bsV)$, then the above argument yields the   Gauss-Bonnet theorem for submanifolds of a Euclidean space. \qed
\label{rem: GB}
\end{remark}

We say that a quadruple $(M,g,\bsV,h)$ as in Corollary \ref{cor: av} is \emph{homogeneous}  with respect to a compact Lie group $G$ if  the following hold.

\begin{itemize}

\item The group $G$ acts   transitively  and isometrically on $M$.

\item For any function $\bv\in \bsV$, and any $g\in G$, the pullback $g^*\bv$ is also  a function in $\bsV$.

\item The  metric $h$ is invariant with respect to the induced right action of $G$ on $\bsV$ by pullback.

\end{itemize}

For homogeneous quadruples   formula (\ref{eq:  av7})   simplifies considerably  because  in this case the function $\rho_{g,\bsV}$ is constant.  We deduce that in this case we have
\begin{equation}
\mu(M,g,\bsV,h)= \frac{\Delta_{\bx_0}(\bsV, {{g}})}{\bsi_{N-1}J(\eA^\dag_{\bx_0},g, h)}\cdot {\rm vol}_g(M),
\label{eq: av-hom}
\end{equation}
where $\bx_0$ is  an arbitrary point in $M$.

Let us  observe that to any triple $(M,g,\bsV)$, $\bsV\subset C^\infty(M)$, we can associate in  a canonical fashion a  quadruple  $(M,g,\bsV, h_g)$, where $h_g$ is the inner product on $\bsV$ induced by the  $L^2(M,|dV_g|)$ inner product on $C^\infty(M)$.  The expectation of such a triple  is, by definition, the expectation of the associated quadruple.  We will denote it by $\mu(M,g,\bsV)$. We  say that a triple $(M,g,\bsV)$ is \emph{homogeneous} if the associated quadruple is so.

\section{Random polynomials on spheres}
\setcounter{equation}{0}

  As is well known,  the spectrum of the Laplacian on the unit sphere $(S^{d-1}, g_d)\subset \bR^d$ is
\[
\bigl\{ \lambda_n(d)= n(n+d-2);\;\; n\geq 0\}.
\]
 We denote  by $\eY_{n,d}$ the eigenspace corresponding to  the eigenvalue $\lambda(d)$. As indicated in Appendix \ref{s: b}, the space  $\eY_{n,d}$ has dimension
 \[
 M(n,d)=\frac{2n+d-2}{n+d-2}\binom{n+d-2}{d-2}\sim 2\frac{n^{d-2}}{(d-2)!}\;\;\mbox{as}\;\;n\ra \infty,
 \]
 and  can be explicitly described as the  space of restrictions to $S^{d-1}$ of harmonic homogeneous polynomials of degree $n$ on $\bR^d$.    For any positive integer $\nu$, we set
 \[
 \bsV_\nu=\bsV_\nu(d):=\bigoplus_{n=0}^\nu\eY_{n,d}.
 \]
The space $\bsV_\nu(d)$ can be identified with the space of restrictions to $S^{d-1}$ of polynomials  of degree $\leq \nu$ in $d$ variables. Note that
 \begin{equation}
 N_\nu:=\dim\bsV_\nu(d)\sim  \frac{2\nu^{d-1}}{(d-1)!}\;\;\mbox{as}\;\;\nu\ra \infty.
 \label{eq: asy-dim}
 \end{equation}
 The  resulting triple $(S^{d-1}, g_d, \bsV_\nu)$ is homogeneous, and we denote by  $\mu(S^{d-1},\bsV_\nu)$ its expectation.
 The goal  of this section is  to  describe the asymptotics   of  $\mu(S^{d-1},\bsV_\nu(d))$  as $\nu \ra \infty$ in the case $d\geq 3$. The simpler case $d=2$  will be analyzed separately in Corollary \ref{cor: av-circle}.

 \begin{theorem} For any $d\geq 3$ there exists a positive constant $K=K_d$ that depends only on $d$ such that
 \begin{equation}
\mu(S^{d-1},\bsV_\nu(d))\sim K_d\dim\bsV_\nu(d)\;\;\mbox{as}\;\;\nu\ra \infty.
 \label{eq: av-sph-harm}
 \end{equation}
 In particular,
 \[
 \log\mu(S^{d-1},\bsV_\nu(d))\sim \log \dim \bsV_\nu(d)\;\;\mbox{as}\;\;\nu\ra \infty.
 \]
 \label{th: sph-harm}
 \end{theorem}

 \begin{proof}   For simplicity, we will write $\bsV_\nu$ instead of $\bsV_\nu(d)$. We will  rely on some  classical facts about spherical harmonics  surveyed in Appendix \ref{s: b}.    For any integer $d\geq 2$, we denote by $\eB_{n,d}$ the canonical orthonormal basis  of $\eY_{n,d}$ constructed by the inductive process outlined in Appendix \ref{s: b} and described in more detail below.

 According to  Corollary \ref{cor: av},  it suffices  to describe the density of $\bsV_\nu$ at the  North Pole $p_0=(0,0,\dotsc, 0,1)\in S^{d-1}$.   Denote by $\bsK_\nu(p_0)$ the subspace of $\bsV_\nu$ consisting of functions    for which $p_0$ is a critical point. Note that
 \[
 \dim \bsK_\nu(p_0)=\dim\bsV_\nu- \dim S^{d-1}= N_\nu- (d-1).
 \]
 Near $p_0$ we  use $\bx'=(x_1,\dotsc,x_{d-1})$ as local coordinates so that
 \begin{equation}
 x_d=\sqrt{(1-|\bx'|^2}=1-\frac{1}{2}|\bx'|^2+\mbox{higher order terms}.
 \label{eq: tayxd}
 \end{equation}
Note that, at $p_0$, the tangent vectors $\pa_{x_1},\dotsc,\pa_{x_{d-1}}$ form an \emph{orthonormal} frame of $T_{p_0}S^{d-1}$.

For any  function $f\in C^\infty(S^{d-1})$, we denote by $\Hess(f)$ the Hessian of  $f$ at $p_0$, i.e., the $(d-1)\times (d-1)$ symmetric matrix  with entries
\[
\Hess(f)_{ij}=\pa^2_{x_i x_j}f(p_0),\;\;1\leq i, j\leq d-1.
\]
We set
\[
B_{j,d}:=\bigl\{\, 1,\dotsc, M(j,d-1)\,\bigr\},
\]
and  we parametrize  the basis $\eB_{j,d-1}$ as
\[
\eB_{j,d-1}=\bigl\{ \,Y_{j,\beta};\;\;\beta\in B_{j,d}\,\bigr\},
\]
where $Y_{j,\beta}$ is a homogeneous  harmonic polynomial of degree $j$ in the variables $\bx'=(x_1,\dotsc,x_{d-1})$.  For any integers $j, n$, $0\leq j\leq n$, and  any $\beta\in B_{j,d}$, we define $Z_{n,j,\beta}\in C^\infty(S^{d-1})$ by
\begin{equation}
Z_{n, j,\beta}(\bx) := C_{n,j,d}P_{n,d}^{(j)}(x_d)  Y_{j,\beta}(\bx'),\;\;\forall\bx\in S^{d-1},
\label{eq: zn}
\end{equation}
where $P_{n,d}^{(j)}$  denotes the $j$-th order derivative of the  Legendre polynomial $P_{n,d}$ defined by (\ref{eq: leg1}), while the universal constant $C_{n,j,d}$ is described in (\ref{eq: leg2}). Then, for fixed $n$,  the collection of functions
\[
\bigl\{\, Z_{n,j,\beta}\in C^\infty(S^{d-1});\;\;0\leq j\leq n, \;\;\beta\in B_{j,d}\,\bigr\}
\]
is  the orthonormal basis $\eB_{n,d}$.   Any $\bv\in\bsV_\nu$ admits a decomposition
\[
\bv=\sum_{n=0}^\nu\sum_{j=0}^n\sum_{\beta\in B_{j,d}}v_{n,j,\beta} Z_{n,j,\beta},\;\;v_{n,j,\beta}\in\bR,
\]
so that
\[
\Hess(\bv) =\sum_{n=0}^\nu\sum_{j=0}^n\sum_{\beta\in B_{j,d}}v_{n,j,\beta} \Hess\bigl(\, Z_{n,j,\beta}\,\bigr).
\]
From the description (\ref{eq: zn}) we deduce that
\begin{equation}
\Hess(Z_{n,j,\beta})=0,\;\;\forall j\geq 3.
\label{eq:  znhess}
\end{equation}
Next, we observe that when $j=0$ we have $M(0,d-1)=1$ and  $\eB_{0,d-1}$ consists of the constant function $\bsi_{d-2}^{-1/2}$. We deduce
\[
\eB_{0,d-1}=\{\bsi_{d-2}^{-1/2}\},\;\;B_{j,d}=\{1\},
\]
\[
\Hess(Z_{n,0, 1})= C_{n,0, d}\bsi_{d-2}^{-1/2} \Hess\bigl(\,P_{n,d}(x_d)\,\bigr).
\]
Using the equalities
\[
P_{n,d}(t)= P_{n,d}(1)+ P_{n,d}'(1)(t-1)+\mbox{higher order terms},
\]
we deduce
\[
\begin{split}
P_{n,d}(x_d) & = P_{n,d}(1)+ P_{n,d}'(1)(z_d-1)+\mbox{higher order terms}\\
&\stackrel{(\ref{eq: tayxd})}{=}  P_n(1)-\frac{P_{n,d}'(1)}{2}|\bx'|^2+ \mbox{higher order terms}\\
&\stackrel{(\ref{eq: leg3})}{=}1-\frac{1}{4}(n+1)\left(n+\frac{d-3}{2}\right)|\bx'|^2+ \mbox{higher order terms} .
\end{split}
\]
This shows that
\begin{equation}
\Hess\bigl(\,P_{n,d}(x_d)\,\bigr)=-\frac{1}{2}(n+1)\left(n+\frac{d-3}{2}\right)\one_{d-1},
\label{eq: hess-sph1}
\end{equation}
where $\one_{d-1}$ denotes the identity  $(d-1)\times(d-1)$-matrix.   Hence
\begin{equation}
Z_{n,0,d}\in\bsK_\nu(p_0),
\label{eq: crn0}
\end{equation}
\begin{equation}
\Hess\bigl(Z_{n,0,1}\,\bigr)=-\frac{1}{2}\bsi_{d-2}^{-1/2}C_{n,0,d}(n+1)\Bigl(\,n+\frac{d-3}{2}\,\Bigr)\one_{d-1}.
\label{eq: zn0}
\end{equation}
Similarly
\[
P_{n,d}'(x_d)=P_{n,d}'(1)-\frac{1}{2}P_{n,d}^{(2)}(1)|\bx'|^2 +\mbox{higher order terms},
\]
which implies that
\begin{equation}
\Hess(Z_{n,1,\beta})=0,\;\;\forall  n,\;\;\beta\in B_1=\{1,\dotsc, d-1\}.
\label{eq: zn1}
\end{equation}
For any $\beta\in B_{2,d}$, we denote by $H_\beta$ the Hessian   of $Y_\beta(\bx')$ at $\bx'=0\in\bR^{d-1}$. We deduce that
\begin{equation}
Z_{n,2,\beta}\in \bsK_\nu(p_0),
\label{eq: crn2}
\end{equation}
\begin{equation}
\Hess\bigl(Z_{n,2,\beta}) = C_{n,2,d} P_{n,d}^{(2)}(1)  H_\beta.
\label{eq: zn2}
\end{equation}
Using (\ref{eq:  znhess}), (\ref{eq: zn1}) and (\ref{eq: zn2}) we conclude that
\[
\begin{split}
\Hess(\bv) & =\sum_{n=2}^\nu\sum_{\beta\in B_{2,d}} v_{n,2,\beta} \Hess\bigl( Z_{n,2,\beta}\,\bigr) +\sum_{n=0}^\nu v_{n,0,1}\Hess\bigl(\, Z_{n,0,1}\,\bigr)\\
& =\sum_{\beta\in B_{2,d}} \Bigl(\sum_{n=2}^\nu v_{n,2,\beta} C_{n,2,d} P_{n,d}^{(2)}(1) \,\Bigr) H_\beta\\
&-\frac{1}{2}\bsi_{d-2}^{-1/2}\Biggl(\sum_{n=0}^\nu v_{n,0,1} C_{n,0,d}(n+1)\Bigl(\,n+\frac{d-3}{2}\,\Bigr)\,\Biggr)\cdot \one_{d-1}.
\end{split}
\]
The last   equality can be rewritten in a more  convenient form as follows.    Define
\begin{equation}
\ba_0=\ba_0(\nu):= -\frac{1}{2}\bsi_{d-2}^{-1/2}\sum_{n=0}^\nu  C_{n,0,d}(n+1)\Bigl(\,n+\frac{d-3}{2}\,\Bigr)Z_{n,0,1}\in\bsV_\nu,
\label{eq: a0}
\end{equation}
and for $\beta\in B_{2,d}$, set
\begin{equation}
\ba_\beta=\ba_\beta(\nu):= \sum_{n=2}^\nu C_{n,2,d} P_{n,d}^{(2)}(1) Z_{n,2,\beta}\in\bsV_\nu.
\label{eq: abeta}
\end{equation}
We deduce that
\[
\Hess(\bv)=(\bv, \ba_0)\one_{d-1}+\sum_{\beta\in B_{2,d}} (\bv,\ba_\beta)H_\beta.
\]
 Note that the vectors $\ba_0,\ba_\beta$ are mutually orthogonal, and they span a vector space $\bsL_\nu$ of dimension
 \[
 \ell=\ell(d)=M(2,d-1)+1\stackrel{(\ref{eq: leg0})}{=}\binom{d}{2}.
 \]
 Moreover,    the conditions (\ref{eq: crn0}) and (\ref{eq: crn2}) imply that $\bsL_\nu\subset \bsK_\nu(p_0)$.    Define
 \begin{equation}
 \be_0:=\frac{1}{|\ba_0|}\ba_0,\;\;  \be_\beta:=\frac{1}{|\ba_\beta|}\ba_\beta,\;\;\beta\in B_{2,d},
 \label{eq: good-basis-sph}
 \end{equation}
 \[
r_0=r_0(\nu)= |\ba_0|^2=\frac{1}{4}\bsi_{d-2}^{-1}\sum_{n=0}^\nu  C_{n,0,d}^2(n+1)^2\Bigl(\,n+\frac{d-3}{2}\,\Bigr)^2,
\]
\[
r_\beta=r_\beta(\nu)=|\ba_\beta|^2=\sum_{n=2}^\nu C_{n,2,d}^2 P_{n,d}^{(2)}(1)^2 ,\;\;\beta\in B_{2,d}.
 \]
 Note that the collection $\{\be_0,\;\be_\beta,\;\;\beta\in B_{2,d}\}$ is an \emph{orthonormal} basis of $\bsL_\nu$. For any $\bv\in\bsK_\nu(p_0)$, we denote by $\bar{\bv}$ its orthogonal projection onto $\bsL_\nu$, and we set
 \[
 \bar{v}_0:=(\bv,\be_0),\;\;\bar{v}_\beta:=(\bv,\be_\beta),\;\;\beta\in B_{2,d}.
 \]
 We  deduce  that for any $\bv\in \bsK_\nu(p_0)$, we have
 \begin{equation}
 \Hess(\bv)=\Hess(\bar{\bv}) = r_0^{1/2}\bar{v}_0\one_{d-1}+\sum_\beta r_\beta^{1/2} \bar{v}_\beta H_\beta.
 \label{eq: hess-special}
 \end{equation}
 For $\bv\in\bsK_\nu(p_0)$ we set
 \[
 Q_\nu(\bv):=\bigl|\det \Hess(\bv)\bigr|.
 \]
 Note that $Q_\nu(\bv)$ is positively homogeneous of degree $d-1$. Using Lemma \ref{lemma: int1} in the special case
 \[
 n=\dim \bsK_\nu(p_0)= N_\nu-(d-1),\;\;n_1=\dim \bsL_\nu=\ell,\;\; n_0= N_\nu-(d-1)-\ell,
 \]
 we deduce
 \[
\Delta_\bx(\bsV_\nu)= \int_{S(\bsK_\nu(p_0)\,)} Q_\nu(\bv)\,|dS(\bv)|=\bsi_{N_\nu-\ell-d}\int_{B(\bsL_\nu)} (1-|\bar{\bv}|^2)^{\frac{N_\nu-\ell-d-1}{2}}Q_\nu(\bar{\bv})\,|dV(\bar{\bv})|.
 \]
 Using Lemma \ref{lemma: int2}  we deduce
 \[
\int_{B(\bsL_\nu)} (1-|\bar{\bv}|^2)^{\frac{N_\nu-\ell-d-1}{2}}Q_\nu(\bar{\bv})\,|dV(\bar{\bv})|=\frac{\Gamma(\frac{\ell+d-1}{2})\Gamma(\frac{N_\nu-\ell-d+1}{2}}{2\Gamma(\frac{N_\nu}{2})}\underbrace{\int_{S(\bsL_\nu)}Q_\nu(\bar{\bv})\,|dS(\bar{\bv})|}_{=:I_\nu}.
 \]
 Using the  equality (\ref{tag: si}), we conclude
 \begin{equation}
  \int_{S(\bsK_\nu(p_0)\,)} Q_\nu(\bv)\,|dS(\bv)|=\frac{\pi^{\frac{N_\nu-\ell-d+1}{2}}\Gamma(\frac{\ell+d}{2})}{\Gamma(\frac{N_\nu}{2})}I_\nu .
  \label{eq: avhess1}
  \end{equation}
 Next, we compute the  Jacobian of the adjunction map $\eA^\dag$ at $p_0$. We use the coordinates $\bx'$ near $p_0$. For $i=1,\dotsc, d-1$ we  have
 \[
 \eA^\dag_{p_0}\pa_{x_i}=\sum_{n=0}^\nu\sum_{j=0}^n\sum_{\beta\in B_{j,d}}  \pa_{x_i}\bigl(\, Z_{n,j,\beta}(p_0)\,\bigr) Z_{n,j,\beta} =\sum_{n=1}^\nu\sum_{\beta\in B_{1,d}}  \pa_{x_i}\bigl(\, Z_{n,1,\beta}(p_0)\,\bigr) Z_{n,j,\beta}.
 \]
 Using (\ref{eq: leg4}), we deduce that $B_{1,d}=\{1,\dotsc, d-1\}$ and for any $\beta\in B_{1,d}$ we have
 \[
 Y_\beta =\bsi^{-1/2}_{d-3}C_{1,0,d-1}x_\beta,\;\; Z_{n,1,\beta}=\bsi^{-1/2}_{d-3}C_{1,0,d-1}C_{n,1,d}P_{n,d}'(x_d) x_\beta.
 \]
 We deduce that
 \begin{equation}
 \eA^\dag_{p_0}\pa_{x_i}= \bsi^{-1/2}_{d-3}C_{1,0,d-1}\sum_{n=1}^\nu C_{n,1,d}P_{n,d}'(1)  Z_{n,1,i}.
 \label{eq: adj-sph0}
 \end{equation}
 This shows that the vectors $\eA^\dag_{p_0}\pa_{x_i}$, $i=1,\dotsc, d-1$, are mutually orthogonal and they have identical lengths
 \begin{equation}
 | \eA^\dag_{p_0}\pa_{x_i}|=r(\nu)^{1/2},\;\;r(\nu)= \bsi^{-1}_{d-3}C_{1,0,d-1}^2 \sum_{n=1}^\nu \bigl(C_{n,1,d}P_{n,d}'(1) \bigr)^2.
 \label{eq: adj-length}
 \end{equation}
 We deduce that the Jacobian of $\eA^\dag_{p_0}$ is
 \begin{equation}
 J_\nu=r(\nu)^{\frac{d-1}{2}}.
 \label{eq: jac1}
 \end{equation}
 The equalities (\ref{eq: av-hom}), (\ref{eq: avhess1}) and (\ref{eq: jac1}) now imply that
 \[
 \mu(S^{d-1},\bsV_\nu)=\frac{\bsi_{d-1}}{\bsi_{N_\nu-1} }\cdot\frac{\pi^{\frac{N-\ell-d+1}{2}}\Gamma(\frac{\ell+d}{2})}{r(\nu)^{\frac{d-1}{2}}\Gamma(\frac{N_\nu}{2})} I_\nu.
 \]
 Using (\ref{tag: si}) we  can   simplify this  to
 \begin{equation}
 \mu(S^{d-1},\bsV_\nu)=\frac{\Gamma(\frac{\ell+d}{2})}{(\pi r(\nu))^{\frac{d-1}{2}}\Gamma(\frac{d}{2})} I_\nu.
 \label{eq: expect-sph1}
 \end{equation}
 To obtain the asymptotics of  $\mu(S^{d-1},\bsV_\nu)$ as $\nu\ra \infty$ we need to understand the asymptotics of the quantities
 \[
 r(\nu),\;\;r_0(\nu),\;\; r_\beta(\nu),\;\;\beta\in B_{2,d-1}.
 \]
 To achieve this, note first that (\ref{eq: leg3})  and (\ref{eq: leg2}) imply  that
 \begin{equation*}
 \begin{split}
 C_{n,1,d}^2P_{n,1,d}'(1)^2 &= \frac{1}{2^{d-2}\Gamma(\frac{d-1}{2})^2} \frac{(2n+d-2)([n+d-3]_{d-3})^2}{[n+d-2]_{d-1}}\frac{(n+1)^2}{4}\left(n+\frac{d-3}{2}\right)^2\\
 &=\frac{1}{2^{d-1}\Gamma(\frac{d-1}{2})^2}\frac{ A_{2d-1}(n)}{B_{d-1}(n)},
 \end{split}
  \tag{$\br$}
 \label{tag: r}
 \end{equation*}
 where $A_{2d-1}(x)$ (respectively $B_{d-1}(x)$) is a monic polynomial of degree $(2d-1)$ (respectively  $d-1$). Using (\ref{tag: sasy}), we deduce   that
 \[
 \sum_{n=1}^\nu C_{n,1,d}^2P_{n,1,d}'(1)^2\sim \frac{1}{2^{d-1}(d+1)\Gamma(\frac{d-1}{2})^2}\nu^{d+1},\;\;\mbox{as}\;\;\nu\ra \infty,
 \]
 so that
 \begin{equation}
 r(\nu)\sim  \frac{C_{1,0,d-1}^2}{2^{d-1}\Gamma(\frac{d-1}{2})^2 \bsi_{d-3}(d+1)}\nu^{d+1},\;\;\mbox{as}\;\;\nu\ra \infty,
 \label{eq: asyr}
 \end{equation}
 where
 \[
 C_{1,0,d-1}^2=\frac{(d-1)(d-3)!}{2^{d-3}\Gamma(\frac{(d-2)}{2})^2}.
 \]
 Invoking   (\ref{eq: leg3})  and (\ref{eq: leg2})  again we deduce that
 \begin{equation*}
 \begin{split}
 C_{n,2,d}^2P_{n,2,d}^{(2)}(1)^2 &=\frac{1}{2^{d-2}\Gamma(\frac{d-1}{2})^2}\frac{(2n+d-2)([n+d-3]_{d-3})^2}{[n+d-1]_{d+1}} \left(\frac{1}{4}\binom{n+2}{2}\left[n+\frac{d-3}{2}\right]_2\right)^2\\
& = \frac{1}{2^{d+3}\Gamma(\frac{d-1}{2})^2}\frac{A_{2d+3}(n)}{B_{d+1}(n)},
 \end{split}
 \tag{$\br_\beta$}
 \label{tag: rbeta}
\end{equation*}
  where $A_{2d+3}(x)$ (respectively $B_{d-1}(x)$) is a monic  polynomial of degree $(2d+3)$ (respectively  $d+1$). Using (\ref{tag: sasy}) we deduce that
  \begin{equation}
 \forall\beta\in B_{2,d-1};\;\; r_\beta(\nu)= \sum_{n=2}^\nu C_{n,2,d}^2P_{n,2,d}'(1)^2\sim  \frac{1}{2^{d+3}(d+3)\Gamma(\frac{d-1}{2})^2}\nu^{d+3},\;\;\mbox{as}\;\;\nu\ra \infty.
  \label{eq: asib}
  \end{equation}
  Using (\ref{eq: leg2}), we deduce
  \begin{gather*}
  C_{n,0,d}^2 (n+1)^2\left(n+\frac{d-3}{2}\right)^2 =\frac{(2n+d-2)[n+d-3]_{d-3}}{2^{d-2}\Gamma(\frac{d-1}{2})^2}  (n+1)^2 \left(n+\frac{d-3}{2}\right)^2\\
 =\frac{1}{2^{d-3}\Gamma(\frac{d-1}{2})^2} A_{d+2}(n),
  \tag{$\br_0$}
  \label{tag: r0}
  \end{gather*}
  where $A_{d+2}(x)$ denotes a monic polynomial of degree $d+2$.  Invoking (\ref{tag: sasy}) again we deduce  that as $\nu\ra \infty$ we have
\begin{equation}
  r_0(\nu)=\frac{1}{4\bsi_{d-2}}\sum_{n=0}^\nu   C_{n,0,d}^2 (n+1)^2\left(n+\frac{d-3}{2}\right)^2 \sim  \frac{1}{2^{d-1}\Gamma(\frac{d-1}{2})^2\bsi_{d-2} (d+3)}\nu^{d+3}.
  \label{eq: asy0}
  \end{equation}
  Define
  \[
  \bar{r}_0=\lim_{\nu\ra\infty}\nu^{-(d+3)}r_0(\nu),\;\;\bar{r}_\beta= \lim_{\nu\ra\infty}\nu^{-(d+3)}r_\beta(\nu),\;\;\bar{r}=\lim_{\nu\ra \infty}\nu^{-(d+1)}r(\nu).
  \]
The precise values of these constants  can be read off (\ref{eq: asyr})-(\ref{eq: asy0}).    Denote by $\bsL_\infty$ the Euclidean space of dimension $\ell=\binom{d}{2}$ with Euclidean coordinates $u_0$, $u_\beta$, $\beta\in B_{2,d}$, and we set
\begin{equation}
A_\nu(\bu)=   r_0(\nu)^{1/2}u_0\one_{d-1}+\sum_\beta r_\beta(\nu)^{1/2}u_\beta H_\beta,\;\;A_\infty(\bu)=\bar{r_0}^{1/2}u_0\one_{d-1}+\sum_\beta \bar{r}_\beta^{1/2} u_\beta H_\beta.
\label{eq: ainfty}
\end{equation}
We can now rewrite (\ref{eq: expect-sph1}) as follows
 \[
 \mu(S^{d-1},\bsV_\nu)=\frac{\Gamma(\frac{\ell+d}{2})}{(r(\nu)\pi)^{\frac{d-1}{2}}\Gamma(\frac{d}{2})} \int_{S(\bsL_\infty)}|\,\det A_\nu(\bu) \,|\,|dS(\bu)|.
 \]
 The estimates (\ref{eq: asib}) and (\ref{eq: asy0}) show   that  as $\nu\ra \infty$, we have
 \[
 |\det A_\nu(\bu)|\sim \nu^{\frac{(d+3)(d-1)}{2}}|\det A_\infty(\bu)|,
\]
  \emph{uniformly} with respect to $\bu\in S(\bsL_\infty)$.   We deduce that as $\nu\ra \infty$, we have
\begin{equation}
\mu(S^{d-1},\bsV_\nu) \sim \frac{\Gamma(\frac{\ell+d}{2})}{(\pi\bar{r})^{\frac{d-1}{2}}\Gamma(\frac{d}{2})} \nu^{d-1}\int_{|\bu|=1} |\det A_\infty(\bu)\,|\,|dS(\bu)|.
\label{eq: expect-sph2}
\end{equation}
This proves (\ref{eq: av-sph-harm}) where
\[
K_d= \frac{2\Gamma(\frac{\ell+d}{2})}{(\pi\bar{r})^{\frac{d-1}{2}}\Gamma(\frac{d}{2})(d-1)!} \int_{|\bu|=1} |\det A_\infty(\bu)\,|\,|dS(\bu)|,\;\;\ell=\binom{d}{2}.
\]
 \end{proof}

\begin{remark}  We want  to analyze  what happens to the above expectation if we change the     $L^2$-metric product on $\bsV_\nu(d)$ to a new Euclidean metric  so that the resulting quadruple  $(S^{d-1}, g, \bsV_\nu(d), h)$ continues to be homogeneous with respect to the action of $SO(d)$.

To perform such changes we use the fact that each of the spaces $\eY_{n,d}$ is an irreducible representation of $SO(d)$.   Any sequence $w=(w_n)_{n\geq 0}$ of positive real numbers   determines  a Euclidean metric $\|-\|_w$ on $\bsV_\nu(d)$  as follows. If 
\[
\bv= \sum_{n=0}^\nu \bv_n\in\bsV_\nu(d),\;\;\bv_n\in \eY_{n,d},
\]
then we set
\[
\|\bv\|^2_w:=\sum_{n=0}^\nu\frac{1}{w_n^2} \|\bv_n\|^2_{L^2(S^{d-1})}.
\]
In the sequel, we will choose the weights $w$ of the form
\begin{equation}
w_{p,\nu}=(\nu+1)^p,\;\;p\in\bR.
\label{eq: weight}
\end{equation}
The corresponding metric $\|-\|_{w_p}$ is  (equivalent to) the metric of the Sobolev  Hilbert space $H^{-p}$ consisting of  distributions with    ``derivatives up to order $-p$ in $L^2$''.

The  quadruple $(S^{d-1}, g,\bsV_\nu(d),\|-\|_w)$ is homogeneous and we denote by $\mu_p(S^{d-1},\bsV_\nu)$ its expectation. The collection
\[
\bigl\{\, w_nZ_{n,j,\beta}\in C^\infty(S^{d-1});\;\;0\leq j\leq n\leq \nu, \;\;\beta\in B_{j,d}\,\bigr\}
\]
is an orthonormal basis of $\bsV_\nu$ with respect to the inner product $h_w$    associated to $\|-\|_w$. Any $\bv\in\bsV_\nu$ admits a decomposition
\[
\bv=\sum_{n=0}^\nu\sum_{j=0}^n\sum_{\beta\in B_{j,d}}v_{n,j,\beta} w_nZ_{n,j,\beta},\;\;v_{n,j,\beta}\in\bR.
\]
The arguments in the proof of Theorem \ref{th: sph-harm} show that for $\bv\in\bsV_\nu$ we have
\[
\begin{split}
\Hess(\bv) & =\sum_{\beta\in B_{2,d}} \Bigl(\sum_{n=2}^\nu w_nv_{n,2,\beta} C_{n,2,d} P_{n,d}^{(2)}(1) \,\Bigr) H_\beta\\
&-\frac{1}{2}\bsi_{d-2}^{-1/2}\Biggl(\sum_{n=0}^\nu w_nv_{n,0,1} C_{n,0,d}(n+1)\Bigl(\,n+\frac{d-3}{2}\,\Bigr)\,\Biggr)\cdot \one_{d-1}.
\end{split}
\]
In particular, if $\Hess(\bv)\neq 0$, then the North Pole is a critical point of $\bv$, i.e., $\bv\in \bsK_\nu(p_0)$.  Define
\begin{equation}
\ba_0=\ba_0(\nu,w):= -\frac{1}{2}\bsi_{d-2}^{-1/2}\sum_{n=0}^\nu  w_nC_{n,0,d}(n+1)\Bigl(\,n+\frac{d-3}{2}\,\Bigr)Z_{n,0,1}\in\bsV_\nu,
\label{eq: a0w}
\end{equation}
and for $\beta\in B_{2,d}$ set
\begin{equation}
\ba_\beta=\ba_\beta(\nu,w):= \sum_{n=2}^\nu w_nC_{n,2,d} P_{n,d}^{(2)}(1) Z_{n,2,\beta}\in\bsV_\nu.
\label{eq: abetaw}
\end{equation}
We deduce that
\[
\Hess(\bv)=(\bv, \ba_0)\one_{d-1}+\sum_\beta (\bv,\ba_\beta)H_\beta.
\]
Define
 \begin{equation}
 \be_0:=\frac{1}{|\ba_0|}\ba_0,\;\;  \be_\beta:=\frac{1}{|\ba_\beta|}\ba_\beta,\;\;\beta\in B_{2,d},
 \label{eq: good-basis-sphw}
 \end{equation}
 \[
r_0=r_0(\nu,w):= |\ba_0|^2,\;\;r_\beta=r_\beta(\nu,w):=|\ba_\beta|^2,\;\;\beta\in B_{2,d}.
 \]
For any $\bv\in\bsK_\nu(p_0)$,  we set
 \[
 Q_\nu(\bv):=|\det \Hess(\bv)|,\;\;\bar{v}_0=(\bv,\be_0),\;\;\bar{v}_\beta=(\bv,\be_\beta),\;\;\beta\in B_{2,d},
 \]
and we deduce
 \begin{equation}
 \Hess(\bv)=\Hess(\bar{\bv}) = r_0^{1/2}\bar{v}_0\one_{d-1}+\sum_\beta r_\beta^{1/2} \bar{v}_\beta H_\beta,
 \label{eq: hess-specialw}
 \end{equation}
 \begin{equation}
 \int_{S(\bsK_\nu(p_0))}Q_\nu(\bv)\,|dS(\bv)|=\frac{\pi^{\frac{N_\nu-\ell-d+1}{2}}\Gamma(\frac{\ell+d}{2})}{\Gamma(\frac{N_\nu}{2})}
  \int_{S(\bsL_\nu)}Q_\nu(\bar(\bv))\,|dS(\bar{\bv})|,
\label{eq: avhess1qw}
\end{equation}
where $\bar{\bv}$ denotes the   orthogonal projection of $\bv$ onto the space $\bsL_\nu$ spanned by $\be_0,\be_\beta$.       Similarly, we have
\begin{equation}
 \eA^\dag_{p_0}\pa_{x_i}= \bsi^{-1/2}_{d-3}C_{1,0,d-1}\sum_{n=1}^\nu w_nC_{n,1,d}P_{n,d}'(1)  Z_{n,1,i}.
 \label{eq: adj-sph0w}
 \end{equation}
This shows again  that the vectors $\eA^\dag_{p_0}\pa_{x_i}$, $i=1,\dotsc, d-1$, are mutually orthogonal and they have identical length
 \begin{equation}
 | \eA^\dag_{p_0}\pa_{x_i}|=r(\nu)^{1/2},\;\;r(\nu)= \bsi^{-1}_{d-3}C_{1,0,d-1}^2 \sum_{n=1}^\nu \bigl(w_nC_{n,1,d}P_{n,d}'(1) \bigr)^2.
 \label{eq: adj-lengthw}
 \end{equation}
 We deduce that the Jacobian of $\eA^\dag_{p_0}$ is
 \begin{equation}
 J_\nu=r(\nu)^{\frac{d-1}{2}}.
 \label{eq: jac1w}
 \end{equation}
If the exponent $p$ in (\ref{eq: weight}) is nonnegative,  then using (\ref{tag: r}), (\ref{tag: rbeta}), (\ref{tag: r0})  and the Euler-Maclaurin summation formula (\ref{tag: sasy}), we deduce as before that as $\nu\ra \infty$ we have
\begin{equation}
 r(\nu,w)\sim K_1\nu^{d+1+2p},\;\;r_\beta(\nu,w)\sim K_2r^{d+1+2p},\;\;r_0(\nu,w)\sim K_3\nu^{d+3+2p},
 \label{eq: rrr}
 \end{equation}
 where above and in the sequel we  will use the symbols $K_1,K_2,\dotsc,$ to denote positive constants that depend only on $d$ and $p$.   This shows that
 \[
 \mu_p(S^{d-1}, \bsV_\mu) \sim K_4 \dim \bsV_\nu.
 \]
 If  the exponent $p$ in (\ref{eq: weight}) is $\ll 0$,  then a similar argument shows that
 \[
  \mu_p(S^{d-1}, \bsV_\mu) \sim K_5,\;\;\mbox{as}\;\;\nu\ra \infty.\proofend
  \]
  \label{rem: depend}
\end{remark}

We conclude this section with a computation suggested by the recent  results of Nazarov-Sodin, \cite{NS}.

\begin{theorem} We  denote by $\eY_n$ the   eigenspace corresponding to the eigenvalue $\lambda_n=n(n+1)$ of the Laplacian on $S^2$, and we set   $\mu(\eY_n):=\mu(S^2,\eY_n)$.   Then\footnote{Let us point out that $\frac{2}{\sqrt{3}} \approx 1.154$.}
\begin{equation}
\mu(\eY_n)\sim  \frac{2}{\sqrt{3}}n^2 \;\;\mbox{as}\;\;n\ra \infty.
\label{eq: harm2}
\end{equation}
\label{th: harm2}
\end{theorem}

\begin{proof} The computation is very similar to the computations in Theorem \ref{th: sph-harm}, but much simpler.   We continue to use the notations in the proof of that theorem.  In particular, $\bsK_n$ denotes the space of harmonic polynomials in $\eY_n$ that admit the North Pole $p_0=(0,0,1)\in\bR^3$ as a critical point.

 An orthonormal basis of $\eY_n$ is given by the
\[
Z_{n,j,\beta},\;\; 0\leq j\leq n,\;\;\beta\in B_j:= B_{j,3}.
\]
Any $\by\in\eY_n$ admits a decomposition
\[
\by=\sum_{j=0}^n\sum_{\beta\in \eB_j}  y_{j,\beta} Z_{n, j,\beta}.
\]
We conclude   as in the proof of Theorem \ref{th: sph-harm} that
\begin{equation}
\Hess(\by)= \sum_{\beta\in B_2}y_{2,\beta} \Hess(Z_{n, 2,\beta})+ y_{0,1}\Hess(Z_{n,0,1}).
\label{eq: hess2}
\end{equation}
We have
\[
\Hess(Z_{n,0,1})= -\frac{1}{2\bsi_1^{1/2}}C_{n,0,3} n(n+1)\one_2=-\frac{1}{2(2\pi)^{1/2}}n(n+1)\sqrt{n+\frac{1}{2}}\one_2.
\]
In this case, the basis  $B_2$ consists of  two elements, $B_2=\{1,2\}$, and we have
\[
Z_{n,2,1}= C_{n,2,3}P^{(2)}_{n,3}(x_3) Y_1(x_1,x_2),\;\; Z_{n,2,2}=  C_{n,2,3}P^{(2)}_{n,3}(x_3) Y_2(x_1,x_2),
\]
where
\[
Y_1( x_1, x_2)= C_1(x_1^2-x_2^2),\;\;Y_2(x_1,x_2)= C_2x_1x_2,
\]
and  the constants $C_1, C_2$ are found from the identities
\[
1= C_1^2\int_{S^1}(x_1^2-x_2^2)^2 |ds|= C_2^2\int_{S^1} x_1^2x_2^2 |ds|.
\]
Note that
\[
\int_{S^1} x_1^2x_2^2 |ds|=\frac{1}{4}\int_{S^1} \sin^2(2\theta)\,|d\theta|=\frac{\pi}{4},\;\;\mbox{and}\;\; \int_{S^1}(x_1^2-x_2^2)^2 |ds|= \int_{S^1}\cos^2(2\theta) |d\theta|= \pi,
\]
so that
\[
C_1=\pi^{-1/2},\;\;\mbox{and}\;\;C_2=\frac{2}{\pi^{1/2}}.
\]
We set
\[
H_1:=\Hess(Y_1)= \frac{2}{\pi^{1/2}}\left[
\begin{array}{cc}
1 & 0\\
0 &-1
\end{array}
\right],\;\; H_2:= \Hess(Y_2)= \frac{2}{\pi^{1/2}}\left[
\begin{array}{cc}
0 &1\\
1& 0
\end{array}
\right].
\]
We deduce that
\[
\Hess(\by)= y_{2,1}C_{n,2,3}P_{n,3}^{(2)}(1) H_1+ y_{2,2} C_{n,2,3}P^{(2)}_{n,3}(1) H_2 -y_{0,1} \frac{1}{2\bsi_1^{1/2}}C_{n,0,3} n(n+1)\one_2
\]
\[
=\frac{2}{\pi^{1/2}}C_{n,2,3}P_{n,3}^{(2)}(1)\left(y_{2,1} \left[
\begin{array}{cc}
1 & 0\\
0 &-1
\end{array}
\right]\ + y_{2,2}  \left[
\begin{array}{cc}
0 & 1\\
1 & 0
\end{array}
\right]\right) -y_{0,1}\frac{1}{2(2\pi)^{1/2}}n(n+1)\left(n+\frac{1}{2}\right)^{1/2}\one_2.
\]
We set
\begin{equation}
a(n):=C_{n,2,3}P_{n,3}^{(2)}(1)=\frac{1}{8}\left(\Bigl(n+\frac{1}{2}\Bigr)[n+2]_4\right)^{1/2},\;\;b(n):=\frac{1}{4\sqrt{2}} n(n+1)\left(n+\frac{1}{2}\right)^{1/2}.
\label{eq: ab}
\end{equation}
Note  that
\[
a(n)^2\sim \frac{n^5}{64},\;\;b(n)^2\sim  \frac{n^5}{32} \;\;\mbox{as}\;\;n\ra \infty.
\]
To ease the presentation, we  set
\[
u_1:=y_{2,1},\;\;u_2:=y_{2,2}, \;\;u=y_{0,1},
\]
and we deduce that
\[
\Hess(\by)=\frac{2}{\pi^{1/2}}\left[
\begin{array}{cc}
a u_1-bu & au_2\\
au_2 & -au_1-bu
\end{array}
\right].
\]
Denote by $\bsL$ the space spanned by $Z_{n,2,1}, Z_{n,2,2}, Z_{n,0,1}$, $\ell: =\dim \bsL=3$.   It is contained in $\bsK_n$, and if $\by\perp \bsL$, then $\Hess (\by)=0$. For
\[
\by= u_1Z_{n,2,1}+u_2 Z_{n,2,2}+ u Z_{n,0,1}\in \bsL,
\]
we have
\[
|\det \Hess (\by)|=\frac{4}{\pi} |b^2 u^2- a^2u_1^2-a^2 u_2^2|\sim \frac{n^5}{16\pi}|2u^2-u_1^2-u_2^2|,\;\;n\ra \infty.
\]
Arguing as in (\ref{eq: adj-sph0}),  we deduce that
\[
\eA^\dag_{p_0}\pa_{x_i}= (\pa_{x_i}Z_{n,1,i}(0))Z_{n,1,i},\;\;i=1,2.
\]
Recall that
\[
Z_{n,1,i}=C_0 C_{n,1,3}P'_{n,3}(x_3)x_i,\;\; 1=C_0^2\int_{S^1} x_i^2  |ds|=\pi C_0^2,
\]
so that
\[
|\eA^\dag_{p_0}\pa_{x_i}|= C_0 C_{n,1,3}P_{n,3}'(1)  Z_{n,1,i}\stackrel{(\ref{eq: leg2})}{=} \pi^{-1/2} \left(\frac{2n+1}{2n(n+1)}\right)^{1/2}\times \frac{1}{2}n(n+1)
\]
\[
=\frac{1}{(4\pi)^{1/2}}\left(\,\Bigl(n+\frac{1}{2}\,\Bigr)n(n+1)\,\right)^{1/2}.
\]
Hence
\[
J(\eA^\dag_{p_0})=\frac{n(n+1)(n+\frac{1}{2})}{4\pi}\sim\frac{n^3}{4\pi}\;\;\mbox{as}\;\;n\ra \infty.
\]
Putting together all of the above and invoking (\ref{tag: bmup}), we deduce that
\[
\frac{\mu(\eY_n)}{n^2}\sim {\rm area}\,(S^2)\times \pi^{-\frac{3+2}{2}} \times \frac{1}{4}\int_{\bR^3} e^{-(u^2+u_1^2+u_2^2)}|2u^2-(u_1^2+u_2^2)|\,|dudu_1du_2|
\]
\[
=\frac{1}{\pi^{3/2}}\int_{\bR^3} e^{-(u^2+u_1^2+u_2^2)}|2u^2-(u_1^2+u_2^2)|\,|dudu_1du_2|.
\]
Using  cylindrical coordinates $(u,r,\theta)$, $u_1=r\cos\theta$, $u_2=r\sin\theta$, we  deduce
\[
\int_{\bR^3} e^{-(u^2+u_1^2+u_2^2)}\,\bigl|2u^2-(u_1^2+u_2^2)\bigr|\,|dudu_1du_2|=\int_0^{2\pi}\int_{-\infty}^\infty\int_0^\infty e^{-(u^2+r^2)}|2u^2-r^2|rdr du d\theta
\]
\[
=2\pi \int_{-\infty}^\infty\int_0^\infty e^{-(u^2+r^2)}|2u^2-r^2|rdr du =4\pi \underbrace{\int_{0}^\infty\int_0^\infty e^{-(u^2+r^2)}|2u^2-r^2|rdr du}_{=: I}.
\]
Hence
\[
\mu(\eY_n)\sim 4\pi^{-1/2}I n^2.
\]
To proceed further, we use polar coordinates $u= t\cos\vfi$, $r=t\sin\vfi$,  $0<\vfi<\frac{\pi}{2}$,  $t\geq 0$  and we deduce
\[
I= \int_0^{\infty}\left(\int_0^{\pi/2} |2\cos^2\vfi-\sin^2\vfi|\sin\vfi d\vfi\right) e^{-t^2}t^4 dt
\]
\[
=\left(\int_0^\infty e^{-t^2} t^4 dt\right)\,\cdot \,\left(\int_0^{\pi/2} |3\cos^2\vfi-1| \sin\vfi d\vfi\right)
\]
(use the substitutions $s=t^2$,  $x=\cos\vfi$)
\[
=\frac{1}{2}\left(\int_0^\infty e^{-s} s^{3/2} ds\right)\,\cdot\,\left(\int_0^1 |3x^2-1| dx\right)=\frac{1}{2}\Gamma(5/2)\cdot\frac{4}{3\sqrt{3}}=\frac{\pi^{1/2}}{2\sqrt{3}}.
\]
Hence
\[
\mu(\eY_n)\sim \frac{2}{\sqrt{3}} n^2\;\;\mbox{as}\;\;n\ra \infty.
\]
\end{proof}

\begin{remark}Let us observe that for $n$ very large, a typical spherical harmonic $\by\in\eY_n$ is a Morse function on $S^2$ and $0$ is a regular value. The  nodal set $\{\by=0\}$ is  disjoint union of smoothly embedded  circles. According to  the   classical theorem of  Courant \cite[\S VI.6]{CH},  the complement of  the nodal set has   at most $n^2$ connected components called \emph{nodal domains}.  We denote by $\eD_\by$ the collection of nodal domain,  and we set
\[
\delta(\by):=\# \eD_\by\leq n^2,\;\;\delta_n:=\pi^{-\frac{\dim\eY_n}{2}}\int_{\eY_n} e^{-|\by|^2} \delta(\by)\,|dV(\by)|=\frac{1}{{\rm area}\,(S(\eY_n))}\delta(\by)\,|dS(\by)|.
\]
In \cite{NS}, it is shown that there exists a positive  constant   $a>0$ such that
\[
\delta_n\sim an^2  \;\;\mbox{as}\;\;n\ra \infty.
\]
Additionally, for large $n$,      with high probability,  $\delta(\by)$ is close to  $an^2$ (see \cite{NS} for a precise statement).

Denote by $p(\by)$ the number of local minima and maxima of $\by$, and by $s(\by)$ the number of saddle points. Then
\[
\mu(\by)=p(\by)+s(\by),\;\;p(\by)-s(\by)=\chi(S^2)=2.
\]
This proves that
\[
p(\by)=\frac{1}{2}(\mu(\by)+2).
\]
For every    nodal region  $D$,  we denote by $p(\by, D)$ the number of local minima and maxima\footnote{A simple application of the  maximum principle shows that on  each  nodal domain, all the local extrema of $\by$ are of the same type: either all local minima or all local maxima.  Thus $p(\by, D)$ can be visualized as the number of  \textit{p}eaks of $|\by|$ on $D$.} of $\by$ on $D$.   Note that $p(\by, D)>0$ for any $D$ and thus the number  $p(\by)=\sum_{D\in \eD_\by} p(\by, D)$ can be viewed as a weighted  count of      nodal domains. We set
\[
p(\eY_n):=\pi^{-\frac{\dim\eY_n}{2}}\int_{\eY_n} e^{-|\by|^2}p(\by)\,|dV(\by)|.
\]
Theorem \ref{th: harm2} implies that
\[
p(\eY_n)\sim \frac{1}{\sqrt{3}}n^2\;\;\mbox{as}\;\;n\ra \infty.
\]
Since $\delta(\by)\leq p(\by)$,  this shows that $a\leq \frac{1}{\sqrt{3}}$. \qed
\label{rem: nodal-dom}
\end{remark}

\begin{remark}   We can use Remark \ref{rem: GB} as a simple test for the accuracy of the computations in Theorem \ref{th: harm2}. As explained in  Remark \ref{rem: GB}, the Euler characteristic  of $S^2$ is described by a integral  very similar to the one  describing $\mu(\eY_n)$. More precisely, we should have
\[
\pm 2=\pm \chi(S^2)= \underbrace{{\rm area}\,(S^2)\times \pi^{-\frac{3+2}{2}}\times \frac{1}{J(\eA^\dag_{p_0})}\times\int_{\bsL}e^{-|\by|^2}\det \Hess(\by)\,|dV(\by)|}_{=:\bsJ}.
\]
The  term $\bsJ$ can be computed as follows.
\[
\bsJ=\frac{16}{\pi^{1/2}n(n+1)(n+\frac{1}{2})}\int_{\bsL}e^{-|\by|^2}\det \Hess(\by)\,|dV(\by)|
\]
\[
=\frac{64}{\pi^{3/2}n(n+1)(n+\frac{1}{2})}\int_{\bR^3}e^{-(u^2+u_1^2+u_2^2)}\bigl(\,b(n)^2 u^2- a(n)^2(\,u_1^2+u_2^2\,)\,\bigr)\,|dudu_1du_2|
\]
\[
\stackrel{(\ref{eq: ab})}{=} \frac{1}{\pi^{3/2}}\int_{\bR^3} e^{-(u^2+u_1^2+u_2^2)}\bigl(\,2\beta(n) u^2- \alpha(n)(\,u_1^2+u_2^2\,)\,\bigr)\,|dudu_1du_2|,
\]
where
\[
\beta(n)=n(n+1),\;\;\alpha(n)=(n+2)(n-1).
\]
Arguing exactly as in the proof of Theorem \ref{th: harm2}, we deduce
\[
\int_{\bR^3} e^{-(u^2+u_1^2+u_2^2)}\bigl(\,2\beta(n) u^2- \alpha(n)(\,u_1^2+u_2\,)^2\,|dudu_1du_2|
\]
\[
=4\pi\int_0^\infty\int_0^\infty e^{-(u^2+r^2)}(2\beta(n)u^2-\alpha(n)r^2) r |dr du|
\]
($u=t\cos\vfi$, $r=t\sin\vfi$)
\[
=4\pi\int_0^\infty e^{-t^2}\int_0^{\pi/2}e^{-t^2}t^2\bigl( 2\beta(n)\cos^2\vfi-\alpha(n)\sin^2\vfi\,\bigr) \sin\vfi |dtd\vfi|
\]
($x=\cos\vfi$, $s=t^2$)
\[
=4\pi\times \frac{1}{2}\int_0^\infty e^{-s}s^{\frac{3}{2}} ds \times \int_0^1 \Bigl(\,\bigl(\,2\beta(n)+\alpha(n)\,\bigr) x^2 - \alpha(n)\,\Bigr)dx
\]
\[
=\frac{3\pi^{3/2}}{2} \times \frac{2}{3}\bigl(\, \beta(n)-\alpha(n)\,\bigr)=2\pi^{3/2}.
\]
This confirms the prediction in Remark \ref{rem: GB}, namely,  $J=\pm 2$.\qed
\label{rem: GB1}
\end{remark}

\begin{remark}  Most of the arguments in the proof of Theorem \ref{th: harm2} work  with minor changes  for    spherical harmonics of an arbitrary number of variables and lead to the conclusion
\[
\mu(\eY_{n,d})\sim \eZ_d n^{d-1}\;\;\mbox{as}\;\;n\ra \infty,
\]
but   the constant $\eZ_d$ is a bit more mysterious. Here are the details.

If
\[
\by=\sum_{j=0}^n\sum_{\beta\in \eB_{j,d}}  y_{j,\beta} Z_{n, j,\beta}\in\eY_{n,d},
\]
then
\[
\Hess(\by)= \sum_{\beta\in B_{2, d}}y_{2,\beta} \Hess(Z_{n, 2,\beta})+ y_{0,1}\Hess(Z_{n,0,1}).
\]
From (\ref{eq: zn0}), we deduce
\[
\Hess(Z_{n,0,1})=-\frac{1}{2}\bsi_{d-2}^{-1/2}C_{n,0,d}(n+1)\Bigl(\,n+\frac{d-3}{2}\,\Bigr)\one_{d-1}
\]
\[
= -\frac{1}{2}\bsi_{d-2}^{-1/2}\times \left(\frac{(2n+d-2)[n+d-3]_{d-3}}{2^{d-2}}\right)^{1/2}\times\frac{(n+1)(n+\frac{d-3}{2})}{\Gamma(\frac{d-1}{2})}\one_{d-1}
\]
\[
\sim\,\underbrace{-\frac{1}{2^{\frac{d-1}{2}}\bsi_{d-2}^{1/2}}}_{=: a(d)} \frac{n^{\frac{d+2}{2}}}{\Gamma(\frac{d-1}{2})}\one_{d-1}\;\;\mbox{as}\;\;n\ra \infty.
\]
As in (\ref{eq: zn}), we have
\[
Z_{n, 2,\beta}(\bx) := C_{n,j,d}P_{n,d}^{(2)}(x_d)  Y_{2,\beta}(\bx'),\;\;\bx=(\bx',x_d).
\]
If we denote by $H_\beta$ the Hessian  of $Y_\beta$ at $\bx'=0$, we deduce
\[
\Hess(Z_{n,2,\beta})=  C_{n,2,d} P_{n,d}^{(2)}(1) H_\beta.
\]
Using (\ref{tag: rbeta}), we deduce
\[
\Hess(Z_{n,2,\beta})\sim\underbrace{\frac{1}{2^{\frac{d+3}{2}}}}_{=:b(d)}\frac{n^{\frac{d+2}{2}}}{\Gamma(\frac{d-1}{2})} H_\beta\;\;\mbox{as}\;\;n\ra \infty.
\]
Arguing as in the proof of  (\ref{eq: adj-sph0})
 we deduce
\[
 \eA^\dag_{p_0}\pa_{x_i}=\pa_{x_i}Z_{n,1,i}(P_0) Z_{n,1,i},
 \]
 where
 \[
 Z_{n,1,i}=  C_{n,1,d} P'_{n,d}(x_d)\cdot C_dx_i,\;\; \int_{S^{d-2}} C_d^2 x_i^2\,|dS(\bx')|=1.
 \]
Hence
\[
\eA^\dag_{p_0}\pa_{x_i}= C_d C_{n,1,d} P'_{n,d}(1) Z_{n,1,i}.
\]
 We have
 \[
 C_{n,1,d} \sim \frac{1}{2^{\frac{d-3}{2}}\Gamma(\frac{d-1}{2})} n^{\frac{d-4}{2}}\;\;\mbox{as}\;\;n\ra \infty.
 \]
 Using (\ref{eq: leg3}) we  obtain
\[
P_{n,d}'(1)\sim \frac{1}{2}n^2\;\;\mbox{as}\;\;n\ra \infty.
\]
Using  (\ref{eq: bclass}) we deduce
\[
C_d^2=\frac{\Gamma(\frac{d+1}{2})}{\pi^{\frac{d-1}{2}}}.
\]
Hence,
\[
|\eA^\dag_{p_0}\pa_{x_i}|\sim \frac{C_d}{2^{\frac{d-1}{2}}\Gamma(\frac{d-1}{2})} n^{\frac{d}{2}}\;\;\mbox{as}\;\;n\ra \infty,
\]
so that
\[
J(\eA^\dag_{p_0})\sim \left(\frac{C_d}{2^{\frac{d-1}{2}}}\right)^{d-1} \frac{n^{\frac{d(d-1)}{2}}}{\Gamma(\frac{d-1}{2})^{(d-1)}}.
\]
Denote by $\bsL$ the subspace of $\eY_{n,d}$ spanned by the orthonormal collection of spherical harmonics
\[
\bigl\{\,Z_{n,0,1}, \;Z_{n,2,\beta}, \;\;\beta\in B_{2,d}\,\bigr\}.
\]
It has dimension
\[
\dim \bsL= N_{d-1}:=\dim \Sym(T_{p_0} S^{d-1})=\binom{d}{2}.
\]
Using (\ref{tag: bmup}), Corollary \ref{cor: av} and the above computations we deduce
\[
\frac{\mu(\eY_{n,d})}{n^{d-1}}\sim  \bsi_{d-1}\times\pi^{-\frac{N_{d-1}+d-1}{2}}\left(\frac{C_d}{2^{\frac{d-1}{2}}}\right)^{-(d-1)} \times \int_{\bsL} e^{-|\by|^2} |\det A_\infty(\by)|\,|dV(\by)|,
\]
where for
\[
\by = y_1Z_{N,0,1} +\sum_{\beta\in B_{2,d}}y_\beta Z_{n,2,\beta}\in\bsL,
\]
we have
\[
\bsA_\infty(\by)= y_1 a(d)\one_{d-1}+\sum_\beta y_\beta b(d) H_\beta.
\]
We interpret  $A_\infty$ as  an isometry  from $\bsL$ to  the space $\Sym_{d-1}=\Sym(T_{p_0}S^{d-1}, g)$  such that the collection
\[
a(d)\one_{d-1},\;\;b(d) H_\beta
\]
is an orthonormal basis of $\Sym_{d-1}$. We denote by $g_{a,b}$ this    $O(d-1)$-invariant metric on $\Sym_{d-1}$ and by $|d\eV_{a,b}|$ the associated volume  density. We deduce
\[
\int_{\bsL} e^{-|\by|^2} |\det A_\infty(\by)|\,|dV(\by)|=\int_{\Sym_{d-1}}e^{-|A|^2_{a,b}} |\det (A)|\,|d\eV_{a,b}(A)|.
\]
On  $\Sym_{k-1}$ we have a canonical $O(d-1)$ metric $|-|_*$ defined by
\[
|A|_*^2=\tr A^2.
\]
We denote by $|d\eV_*|$ the associated volume density.   From (\ref{eq: cgamma}), we deduce
\[
|d\eV_{a,b}|=\gamma_d|d\eV_*|,\;\;\gamma=\frac{1}{|a|(d-1)^{1/2}(bR)^{N_{d-1}-1}},\;\;R^2=\frac{4\Gamma(\frac{d+3}{2})}{\pi^{\frac{d-1}{2}}}.
\]
From (\ref{eq: cab}), we  deduce
\[
|A|^2_{a,b}= \alpha \tr A^2  +\beta (\tr A)^2,
\]
where
\[
\alpha=\frac{1}{b^2R^2}=\frac{2^{d+1}\pi^{\frac{d-1}{2}}}{\Gamma(\frac{d+3}{2})},
\]
\[
\beta=\frac{1}{d-1}\left(\frac{1}{(d-1)a^2}-\frac{1}{b^2R^2}\right) = \frac{1}{(d-1)}\left(\frac{2^{d-1}\bsi_{d-2}}{(d-1)}-\frac{2^{d+1}\pi^{\frac{d-1}{2}}}{\Gamma(\frac{d+3}{2})}\right)
\]
\[
=\frac{2^d\pi^{ \frac{d-1}{2} } }{(d-1)} \left(\frac{1}{\Gamma(\frac{d-1}{2})(d-1)}-\frac{2}{\Gamma(\frac{d+3}{2})}\right)=\frac{2^d\pi^{ \frac{d-1}{2} } }{(d-1)\Gamma(\frac{d+1}{2})}\left(\frac{1}{2}-\frac{4}{(d+1)}\right)
\]
\[
=\frac{2^{d-1}\pi^{ \frac{d-1}{2} }(d-7) }{(d^2-1)\Gamma(\frac{d+1}{2})}=\frac{d-7}{8(d-1)}\alpha.
\]
We deduce
\[
\int_{\Sym_{d-1}}e^{-|A|^2_{a,b}} |\det (A)|\,|d\eV_{a,b}(A)|=\gamma_d\int_{\Sym_{d-1}}  e^{-\alpha \tr A^2 -\beta(\tr A)^2} |\det A|\,|d\eV_*(A)|.
\]
As explained in Appendix  \ref{s: c}, the last integral can be further  simplified to
\[
\int_{\Sym_{d-1}}  e^{-\alpha \tr A^2 -\beta(\tr A)^2} |\det A|\,|d\eV_*(A)|
\]
\[
=\eZ_d\underbrace{\int_{\bR^{d-1}} e^{-\frac{|\bx|^2}{2}-\frac{\beta}{2\alpha}(\tr\bx)^2} \prod_{i=1}^{d-1}|x_i|\cdot \prod_{1\leq i<j\leq d-1}|x_i-x_j|\,|dV(\bx)}_{=:I_d},
\]
where $\tr(\bx):=x_1+\cdots +x_{d-1}$ , and $\eZ_d$ is a positive constant  that can be   determined  explicitly.   The integral $I_d$ seems  difficult to evaluate. The trick used in \cite{Fy} does not work  when $d>7$,  since in that case $\beta>0$. The asymptotics of $I_d$ as $d\ra \infty$  are   very intriguing. \qed

\end{remark}

\section{Random trigonometric polynomials with given Newton polyhedron}
\label{s: torus}
\setcounter{equation}{0}

Fix a positive integer $L$  and denote by $\bT^L$ the $L$-dimensional torus $\bT^L: =\bR^L/(2\pi\bZ)^L$ equipped with the induced  flat metric. Let $\vec{\theta}=(\theta_1,\dotsc, \theta_L)$ denote the angular coordinates induced from the  canonical Euclidean coordinates on $\bR^L$. For any $\vec{m}\in\bZ^L$ we set
\[
p(\vec{m})= \begin{cases}
\frac{2^{1/2}}{(2\pi)^{L/2}}, &|\vec{m}|\neq 0\\
&\\
\frac{1}{(2\pi)^{L/2}},& |\vec{m}|=0,
\end{cases}
\]
\[
\ha_{\vec{m}}:=p(\vec{m})\cos\left(\sum_{j=1}^L m_j\theta_j\right),\;\;  \hb_{\vec{m}}:=p(\vec{m})\sin\left(\sum_{j=1}^L m_j\theta_j\right).
\]
The lattice  $\bZ^L$ is  equipped  with the  lexicographic order $\prec$, where we define $\vec{m}\prec \vec{n}$ if the first non zero element in the sequence $n_1-m_1,\dotsc, n_L-m_L$ is positive.   We define $\eC_L$ to be the positive cone 
\[
\eC_L:=\bigl\{ \vec{m}\in  \bZ^L;\;\;\vec{0}\prec \vec{m}\,\bigr\}.
\]
The collection
\[
\bigl\{\ha_{\vec{0}}\,\bigr\}\cup \bigl\{\ha_{\vec{m}};\;\;\vec{m}\in \eC_L\,\bigr\}\cup\bigl\{ \hb_{\vec{m}};  \vec{m}\in\eC_L\,\bigr\}
\]
is an orthonormal basis of $L^2(\bT^L)$.      A  finite set  $\eM\subset\eC_L$ is  called \emph{symmetric} if for any permutation $\vfi$ of $\{1,\dotsc, L\}$ we have
\[
(m_1,\dotsc,m_L)\in\eM\Llra (m_{\vfi(1)},\dotsc,m_{\vfi(L)})\in\eM\cup -\eM.
\]
For example, the set $\{ (2,-1), (1,-1), (1,-2)\}\subset \eC_2$ is symmetric.

For any finite  set  $\eM\subset \eC_L$  we define 
\[
\bsV(\eM):=\spa\bigl\{\ha_{\vec{m}},\hb_{\vec{n}};\;\; \vec{m},\;\;\vec{n}\in\eM\,\bigr\},
\]
 the scalars
\begin{equation}
a_{jk}=a_{jk}(\eM)=\frac{2}{(2\pi)^{L}}\sum_{\vec{m}\in\eM} m_jm_k,
\label{eq: alp-bet}
\end{equation}
and the vectors
\begin{equation}
\vec{a}_{jk}=\vec{a}_{jk}(\eM)=-\frac{2^{\frac{1}{2}}}{(2\pi)^{\frac{L}{2}}}\sum_{\vec{m}\in\eM} m_jm_k\ha_{\vec{m}}\in\bsV(\eM).
\label{eq: alp-bet-vec}
\end{equation}
If $\eM$ is  symmetric then the scalars  $a_{jj}$ are independent of $j$ and we denote  their common value by $\alpha(\eM)$.  Similarly,  the  scalars $a_{jk}$, $j\neq k$, are independent of $j\neq k$, and we denote their common value by $b(\eM)$.

\begin{theorem} Suppose  $\eM\subset \eC_L$ is a symmetric    finite subset of cardinality $N>L$. We set
\[
a:=a(\eM),\;\;b:=\beta(\eM),\;\;\vec{a}_{ij}=\vec{a}_{ij}(\eM). 
\]
Then the following hold.

\smallskip

\noindent (a) The sample space $\bsV(\eM)$ is ample if and only if $a\neq b$.

\noindent (b)  Suppose that 
\begin{equation*}
\mbox{\emph{the vectors $\vec{a}_{ij}$  are linearly independent.}}
\tag{$\#$}
\label{tag: sharp}
\end{equation*}
   Denote by $\Sym_L$ the \emph{Euclidean space}  of symmetric $L\times L$ matrices with \emph{orthonormal basis} $(H_{ij})_{1\leq i\leq j\leq L}$, where $H_{jk}$  is the symmetric $L\times L$ matrix  with  nonzero entries  only in  at locations $(j,k)$ and $(k,j)$, and those entries are $1$.   We denote by  $\lan-,-\ran$ the resulting inner product on $\Sym_L$. Then
\begin{equation}
\mu(M)= \frac{1}{(2\pi)^{\frac{1}{2}\binom{L}{2}}(a-b)^{\frac{L-1}{2}}\bigl(\,a+(L-1)b\,\bigr)^{1/2}}\int_{\Sym_L} e^{-\frac{1}{2}\lan \bsC^{-1} X, X\ran} |\det X|\, |dX|,
\label{eq: expect-trig1}
\end{equation}
where for $X=\sum_{i\leq j}x_{ij} H_{ij}$
\[
|dX|=\left|\prod_{i\leq j} dx_{ij}\right|,
\]
and $\bsC:\Sym_L\ra \Sym_L$ is the symmetric linear operator  described in the orthonormal basis $(H_{ij})$ by the matrix
\[
\bsC_{ij;k\ell}=  (\vec{a}_{ij},\vec{a}_{k\ell})=\frac{2}{(2\pi)^L}\sum_{\vec{m}\in\eM}m_im_jm_km_\ell.
\]
\label{th: trig-newton}
\end{theorem}

\begin{proof}  We will compute $\mu(\eM)$  via the identity (\ref{eq: av-hom}). Observe first that $\bsV(\eM)$ is invariant under the action of $\bT^L$ on itself, and  the induced   action on $\bsV(\eM)$ is by isometries.  Let $p=(0,\dotsc, 0)\in \bT^d$, and denote by $\bsK_p$ the  subspace of $\bsV(\eM)$ consisting of trigonometric polynomials that admit $p$ as a critical point.  Set $\pa_j:=\pa_{\theta_j}$,  $\bsf_j:=\pa_j|_p$, $j=1,\dotsc,d$. We have
\[
\eA^\dag_p\bsf_j=\sum_{\vec{m}\in \eM} \left(\pa_j\ha_{\vec{m}}(\vec{\theta})\ha_{\vec{m}}+\pa_j\hb_{\vec{m}}(\vec{\theta})\hb_{\vec{m}}\right)|_{\vec{\theta}=\vec{0}}=\sum_{\vec{m}\in \eM}m_jp(\vec{m})\hb_{\vec{m}}.
\]
We have
\[
G_{jk}:=\bigl(\,  \eA^\dag_p\bsf_j, \eA^\dag_p\bsf_k)=\frac{2}{(2\pi)^L} \sum_{\vec{m}\in\eM}m_jm_k=a_{jk}(\eM)
\]
Since $\eM$ is symmetric we deduce  that   $\eA_p\eA^\dag_p$ is described by the symmetric  $L\times L$  matrix  $G_L(a,b)$ whose diagonal entries are all equal to $a$, and all the off-diagonal entries are equal to $b$. We denote  by $\Delta_L(a,b)$ its determinant.    We deduce\footnote{ If $C_L$ denotes the $L\times L$ matrix  with all entries $1$, then $G_L(a,b)= (a-b)\one_L+bC_L$. The matrix $C_L$ has rank $1$ and a single nonzero eigenvalue equal to $L$. This implies (\ref{eq: detab}).}
\begin{equation}
\Delta_L(a,b)=(a-b)^{L-1}\bigl( a+(L-1)b\,\bigr),
\label{eq: detab}
\end{equation}  
so that the Jacobian of $\eA^\dag_p$ is
\begin{equation}
J(\eA^\dag_p)=\Delta_d(a,b)^{1/2}.
\label{eq: jac-torus}
\end{equation}
Observe  that $\bsV(\eM)$ is ample if and only if  the  Jacobian of the adjunction map $\eA^\dag_p$ is nonzero, i.e., if and only if  $\Delta_L(a,b)\neq 0$. This proves part  (a).

If 
\[
\bv=\sum_{\vec{m}\in\eM} \bigl(a_{\vec{m}}\ha_{\vec{m}}+b_{\vec{m}}\hb_{\vec{m}}\,\bigr),
\]
then
\[
\pa_j\pa_k\bv(p)= -\frac{2^{\frac{1}{2}}}{(2\pi)^{\frac{L}{2}}}\sum_{\vec{m}\in\eM}a_{\vec{m}}m_jm_k.
\]
We deduce that
\begin{equation}
\Hess_p(\bv)=-\frac{2^{\frac{1}{2}}}{(2\pi)^{\frac{L}{2}}}\sum_{j\leq k} \left(\,\sum_{\vec{m}\in\eM}a_{\vec{m}}m_jm_k\right)H_{jk}
\label{eq: hess-tor}
\end{equation}
Using the notations (\ref{eq: alp-bet}) and (\ref{eq: alp-bet-vec})  we can rewrite the equality (\ref{eq: hess-tor}) as 
\[
\Hess_p(\bv)= \sum_{i\leq j}(\bv,\vec{a}_{ij}) H_{ij}.
\]
 Using      Lemma \ref{lemma: gauss}  and the equality $\dim\bsV(\eM)=2N$ we deduce that
\[
\int_{S(\bsK_p)}|\det\Hess(\bv)|\,|dS(\bv)|= \frac{2}{\Gamma(N)}\int_{\bsK_p} e^{-|\bv|^2} |\det \Hess(\bv)|\,|dV(\bv)|
\]
\[
=\frac{2^{-\frac{\dim\bsK_p+L}{2}+1}}{\Gamma(N)} \int_{\bsK_p} e^{-\frac{|\bv|^2}{2}} |\det \Hess(\bv)|\,|dV(\bv)|
\]
\[
=\frac{2^{-\frac{L}{2}+1}\pi^{\frac{\dim \bsK_p}{2}}}{\Gamma(N)} \underbrace{\int_{\bsK_p} |\det \Hess(\bv)|\,\frac{e^{-\frac{|\bv|^2}{2}}}{(2\pi)^{\frac{\dim\bsK_p}{2}}} |dV(\bv)|}_{=: \bsI(\eM)}.
\]
We performed all this yoga to observe that   $\bsI(\eM)$ is   an integral with respect to a Gaussian  density over $\bsK_p$.   Denote by  $\eT=\eT_{\eM}$ the linear map
\[
\eT:\bsK_p\ra \Sym_L,\;\;\bv\mapsto \Hess_p(\bv)= \sum_{i\leq j}(\bv,\vec{a}_{ij}) H_{ij}.
\]
Since the  vectors $\vec{a}_{ij}$ are assumed to be linearly independent, the map $\eT$  is surjective. Clearly, $\det \Hess(\bv)$ is constant along  the  fibers of $\eT$.    As is  well known (see e.g. \cite[\S 16]{Shi}) the pushfowrad of a Gaussian  measure via a surjective linear map is also a Gaussian measure.    Thus the density
\[
|d\bgamma_{\eM}|:=\eT_*\left(\,\frac{e^{-\frac{|\bv|^2}{2}}}{(2\pi)^{\frac{\dim\bsK_p}{2}}} |dV(\bv)|\,\right),
\]
is a Gaussian density on $\Sym_L$. Since   the density  $\frac{e^{-\frac{|\bv|^2}{2}}}{(2\pi)^{\frac{\dim\bsK_p}{2}}} |dV(\bv)|$  is centered, i.e., its expectation is trivial,  we deduce that its   pushforward by $\eT$ is also  centered. The Gaussian density  $|d\gamma_{\eM}|$ is thus determined by its covariance operator 
\[
\bsC=\bsC_{\eM}:\Sym_L\ra \Sym_L
\]
described in the orthonormal basis $(H_{ij})$ by the matrix
\[
\bsC_{ij;k\ell}=  (\vec{a}_{ij},\vec{a}_{k\ell})=\frac{2}{(2\pi)^L}\sum_{\vec{m}\in\eM}m_im_jm_km_\ell.
\]
The symmetry condition on $\eM$ imposes many relations between these numbers.  We deduce
\[
\bsI(\eM) =\frac{1}{(2\pi)^{\frac{s(L)}{2}}(\det\bsC)^{1/2}} \int_{\Sym_L} e^{-\frac{1}{2}\lan \bsC^{-1} X, X\ran} |\det X|\, |dX|,\;\;s(L):=\dim\Sym_L.
\]
Using  (\ref{eq: av-hom}) we deduce
\[
\mu(\eM)=\frac{{\rm vol}\,(\bT^L)}{\bsi_{2N-1}\Delta_L(a,b)^{1/2}}\times \frac{2^{-\frac{L}{2}+1}\pi^{\frac{\dim \bsK_p}{2}}}{\Gamma(N)}  \bsI(\eM)
\]
\[
=\frac{{\rm vol}\,(\bT^L)}{\Delta_L(a,b)^{1/2}}\times \frac{2^{-\frac{L}{2}}\pi^{\frac{\dim \bsK_p}{2}}}{\pi^N}  \bsI(\eM)=\frac{(2\pi)^{\frac{L}{2}}}{\Delta_L(a,b)^{1/2}}\bsI(\eM).
\]
This proves (\ref{eq: expect-trig1}).
\end{proof}

We will put the above theorem to work in several special cases.  Let us observe  that  the assumption (\ref{tag: sharp}) is automatically satisfied if $\eM$ contains the points
\[
(1,0,\dotsc, 0),\;\; (1,1,0,\dotsc,0).
\]
Indeed, the symmetry of $\eM$ implies that all the  functions $\cos(\theta_i)$ and $\cos(\theta_i+\theta_j)$, $1\leq i\neq j\leq L$,   belong to $\bsV(\eM)$ and the hessians of these functions span the whole space of $L\times L$ matrices.

Suppose now that   $\eM=\eM_\nu^L :=\Lambda^L_\nu\cap\eC$, where $\nu$ is a (large)  positive integer, and $\Lambda_\nu$ is the cube
\[
\Lambda^L_\nu:= \bigl\{ \vec{m}\in\bZ^L;\;\; |m_j|\leq \nu,\;\;\forall j=1,\dotsc, L\,\bigr\}.
\]
Let us observe that
\begin{equation}
\Lambda^L_\nu= \eM_\nu^L\cup(-\eM_\nu^L)\cup\{\vec{0}\}.
\label{eq: symmetry}
\end{equation}
Among other things, this proves that $\eM_\nu$ is symmetric.   We want to investigate the behavior of $\mu(\eM_\nu^L)$ as $\nu\ra \infty$. To formulate our next  result we need to introduce additional  notation.

Let us observe that we have an orthogonal decomposition
\begin{equation}
\Sym_L=\eD_L\oplus \eD_L^\perp,
\label{eq: sym-decomp}
\end{equation}
where $\eD_L$ consits of diagonal matrices
\[
\eD_L=\spa\bigl\{H_{ii};\;\;1\leq i\leq L\}
\] and
\[
\eD_L^\perp = \spa\bigl\{ H_{ij};\;\;1\leq i<j\leq L\,\bigr\}.
\]
For any   real numbers $a,b$ we denote by $G_L(a,b)$ the $L\times L$-matrix with entries
\[
g_{ij}=\begin{cases}a,& i=j\\
b, &i\neq j.
\end{cases}.
\]

\begin{theorem} Let
\[
\eM_\nu^L:=\bigl\{ \vec{m}\in\eC_L;\;\;|m_j|\leq \nu,\;\;\forall 1\leq j\leq L\,\bigr\}.
\]
Then, as $\nu\ra \infty$ we  have
\begin{equation}
\mu(\eM_\nu^L)\sim \left(\frac{\pi}{6}\right)^{\frac{L}{2}} \bigl\lan\, |\det X|\,\bigr\ran_{\overline{\bsC}_\infty} \dim \bsV(\eM_\nu^L),
\label{eq: tor-symp10}
\end{equation}
where   $  \bigl\lan\, |\det X|\,\bigr\ran_{\overline{\bsC}_\infty} $ the  expectation of $|\det X|$  with respect to the centered gaussian probability  measure on $\Sym_L$ with  covariance matrix that has the block  description
\[
\overline{\bsC}_\infty= G_L\left(\frac{9}{5},1\right)\oplus \one_{\binom{L}{2}}
\]
with respect to the decomposition (\ref{eq: sym-decomp}).
\label{th: tor-complexity}
\end{theorem}

\begin{proof}Let us first compute
\[
a(\nu)=a(\eM_\nu^L)=\frac{2}{(2\pi)^L}\sum_{\vec{m}\in\eM_\nu} m_1^2 \stackrel{(\ref{eq: symmetry})}{= }\frac{1}{(2\pi)^L} \sum_{\vec{m}\in\Lambda^L_\nu} m_1^2=\frac{|\Lambda^{L-1}_\nu|}{(2\pi)^L}\sum_{|m_1|\leq \nu}  m_1^2 |\Lambda^{L-1}_\nu|
\]
\[
=\frac{2(2\nu+1)^{L-1}}{(2\pi)^L}\sum_{k=1}^\nu k^2=\frac{2(2\nu+1)^{L-1}}{3(2\pi)^L}B_3(\nu+1)\sim  \frac{1}{3\pi^L}\nu^{L+2}\;\;\mbox{as $\nu\ra \infty$}.
\]
Similarly, we have
\[
b(\nu)=b(\eM_\nu^L)=\frac{1}{(2\pi)^L}\sum_{\vec{m}\in\Lambda^L_\nu} m_1m_2.
\]
The last sum is  $0$ due to the invariance of $\Lambda^L_\nu$ with respect to the reflection
\[
(m_1,m_2,\dotsc, m_L) \longleftrightarrow (-m_1,m_2,\dotsc, m_L).
\]
Thus, in this case
\begin{equation}
\Delta_L(a,b)= a(\nu)^L\sim \frac{1}{3^L\pi^{L^2}}\nu^{L(L+2)} \;\;\mbox{as $\nu\ra \infty$}.
\label{eq: jac-torus2}
\end{equation}
To compute the covariance operator $\bsC$ we observe first that, in view of the symmetry of $\eM_\nu$ it suffices to compute only the entries
\[
\bsC_{11;ij}, \;\;i\leq j\;\;\mbox{and}\;\; \bsC_{12; ij},\;\;i<j.
\]
We have
\[
\bsC_{11;11}=\frac{1}{(2\pi)^L}\sum_{\vec{m}\in\Lambda^L_\nu} m_1^4=\frac{2|\Lambda_\nu^{L-1}|} {(2\pi)^L}\sum_{k=1}^\nu k^4=\frac{2(2\nu+1)^{L-1}}{5(2\pi)^L}B_5(\nu+1)\sim \frac{1}{5\pi^L}\nu^{L+4}.
\]
For  $i>1$ we have
\[
\bsC_{11;ii}=\frac{1}{(2\pi)^L}\sum_{\vec{m}\in\Lambda^L_\nu} m_1^2m_i^2=\frac{1}{(2\pi)^L}\sum_{\vec{m}\in\Lambda^L_\nu} m_1^2m_2^2=\frac{|\Lambda^{L-2}_\nu|}{(2\pi)^L}\sum_{\vec{m}\in\Lambda^2_\nu}m_1^2m_2^2
\]
\[
= \frac{|\Lambda^{L-2}_\nu|}{(2\pi)^L}\left(\sum_{|k|\leq \nu}k^2\,\right)^2=\frac{4(2\nu+1)^{L-2}}{(2\pi)^L}\left(\sum_{k=1}^\nu k^2\right)^2=\frac{4(2\nu+1)^{L-2}}{9(2\pi)^L}B_3(\nu+1)\sim \frac{1}{9\pi^L}\nu^{L+4}.
\]
Using   the invariance of $\Lambda_\nu$ with respect to the reflections
\begin{equation}
(m_1,\dotsc, m_i,\dotsc, m_L)\longleftrightarrow (m_1,\dotsc, -m_i,\dotsc, m_L)
\label{eq: refl}
\end{equation}
we  deduce that for any $i<j$ we have
\[
\bsC_{11,ij}=0.
\]
To summarize,  we have shown  that
\begin{subequations}
\begin{equation}
x_\nu= \bsC_{ii;ii}=\bsC_{11;11}\sim \frac{1}{5\pi^L}\nu^{L+4}
\label{eq: xnutor}
\end{equation}
\begin{equation}
y_\nu=\bsC_{ii;jj}=\bsC_{11,jj}\sim \frac{1}{9\pi^L}\nu^{L+4}\;\;\forall 1\leq i <j.
\label{eq: ynutor}
\end{equation}
\begin{equation}
\bsC_{ii;jk}=0,\;\;\forall i,\;\;j<k.
\label{eq: zero-corel1}
\end{equation}
\end{subequations}
Next, we observe that
\[
\bsC_{12;12} =\frac{1}{(2\pi)^L}\sum_{\vec{m}\in\Lambda^L_\nu}m_1^2m_2^2=\frac{4(2\nu+1)^{L-2}}{9(2\pi)^L}B_3(\nu+1)=y_\nu\sim \frac{1}{9\pi^L}\nu^{L+4}.
\]
Using the reflections (\ref{eq: refl}) we deduce that
\[
\bsC_{12;ij}=0,\;\;\forall i<j,\;\;(i,j)\neq (1,2).
\]
With respect to the  decomposition (\ref{eq: sym-decomp}) the  covariance operator has a bloc decomposition
\[
\bsC=\left[
\begin{array}{cc}
\eG & \eF\\
\eF^\dag &\eH
\end{array}
\right],
\]
where $\eF: \eD_L^\perp\ra \eD_L$.  The above computations show that
\[
\eF=0,\;\;\eH = y_\nu\one_{\eD_L^\perp}=y_\nu\one_{\binom{L}{2}}.
\]
The  operator $\eG$ is  described in the basis $(H_{ii})$ of $\eD_L$ by the matrix $G_L(x_\nu,y_\nu)$ .  We deduce that
\[
\bsC =G_L(x_\nu, y_\nu)\oplus y_\nu\one_{\binom{L}{2}}= y_\nu \times\underbrace{\bigl(\, G_L(z_\nu,1)\oplus \one_{\binom{L}{2}}\,\bigr)}_{=:\overline{\bsC}_\nu},\;\; z_\nu=\frac{x_\nu}{y_\nu}.
\]
Using (\ref{eq: xnutor}) and (\ref{eq: ynutor}) we deduce that
\[
\lim_{\nu\ra\infty}z_\nu=\frac{9}{5}.
\]
We conclude that
\begin{equation}
\det \bsC\sim y_\nu^{\binom{L}{2}+L}\det G_L\left(\frac{9}{5},1\right)\sim \left(\frac{4}{5}\right)^{L-1}\left(\frac{4}{5}+L\right)y_\nu^{\binom{L}{2}+L},\;\;\mbox{as $\nu\ra \infty$}.
\label{eq: det-cov}
\end{equation}
 Using (\ref{eq: expect-trig1}), (\ref{eq: jac-torus2}) we deduce
\[
\mu(\eM_\nu^L)= \frac{ 3^{\frac{L}{2} }\pi^{\frac{L^2}{2} } }{ (2\pi)^{\frac{1}{2}\binom{L}{2} }\nu^{\frac{L(L+2)}{2} }y_\nu^{\frac{1}{2}\binom{L}{2}+\frac{L}{2} }(\det\overline{\bsC}_\nu)^{1/2} }\int_{\Sym_L} e^{-\frac{1}{2y_\nu}\lan \overline{\bsC}_\nu^{-1} X, X\ran} |\det X|\, |dX|,
\]
making the change in variables  $X=y_\nu^{1/2}Y$ we deduce
\[
\mu(\eM_\nu^L)= \frac{ 3^{ \frac{L}{2} }\pi^{ \frac{L^2}{2} } y_\nu^{ \frac{1}{2}(\dim \Sym_L +L) } }{ (2\pi)^{ \frac{1}{2}\binom{L}{2} }\nu^{ \frac{L(L+2)}{2} }y_\nu^{\frac{1}{2}\binom{L}{2}+\frac{L}{2} }(\det\overline{\bsC}_\nu)^{1/2} }\int_{\Sym_L} e^{-\frac{1}{2}\lan \overline{\bsC}_\nu^{-1} Y, Y\ran} |\det Y|\, |dY|
\]
\[
= \frac{ 3^{ \frac{L}{2} }\pi^{ \frac{L^2}{2} } y_\nu^{ \frac{L}{2} } }{ (2\pi)^{ \frac{1}{2}\binom{L}{2} }\nu^{ \frac{L(L+2)}{2} } (\det\overline{\bsC}_\nu)^{1/2} }\int_{\Sym_L} e^{-\frac{1}{2}\lan \overline{\bsC}_\nu^{-1} Y, Y\ran} |\det Y|\, |dY|.
\]
As $\nu\ra \infty$ we have
\[
\overline{\bsC}_\nu\ra \overline{\bsC}_\infty:=G_L\left(\frac{9}{5},1\right)\oplus \one_{\binom{L}{2}}.
\]
Using (\ref{eq: ynutor}) we deduce that as $\nu\ra \infty$ we have
\[
\mu(\eM_\nu)\sim Z_L\nu^L,\;\;Z_L=\frac{1}{3^{\frac{L}{2}} (2\pi)^{ \frac{1}{2}\binom{L}{2} } (\det\overline{\bsC}_\infty)^{1/2}} \int_{\Sym_L} e^{-\frac{1}{2}\lan \overline{\bsC}_\infty^{-1} Y, Y\ran} |\det Y|\, |dY|.
\]
Since
\[
\dim\bsV(\eM_\nu^L)\sim (2\nu)^L\;\;\mbox{as $\nu\ra \infty$},
\]
we deduce
\[
\mu(\eM_\nu^L)\sim \left(\frac{\pi}{6}\right)^{\frac{L}{2}}\times \frac{1}{ (2\pi)^{\frac{\dim\Sym_L}{2}} (\det \overline{\bsC}_\infty)^{\frac{1}{2}}} \int_{\Sym_L} e^{-\frac{1}{2}\lan \overline{\bsC}_\infty^{-1} Y, Y\ran} |\det Y|\, |dY|.
\]
This proves (\ref{eq: tor-symp10}). \end{proof}

Let us apply the above result in the case $L=1$. In this case $\eM_\nu^1$ consists of trigonometric polynomials of degree $\leq  \nu$ on $S^1$, and $\Sym_L=\bR$. In this case  we have
\[
 \bigl\lan\, |\det X|\,\bigr\ran_{\overline{\bsC}_\infty}=\frac{\sqrt{5}}{3}\times\frac{1}{\sqrt{2\pi}}\int_\bR e^{-\frac{5x^2}{18}} |x| dx= \frac{2\sqrt{5}}{3} \times\frac{1}{\sqrt{2\pi}}\,\underbrace{\int_0^\infty e^{-\frac{5x^2}{18}} x dx}_{=\frac{9}{5}}=\left(\frac{6}{\pi}\right)^{\frac{1}{2}}\sqrt{\frac{3}{5}}.
\]
We deduce the following result.
\begin{corollary}
\begin{equation}
\mu(\eM_\nu^1)\sim 2\sqrt{\frac{3}{5}}\nu,\;\;\mbox{as $\nu\ra \infty$}.
\label{eq: trig-poly1}
\end{equation}
\label{cor: trig-poly1}
\end{corollary}
When $L=2$, the computations are a bit more complicated, but we can still be quite explicit.

\begin{corollary}
\begin{equation}
\mu(\eM_\nu^2)\sim Z_2\dim\bsV(\eM_\nu^2),\;\;Z_2\approx 0.4717,\;\;\mbox{as $\nu\ra \infty$}.
\label{eq: trig-poly2}
\end{equation}
\label{cor: trig-poly2}
\end{corollary}
\begin{proof}  We decompose the operators $X\in\Sym_2$ as
\[
X=xH_{11}+ yH_{22}+z H_{12}
\]
so that $\det X=(xy-z^2)$. We write $a:= \frac{9}{5}$, $b:=1$. Then
\[
\frac{1}{2}\lan\overline{\bsC}_\infty^{-1} X,X\ran=\frac{1}{2(a^2-b^2)}(ax^2+ay^2-2bxy)-\frac{1}{2}z^2,\;\;\det\bsC_\infty= (a^2-b^2)
\]
\[
\lan |\det X|\ran_{\overline{\bsC}_\infty}= \frac{1}{(2\pi)^{3/2}(a^2-b^2)^{1/2}}\underbrace{\int_{\bR^3} e^{- \frac{1}{2(a^2-b^2)}(ax^2+ay^2-2bxy)-\frac{1}{2}z^2} |xy-z^2||dxdydz|}_{=:\bsI(a,b)}.
\]
As shown  in Proposition \ref{prop: iab}, the integral  $\bsI(a,b)$ can be reduced to a $1$-dimensional integral
\[
\bsI(a,b)=\sqrt{2\pi(a^2-b^2)}\left(\int_0^{2\pi}\frac{2c^{3/2}}{(c+2)^{1/2}}d\theta-2\pi a+2\pi\right),
\]
where
\[
c(\theta):= (a-b\cos 2\theta).
\]
We deduce
\[
\lan |\det X|\ran_{\overline{\bsC}_\infty}= \frac{1}{2\pi}\int_0^{2\pi}\frac{2c^{3/2}}{(c+2)^{1/2}}d\theta- a+1\approx1.7207...
\]
and
\[
\frac{\pi}{6}\times \lan |\det X|\ran_{\overline{\bsC}_\infty}\approx 0.4717...
\]
\end{proof}

\begin{remark}The antiderivative of   $\frac{c^{3/2}}{(c+2)^{1/2}}$ can be expressed  in   a rather complicated fashion in terms of elliptic integrals.\qed
\end{remark}

Still in the case $L=2$,  suppose that
\begin{equation}
\eM=\bigl\{ (1,0), (0,1), (1,1)\,\bigr\}.
\label{eq: em}
\end{equation}
The space $\bsV(\eM)$ was investigated in great detail by V.I. Arnold, \cite{Ar07a, Ar07, Ar06}.

\begin{theorem} If $\eM$ is given by (\ref{eq: em}), then
\[
\mu(\eM)=\frac{4\pi}{3}\approx 4.188.
\]
\label{th: arn}
\end{theorem}

\begin{proof} We  rely on Theorem \ref{th: trig-newton}, or rather its proof. In this case $L=2$, $\dim\bsV(\eM)=6$.  The collection  $\bigl\{ \ha_{1,0},\ha_{1,1},\ha_{0,1}\,\bigr\}$ is an orthonormal  system, and we  denote by $\bsL$ the vector space they span. Note that $\bsL\subset\bsK_p$, and $\Hess(\bv)=0$ if $\bv\in \bsL^\perp\cap\bsK_p$. We have
\[
a=\frac{1}{2\pi^2}\sum_{\vec{m}\in\eM} m_1^2= \frac{1}{\pi^2},\;\;b=\frac{1}{2\pi^2}\sum_{\vec{m}\in\eM}m_1m_2=\frac{1}{2\pi^2}.
\]
Then
\[
a-b= \frac{1}{2\pi^2},\;\; a+(L-1)b= \frac{3}{2\pi^2},\;\; J(\eA_p^\dag)=(a-b)^{\frac{L-1}{2}}\bigl( a+(L-1)b\,\bigr)^{\frac{1}{2}}= \frac{\sqrt{3}}{2\pi^2}.
\]
We decompose $\bv\in\bsL$ as
\[
\bv=x\ha_{1,0}+y\ha_{0,1}+z\ha_{1,1}.
\]
and we have
\[
\Hess_p(\bv)=-\frac{2^{1/2}}{2\pi}\left[
\begin{array}{cc}
x+z & z\\
z & y+z
\end{array}
\right],
\]
\[
 |\det\Hess_p(\bv)|=\frac{1}{2\pi^2}|xy+yz+zx|,\;\;|\bv|^2=x^2+y^2+z^2.
\]
Using (\ref{tag: bmup}) we deduce
\[
\begin{split}
\mu(\eM) &= \frac{ {\rm vol}\,(S^1\times S^1)}{\pi^{5/2}J(\eA_p^\dag)}\times \frac{1}{2\pi^2}\int_{\bR^3} e^{-(x^2+y^2+z^2)} |xy+yz+zx|\,|dS(x,y,z))|\\
&=\frac{4}{\pi^{1/2}\sqrt{3}}\int_{\bR^3} e^{-(x^2+y^2+z^2)}|xy+yz+zx| |dxdydz|.
\end{split}
\]
The  quadratic form $Q(x,y,z)=xy+yz+zx$ can be diagonalized via an orthogonal change of coordinates.  The matrix   describing $Q$ in the orthonormal coordinates $x,y,z$ is  the symmetric matrix
\[
\frac{1}{2}\left[
\begin{array}{ccc}
0 & 1 & 1\\
1 & 0 & 1\\
1 & 1 & 0
\end{array}
\right].
\]
and its eigenvalues  are $1,-\frac{1}{2},-\frac{1}{2}$. Thus, for  some  Euclidean coordinates  $u,v,w$, we have
\[
Q= \frac{1}{2}(2u^2-v^2-w^2),
\]
and therefore,
\[
\mu(\eM)=\frac{2}{\pi^{1/2}\sqrt{3}}\underbrace{\int_{\bR^3} e^{-(u^2+v^2+w^2)} |2u^2-v^2-w^2| \,|dudvdw|}_{=:I}.
\]
The above integral can be computed using cylindrical coordinates $(u, r, \theta)$,
\[
r=(v^2+w^2)^{1/2},\;\;  v=r\cos\theta,\;\;w=r\sin \theta.
\]
We deduce
\[
I= \int_0^{2\pi}\int_{-\infty}^\infty\int_0^\infty e^{-(u^2+r^2)}|2u^2-r^2| rdr dud\theta=2\pi \int_{-\infty}^\infty\int_0^\infty e^{-(u^2+r^2)}|2u^2-r^2| rdr dud
\]
($u=t\cos\vfi$, $r=t\sin\vfi$  $0\leq \vfi\leq \pi$, $t\geq 0$)
\[
=2\pi\int_0^\infty\left(\int_0^\pi|2\cos^2\vfi-\sin^2\vfi|\sin\vfi d\vfi\right) e^{-t^2} t^4 dt
\]
\[
\stackrel{x=\cos\vfi}{=}2\pi\left(\int_0^\infty e^{-t^2}t^4 dt\right)\cdot \left(\int_{-1}^1|3x^2-1| dx\right)=\pi\left(\int_0^\infty e^{-s} s^{3/2} ds \right)\cdot \left(\int_{-1}^1|3x^2-1| dx\right)
\]
\[
=\pi\cdot \Gamma(5/2)\cdot\left(\frac{8\sqrt{3}}{9}\right)=\frac{2\pi^{3/2}\sqrt{3}}{3}.
\]
Hence,
\[
\mu(\,\eM\,)= \frac{1}{\sqrt{3}}I=\frac{4\pi}{3}\approx  4.188.
\]

\end{proof}

\begin{remark} The typical trigonometric polynomial $\bt\in \bsT(\eM)$ is a  Morse function on $S^1\times S^1$, and thus it has an even number of critical points. Morse inequalities imply that it must have at least $4$ critical points.  We see that the expected number of critical points of a polynomial in $\bsV(\eM)$ is very close to this minimum, and that $\bsV(\eM)$ must contain Morse functions with at least $6$ critical points.  Arnold proved in  \cite{Ar06} that the typical function in $\bsT(\eM)$ has at most $8$ critical points.

 A later result of Arnold,  \cite[Thm. 1]{Ar07} states  that a generic trigonometric polynomial in $\bsV(\eM)$ has at most $6$ critical points. However,   there is an elementary, but consequential error in the proof of this theorem.     More precisely,   a key  concept in the proof is a (real)  linear operator that associates to each holomorphic  function $f:\bC\ra \bC$ a new function holomorphic  function  $\widehat{f}$  defined by  $\widehat{f}(z):=\overline{f(\bar{z})}$.  Arnold states   that if $z_0\in \bC$ is a critical point of $f$, i.e., $\frac{df}{dz}(z_0)=0$, then it is also a critical point of $\widehat{f}$.  Clearly this is true only if $z_0$ is real. For example, $z_0=\ii$ is a critical point of $f(z)=(z-\ii)^2$, but it is not a critical  point of $\widehat{f}(z)= (z+\ii)^2$.\qed
\label{rem: arn}
 \end{remark}

 \section{A product formula}
 \setcounter{equation}{0}

  Suppose that $(M,g,\bsV)$ is a  homogeneous triple, $m=\dim M$.  We say that it is \emph{special}  if  it admits a    \emph{core}, i.e., a quadruple  $(p, \underline{\bs}, \bsL,\bw)$,  where $p$ is a point  in $M$,  $\underline{\bsf}=\{\bsf_1,\dotsc,\bsf_m\}$  is  an orthonormal frame of $T_pM$, $\bsL$ is a subspace of $\bsV$ and $\bw\in \bsV$   such that the following hold.

  \begin{itemize}

 \item[${\bf P}_1$.] The vectors $\eA^\dag_p\bsf_r,\;\;r=1,\dotsc, m$ are mutually orthogonal.  For any $\bv\in \bsV$ we denote by $\Hess(\bv)$ the Hessian of $\bv$ at $p$ computed using the frame $\underline{\bsf}$.

  \item[${\bf P}_2$.]  The subspace  $\bsL$  is contained in $\bsK_p=\ker\eA_p$, and  for any $\bv\in L^\perp$ we have $\Hess(\bv)=0$.

    \item[${\bf P}_3$.] $\bw\in  \bsK_p\cap \bsL^\perp$,  $\bw(p)\neq 0$ and $|\bw|=1$. We set  $\widehat{\bsL}:=\bsL\oplus {\rm span}\,(\bw)$.

   \item[${\bf P}_4$.] $\widehat{L}^\perp\subset \ker \ev_p$.

 \end{itemize}

\begin{remark} (a) The importance of a core stems from the fact that in applications  we often have
\[
\dim\bsL \ll \dim \bsV.
\]
We regard $\Hess$ as a linear map
\[
\bsK_p\ra \Sym(T_pM) :=\mbox{symmetric  linear maps $T_pM\ra T_pM$}.
\]
We observe  that  $\bsL\supset (\ker \Hess)^\perp$, so we would expect the dimension of $\bsL$ to be   at least as big as $\binom{m+1}{2}=\dim \Sym(T_pM)$.  In many applications,  $\dim \bsL$ is only slightly bigger than $\binom{m+1}{2}$.

(b) The  conditions  ${\bf P}_3$, ${\bf P}_4$ can be  somewhat relaxed.  We  can define a core to be a subspace $\widehat{L}\subset \bsK_p$ that contains $\ev_{p}\in\bsV$ and satisfies ${\bf P}_2$.     For example, if $\ev_p\in \bsK_p$, we can choose $\widehat{\bsL}$ to be the sum between the line spanned by $\ev_p$ and the orthogonal complement of $\ker \Hess$, but this space may be difficult  to get a handle on in practice.  For reasons   having to do with the applications   we have in mind, we prefer to work with the above  more flexible  definition.    \qed
\end{remark}

  Suppose that $(p, \underline{\bs}, \bsL,\bw)$  is a core  of  the special   triple $(M,g,\bsV)$. A  basis    of  $\bsV$ is said to be \emph{adapted to the core} if it can be represented as collection of functions $Y_j\in \bsV$, $j\in J$, where $J$ is a set of cardinality $\dim V$  equipped  with a partition
  \[
  J=\{c\}\sqcup I \sqcup  I^*\sqcup  R_m
  \]
  such that the following hold.

  \begin{itemize}

  \item  The collection $(Y_j)_{j\in J}$ is an orthonormal basis of $\bsV$.

  \item The collection $(Y_j)_{j\in I}$ is an orthonormal basis of $\bsL$.

  \item $Y_c=\bw$.

  \item The collection $\{ Y_j;\;\;j\in\{c\}\cup I\cup I^*\,\}$ is an orthonormal basis of $\bsK_p$.

  \item $R_m=\{1,\dotsc, m\}$ and
  \[
  Y_r= \frac{1}{|\eA^\dag_p \bsf_r|} \eA^\dag_p \bsf_r,\;\;\forall r\in R_m.
  \]
  \end{itemize}

 For such a basis, we write $\hat{I}:=\{c\}\cup I$, so that  the collection $(Y_j)_{j\in \hat{I}}$ is an orthonormal basis of $\widehat{\bsL}$.

  \begin{proposition}   (a) Suppose  that $(M,g,\bsV)$ is a special triple and $(p, \underline{\bsf},\bsL,\bw)$ is a core   of this triple.   Set $m:=\dim M$, $\ell=\dim \bsL$ and $N=\dim\bsV$. Then
  \begin{subequations}
  \begin{equation}
  \mu(M,g,\bsV)= \frac{{\rm vol}_g(M)\Gamma(\frac{\ell+m}{2})}{2\pi^{\frac{\ell+m}{2} }\prod_{r=1}^m|\eA^\dag_p\bsf_r| }\int_{S(\bsL)} |\det H(\bv)|\,|dS(\bv)|
  \label{eq:  special}
  \end{equation}
  \begin{equation}
  =\frac{{\rm vol}_g(M)}{\pi^{\frac{\ell+m}{2} }\prod_{r=1}^m|\eA^\dag_p\bsf_r| }\int_{\bsL} e^{-|\bu|^2}|\det H(\bu)|\,|dV(\bu)|.
   \label{eq:  specialb}
  \end{equation}
  \end{subequations}
  
  (b) Suppose $(M_\alpha,g_\alpha,\bsV_\alpha)$, $\alpha=1,2$, are special triples with cores $(p_\alpha, \underline{\bsf}^\alpha, \bsL_\alpha,\bw_\alpha)$. Then the triple $(M_1\times M_2, g_1\oplus g_2, \bsV_1\otimes\bsV_2)$ is special. The core is defined by the datum $(\,p,\underline{\bsf},\bsL, \bw\,)$, where
  \[
  p:=(p_1,p_2),\;\;\underline{\bsf}:=\underline{\bsf}^1\cup\underline{\bsf}^2,\;\;\bw(x_1,x_2):=\bw_1(x_1)\bw_2(x_2),
  \]
  and
  \[
  \bsL := (\bsL_1\ast \bsL_2)\oplus \bsK_{p_1}^\perp\otimes\bsK^\perp_{p_2},
  \]
  where $\bsL_1\ast \bsL_2$ denotes  is the orthogonal complement of $\bw$ in $\widehat{\bsL}_1\otimes\widehat{\bsL}_2$.  Moreover
  \begin{equation}
  |\ev_{p_1,p_2}|=|\ev_{p_1}|\cdot |\ev_{p_2}|,
  \label{eq: ev-prod}
  \end{equation}
  \begin{equation}
  J(\eA^\dag_{(p_1,p_2)})=J(\eA^\dag_{p_1})\cdot J(\eA^\dag_{p_2})\cdot\frac{(|\ev_{p_1}|\cdot|\ev_{p_1}|)^{m_1+m_2}}{|\ev_{p_1}|^{m_1}\cdot |\ev_{p_2}|^{m_2}},
  \label{eq: prod-jac}
  \end{equation}
  where $J(S)$ denotes the Jacobian of a linear map between two Euclidean vector spaces.
  \label{prop: special}
  \end{proposition}

  \begin{proof} (a) Note  that for any $\bv\in \bsK_p$, the Hessian $H(\bv)$ of $\bv$ at $p$ depends only on the projection $\bar{\bv}$ of $\bv$ on $\bsL$.   Using (\ref{eq: int1}) we deduce
  \[
  \int_{S(\bsK_p)} |\det H(\bv)|\,|dS(\bv)|=\bsi_{N-m-\ell-1}\int_{B(\bsL)} (1-|\bx|^2)^{\frac{N-m-\ell-2}{2} }|\det H(\bx)|\,|dV(\bx)|
  \]
  \[
  \stackrel{(\ref{eq: int7})}{=}\frac{\bsi_{N-m-\ell-1} \Gamma(\frac{\ell+m}{2}) \Gamma(\frac{N-m-\ell}{2})}{2\Gamma(\frac{N}{2})}\int_{S(\bsL)} |\det H(\bx)|\,|dS(\bx)|
  \]
  \[
  =\frac{\pi^{\frac{N-m-\ell}{2}}\Gamma(\frac{\ell+m}{2}) }{\Gamma(\frac{N}{2})} \int_{S(\bsL)} |\det H(\bx)|\,|dS(\bx)|.
  \]
  We deduce that
  \[
  \mu(M,g,\bsV)=\frac{ {\rm vol}_g(M) }{ \bsi_{N-1}\prod_{r=1}^m|\eA^\dag_p\bsf_r| }\int_{S(\bsK_p)} |\det H(\bv)|\,|dS(\bv)|
  \]
  \[
  =\frac{{\rm vol}_g(M)\Gamma(\frac{\ell+m)}{2})}{2\pi^{\frac{\ell+m}{2} }\prod_{r=1}^m|\eA^\dag_p\bsf_r| }\int_{S(\bsL)} |\det H(\bv)|\,|dS(\bv)|.
  \]
  (b) Choose  an orthonormal basis $(Y^\alpha_j)_{j\in J_\alpha}$ of $\bsV_\alpha$   adapted to the core $(p_\alpha,\underline{\bsf}^\alpha, \bsL_\alpha,\bw_\alpha)$, where
  \[
  J_\alpha=\{c_\alpha\}\sqcup I_\alpha\sqcup I_\alpha^*\sqcup R_{m_\alpha}.
  \]
 The collection
  \[
  \bigl\{ Y^1_{j_1}Y^2_{j_2}\}_{(j_1,j_2)\in J_1\times J_2}
  \]
  is an orthonormal basis of $\bsV_1\otimes\bsV_2$.

  Observe that $ \Hess(Y^1_{j_1} Y^2_{j_2})$, the Hessian of $Y^1_{j_1} Y^2_{j_2}$ at $p=(p_1,p_2)$,  admits  a block decomposition
  \[
  \left[
  \begin{array}{cc}
  Y^2_{j_2}(p_2)\Hess(Y^1_{j_1}) &   A\\
  A^t & Y^1_{j_1}(p_1)  \Hess (Y^2_{j_2})
  \end{array}
  \right],
  \]
  where for any $(r_1,r_2)\in R$,  the $(r_1,r_2)$ entry of the matrix $A$ is
  \[
  A_{r_1r_2}=\pa_{\bsf^1_{r_1}}Y^1_{j_1}(p_1) \pa_{\bsf^2_{r_2}}Y^2_{j_2}(p_2),\;\;(r_1,r_2)\in R_{m_1}\times R_{m_2}.
  \]
 The properties ${\bf P_2}$ and ${\bf P}_4$ show that if $ \Hess(Y^1_{j_1} Y^2_{j_2})\neq 0$, then
  \[
  \mbox{either}\;\;(j_1,j_2)\in \hat{I}_1\times \hat{I}_2\setminus \{(c_1,c_2)\}\;\;\mbox{or}\;\; (j_1,j_2)\in  R_{m_1}\times R_{m_2}.
  \]
  This shows that if $\bv\perp \bsL$,  then $\Hess(\bv)=0$ at $(p_1,p_2)$. The condition $\bsL\subset \bsK_{p_1,p_2}$ follows from the properties ${\bf P}_2$ and ${\bf P}_4$ of special triples.

  Next observe that for any $r_1\in R_{m_1}$, we have
  \[
  \eA^\dag_{(p_1,p_2)}\bsf^1_{r_1}= \sum_{j_1\in J_1}\sum_{j_2\in J_2} \pa_{\bsf^1_{r_1}} Y^1_{j_1}(p_1) Y^2_{j_2}(p_2) Y^1_{j_1}Y^2_{j_2}
  \]
  \[
  =\sum_{j_2\in J_2} \bigl(\,\pa_{\bsf^1_{r_1}} Y^1_{r_1}(p_1) \,\bigr)Y^2_{j_2}(p_2) Y^1_{r_1}Y^2_{j_2}.
  \]
 This proves that  $ \eA^\dag_{(p_1,p_2)}\bsf^1_{r_1}(p_1,p_2)=0$  (since $Y^1_{r_1}(p_1)=0$ by ${\bf P}_4$ ) and
  \[
  | \eA^\dag_{(p_1,p_2)}\bsf^1_{r_1}|=|\pa_{\bsf^1_{r_1}} Y^1_{r_1}(p_1) |^2\sum_{j_2=0}^{k_2}|Y^2_{j_2}(p_2) |^2
  \]
  \[
  =|\eA^\dag_{p_1}\bsf^1_{r_1}|^2\sum_{j_2=0}^{k_2}|Y^2_{j_2}(p_2) |^2= |\eA^\dag_{p_1}\bsf^1_{r_1}|^2\,\cdot\,|\ev_{p_2}|^2.
  \]
  We have an analogous formula for $\eA^\dag_{(p_1,p_2)}\bsf^2_{r_2}$, $r_2\in R_{m_2}$. The conclusions of part (b) of  Proposition \ref{prop: special} are now obvious.
   \end{proof}

  \begin{ex} An important example of  special triple is $(S^{d-1}, g_d, \bsV_\nu)$, $d\geq 3$, where $(S^{d-1},g_d)$ is the round sphere of radius $1$, and  $\bsV_\nu$    is the space spanned by the eigenfunctions  of the Laplacian  corresponding to eigenvalues $\lambda_n=n(n+d-2)$, $n\leq \nu$. A core can be constructed as   follows.

   As distinguished point, we choose the North Pole $p_0=(0,\dotsc,0,1)$. Near $p_0$ we  use $\bx'=(x_1,\dotsc, x_{d-1})$ as local coordinates, and we set
  \[
  \bsf_r=\pa_{x_r}\in T_{p_0} S^{d-1}.
  \]
We  choose $\bw$ to be the constant function $\bsi_{d-1}^{-1/2}$.     Finally,   the  subspace $\bsL$  is the space spanned by the functions $\be_0,\be_\beta$ defined  by (\ref{eq: a0}), (\ref{eq: abeta}) and (\ref{eq: good-basis-sph}), so that
\[
I=\{0\}\cup B_{2,d}.
\]
   The properties ${\bf P}_1$, ${\bf P}_2$ and ${\bf P}_3$ of a core   are obvious. To prove  property ${\bf P}_4$, we have to show that if a function $\bv\in\bsV_\nu$ is  orthogonal to $\bw,\be_0,\dotsc,\be_\beta$, $\beta\in B_{2,d}$, then $\bv(p_0)=0$.   The function $\bv$ admits a decomposition
\[
\bv=\sum_{n=0}^\nu\sum_{j=0}^n\sum_{\beta\in B_{j,d}} v_{n,j,\beta} Z_{n,j,\beta}.
\]
Since $Z_{n,j,\beta}(p_0)=0$ if and only if $j>0$, we deduce that
\[
\bv(p_0)=\sum_{n=0}^\nu v_{n,0,1} Z_{n,0,1}(p_0).
\]
Note  that $\bw=Z_{0,0,1}$.  Since $\bv\perp \bw$, we deduce $v_{0,0,1}=0$, and therefore,
 \[
\bv(p_0)=\sum_{n=1}^\nu v_{n,0,1} Z_{n,0,1}(p_0).
\]
We now remark that  (\ref{eq: a0}) can be rewritten as
 \[
 \ba_0=-\frac{1}{2}\sum_{n=1}^\nu Z_{n, 0,1}(p_0) Z_{n,0,1}.
 \]
 We   deduce that
 \[
 0=-(\bv, 2\ba_0)=  \sum_{n=1}^\nu v_{n,0,1} Z_{n,0,1}(p_0) =  \bv(p_0).
 \]
 Note that (\ref{eq: hess-special}) implies
\begin{equation}
\Hess (\be_0)=r_0(\nu)^{1/2}\one_{d-1},\;\;\Hess(\be_\beta)=r_\beta(\nu)^{1/2}H_\beta.
\label{eq: hess-special1}
\end{equation}
 From (\ref{eq: adj-sph0}) we deduce
 \begin{equation}
 \eA^\dag_{p_0}\bsf_r= \bsi^{-1/2}_{d-3}C_{1,0,d-1}\sum_{n=1}^\nu C_{n,1,d}P_{n,d}'(1)  Z_{n,1,r}.
 \label{eq: adj-sph1}
 \end{equation}
Using (\ref{eq: adj-length})  and (\ref{eq: asyr}), we deduce that this triple has the additional property that
\begin{equation}
|\eA^\dag_{p_0}\bsf_1|^2=\cdots =|\eA^\dag_{p_0}\bsf_{d-1}|^2=r(d,\nu)\sim \underbrace{\frac{1}{2^{d-1}\Gamma(\frac{d-1}{2})^2 \bsi_{d-3}(d+1)}}_{=:\bar{r}(d)}\nu^{d+1}\;\;\mbox{as}\;\;\nu\ra \infty.
\label{eq: adj-length1}
\end{equation}
Let us compute the length of the evaluation functional $\ev_{p_0}=\ev_{p_0,\nu}:\bsV_\nu\ra \bR$. We will use the notations in  the proof of Theorem \ref{th: sph-harm}.  We have
\[
|\ev_{p_0,\nu}|^2=\sum_{n=0}^\nu\sum_{j=0}^n \sum_{\beta\in B_{j,d}} \bigl|\, C_{n,j,d}P_{n,d}^{(j)}(1) Y_{j,\beta}(0)\,\bigr|^2=\sum_{n=0}^\nu \sum_{\beta\in B_{0,d}} \bigl|\, C_{n,0,d}P_{n,d}(1) Y_{0,\beta}(0)\,\bigr|^2.
\]
In this case, the basis $\eB_{0,d-1}$ consists of a single constant function $Y_0=Y_{0,\beta}=\bsi_{d-2}^{-1/2}$ and
 \[
 P_{n,d}(1)=1,\;\; C_{n,0,d}^2\stackrel{(\ref{eq: leg4})}{=}\frac{(2n+d-2)[n+d-3]_{d-3}}{ 2^{d-2}\Gamma(\frac{d-1}{2})}.
 \]
 Hence,
 \[
 Y_{0,\beta}(0)^2 C_{n,0,d}^2= \frac{(2n+d-2)[n+d-3]_{d-3}}{(2\pi)^{d-1}},
 \]
 and we conclude that
 \begin{equation}
 |\ev_{p_0,\nu}|^2 =\frac{1}{(2\pi)^{d-1}}\sum_{n=0}^\nu(2n+d-2)[n+d-3]_{d-3} \sim \frac{2}{(2\pi)^{d-1} (d-1)}\nu^{d-1}\;\;\mbox{as}\;\;\nu\ra \infty.
 \label{eq: ev-sph}
 \end{equation}
 For $r=1,\dotsc, d-1$, we set
 \begin{equation}
 U_r:=\frac{1}{|\eA^\dag_{p_0}\bsf_r|}\eA^\dag_{p_0}\bsf_r.
 \label{eq: ur}
 \end{equation}
 The   computations in the proof  of Theorem \ref{th: sph-harm} imply that
 \begin{subequations}
 \begin{equation}
 \pa_{x_1}U_1(p_0)=\pa_{x_2}U_2(p_0)=\dotsc= \pa_{x_{d-1}}U_{d-1}(p_0)= c(d,\nu)=r(d,\nu)^{1/2}.
 \label{eq: par}
 \end{equation}
 \begin{equation}
 \pa_{x_i}U_j(p_0)=0,\;\;\forall i,j=1,\dotsc, d-1,\;\;i\neq j.
 \label{eq: par1}
 \end{equation}
 \end{subequations}
\qed
\label{ex: prod-sph}
\end{ex}

\begin{theorem} Assume $d_1,d_2\geq 3$ and fix $r\geq 1$. Then there exists a positive constant $K$ that depends only on $d_1$ and $d_2$ such that,
\[
\mu(S^{d_1-1}\times S^{d_2-1},\bsV_{\nu_1}(d_1)\otimes \bsV_{\nu_2}(d_2)\,\bigr)\sim K\bigl(\dim \bsV_{\nu_1}(d_1)\otimes \bsV_{\nu_2}(d_2)\,\bigr)^{\varpi(d_1,d_2, r)}
\]
 if  $\nu_1,\nu_2\ra \infty$  and $\nu_1\nu_2^{-r}$ converges to a positive constant.  The exponent $\varpi(d_1,d_2, r)$ is described in (\ref{tag: vpi}),
 \[
 \varpi(d_1,d_2,r)=\frac{2(d_1-3)r +2d_2+2}{2(d_1-1)r+ 2d_2-2}.
 \]
\label{th: prod-sph}
\end{theorem}

\begin{proof} Choose cores $(p_\alpha,\underline{\bsf}^\alpha, \bsL_\alpha,\bw_\alpha)$ of $(S^{d_\alpha-1},g_0,\bsV_{\nu_\alpha})$  as indicated  in   Example \ref{ex: prod-sph}. Next, choose  bases adapted to these cores
\[
(Y^\alpha_j)_{j\in J_\alpha},\;\; J_\alpha= \{c_\alpha\}\sqcup I_\alpha\sqcup I_\alpha^*\sqcup R_{d_\alpha-1},\;\;I_\alpha =\{0\}\cup   B_{2,d_\alpha},
\]
 and set
 \[
 S= \bigl(\,\hat{I}_1\times \hat{I}_2\,\bigr)\setminus\{(c_1,c_2)\},\;\; R=R_{d_1-1}\times R_{d_2-1},\;\;I=S\cup R.
 \]
 Recall that for $i_\alpha\in I_\alpha=\{0\}\cup B_{2,d_\alpha}$, we have
 \[
 Y^\alpha_{i_\alpha}=\begin{cases}
 \be_0(\nu_\alpha), &i_\alpha=0\\
 \be_\beta(\nu_\alpha), & i_\alpha=\beta\in B_{2,d_\alpha},
 \end{cases}
 \]
 where the functions $\be_0$, $\be_\beta$ are defined by (\ref{eq: good-basis-sph}). Moreover, for $r_\alpha\in R_{d_\alpha-1}=\{1,\dotsc,d_\alpha-1\}$, we have $Y^\alpha_{r_\alpha}=U_{r_\alpha}$, where $U_r$ is defined by (\ref{eq: ur}).

 We construct a core $(p,\underline{\bsf},\bsL,\bw)$ of $(S^{d_1-1}\times S^{d_2-1}, \bsV_{\nu_1}\otimes\bsV_{\nu_2})$ as in Proposition \ref{prop: special}. Note that the collection
 \[
\bigl\{\, Y_{i_1,j_2}:= Y^1_{i_1} Y^2_{i_2};\;\;(i_1,i_2)\in I\,\bigr\}
\]
is an orthonormal basis of $\bsL$.   For $\bv\in \bsV_{\nu_1}\otimes\bsV_{\nu_2}$ we denote by $\Hess(\bv)$ the Hessian matrix of $\bv$ at $p$ computed using the frame $\underline{\bsf}$. Note that if $(i_1,i_2)\in S$, then
\[
\Hess(Y_{i_1,i_2})=\left[
\begin{array}{cc}
Y^2_{i_2}(p_2) \Hess(Y^1_{i_1}) & 0\\
&\\
0 & Y^1_{i_1}(p_1) \Hess( Y^2_{i_2})
\end{array}
\right].
\]
Using (\ref{eq: par}) and (\ref{eq: par1}) we deduce that for  $(r_1,r_2)\in R$ we have
\[
\Hess(Y_{r_1,r_2})=\left[
\begin{array}{cc}
0 & c(d_1,\nu_1)c(d_2,\nu_2)\Delta_{r_1,r_2}\\
&\\
c(d_1,\nu_1)c(d_2,\nu_2)\Delta_{r_1,r_2}^\dag & 0
\end{array}
\right]
\]
\[
=c(d_1,\nu_1)c(d_2,\nu_2)\underbrace{\left[
\begin{array}{cc}
0 & \Delta_{r_1,r_2}\\
&\\
\Delta_{r_1,r_2}^\dag & 0
\end{array}
\right]}_{=:\widehat{\Delta}_{r_1,r_2}},
\]
where $\Delta_{r_1,r_2}$ denotes the $(d_1-1)\times (d_2-1)$ matrix  whose  entry on the $(r_1,r_2)$ position is $1$, while the other entries are $0$. Thus, if
\[
\bv=\sum_{(i_1,i_2)\in I} v_{i_1,i_2}Y_{i_1,i_2}\in \bsL,
\]
then
\begin{equation}
\Hess(\bv)=\sum_{(i_1,i_2)\in S}  v_{i_1,i_2} \Hess( Y_{i_1,i_2})+ c(d_1,\nu_1)c(d_2\nu_2)\sum_{(r_1,r_2)\in R} v_{r_1,r_2} \widehat{\Delta}_{r_1,r_2}.
\label{eq: hess8}
\end{equation}
To make further progress, we need to     choose the  basis $(Y^\alpha_{i_\alpha})_{i_\alpha\in I_\alpha}$ of $\bsL_\alpha$ as indicated in Example   \ref{ex: prod-sph}.  Using the notations in the proof of Theorem \ref{th: sph-harm} we let
\[
I_\alpha=\{0\}\cup  B_{2,d_\alpha}
\]
and  the functions $Y^\alpha_0$ respectively $Y^\alpha_\beta$, $\beta\in B_{2,d_\alpha}$  are equal to the functions $ \be_0$ and respectively $\be_\beta$ defined by  (\ref{eq: good-basis-sph}
), (\ref{eq: a0}), (\ref{eq: abeta}).  More precisely, for $\alpha=1,2$, we have
\[
Y^\alpha_0=\frac{1}{|\ba_0^\alpha|}\ba_0^\alpha,
\]
where
\[
\ba_0^\alpha=-\frac{1}{2}\bsi_{d_\alpha-2}^{-1/2}\sum_{n=0}^{\nu_\alpha}  C_{n,0,d_\alpha}n\Bigl(\,n+\frac{d_\alpha-3}{2}\,\Bigr)Z_{n,0,1},
\]
and for $\beta\in B_{2,d_\alpha}$, we have
\[
Y^\alpha_\beta=\frac{1}{|\ba^\alpha_\beta|} \ba^\alpha_\beta,
\]
where
\[
\ba^\alpha_\beta=\sum_{n=2}^{\nu_\alpha} C_{n,2,d} P_{n,d}^{(2)}(1) Z_{n,2,\beta}.
\]
We set
\[
r_0(\nu_\alpha):=|\ba_0^\alpha|^2,\;\;r_\beta(\nu_\alpha):=|\ba_\beta^\alpha|^2.
\]
Using (\ref{eq: asib}) and (\ref{eq: asy0}) we   deduce that  for any $d\geq 2$  there exist explicit positive constants $\bar{r}_0(d)$, $\bar{r}_1(d)$ such that

\begin{subequations}
\begin{equation}
r_0(\nu_\alpha) \sim r_0(d_\alpha)\nu_\alpha^{d_\alpha+3},\;\;\mbox{as}\;\;\nu_\alpha\ra \infty,
\label{eq: asy0-1}
\end{equation}
\begin{equation}
r_\beta(\nu_\alpha)\sim r_1(d_\alpha) \nu_\alpha^{d_\alpha+3},\;\;\mbox{as}\;\;\nu_\alpha\ra \infty.
\label{eq: asib-1}
\end{equation}
\end{subequations}
Let us observe that for any $\beta\in B_{2,d_\alpha}$ we have
\begin{equation}
Y^\alpha_\beta(p_\alpha)=0,\; \Hess(Y^\alpha_\beta)\stackrel{(\ref{eq: hess-special1})}{=} r_\beta(\nu_\alpha)^{1/2}H_\beta \sim r_1(d_\alpha) \nu_\alpha^{\frac{d_\alpha+3}{2}}H_\beta,\;\;\mbox{as}\;\;\nu_\alpha\ra \infty.
\label{eq: asy-hess-b}
\end{equation}
Moreover,
\begin{equation}
\Hess(Y^\alpha_0)=r_0(\nu_\alpha)^{1/2}\one_{d_\alpha-1} \sim r_0(d_\alpha)^{1/2}\nu_\alpha^{\frac{d_\alpha+3}{2}}\one_{d_\alpha-1},\;\;\mbox{as}\;\;\nu_\alpha\ra \infty.
\label{eq: asy-hess-0}
\end{equation}
Next,
\[
Y^\alpha_0(p_\alpha)=-\frac{1}{2r_0(\nu_\alpha)^{1/2}}\bsi_{d_\alpha-2}^{-1/2}\sum_{n=0}^{\nu_\alpha}  C_{n,0,d_\alpha}n\Bigl(\,n+\frac{d_\alpha-3}{2}\,\Bigr)Z_{n,0,1}(p_0),
\]
where
\[
Z_{n,0,1}(p_\alpha)\stackrel{(\ref{eq: zn})}{=} C_{n,0,d_\alpha} \bsi_{d_\alpha-2}^{-1/2}.
\]
We deduce
\[
Y^\alpha_0(p_\alpha)= -\frac{1}{2r_0(\nu_\alpha)^{1/2}\bsi_{d_\alpha-2}}\sum_{n=0}^{\nu_\alpha}  C_{n,0,d_\alpha}^2n\Bigl(\,n+\frac{d_\alpha-3}{2}\,\Bigr)
\]
\[
=-\frac{1}{2^{d-1}r_0(\nu_\alpha)^{1/2}\bsi_{d_\alpha-2}\Gamma(\frac{d_\alpha-1}{2})}\sum_{n=0}^{\nu_\alpha}n\Bigl(\,n+\frac{d_\alpha-3}{2}\,\Bigr)(2n+d_\alpha-2)[n+d_\alpha-3]_{d_\alpha-3}
\]
\[
=-\frac{1}{2(4\pi)^{\frac{d-1}{2}}r_0(\nu_\alpha)^{1/2}}\sum_{n=0}^{\nu_\alpha}n\Bigl(\,n+\frac{d_\alpha-3}{2}\,\Bigr)(2n+d_\alpha-2)[n+d_\alpha-3]_{d_\alpha-3}.
\]
Note that
\[
\sum_{n=0}^{\nu_\alpha}n\Bigl(\,n+\frac{d_\alpha-3}{2}\,\Bigr)(2n+d_\alpha-2)[n+d_\alpha-3]_{d_\alpha-3}\sim\frac{2}{d_\alpha+1}\nu_\alpha^{d_\alpha+1}\;\;\mbox{as}\;\;\nu_\alpha\ra \infty.
\]
Using (\ref{eq: asy0-1}), we deduce  that
\begin{equation}
Y^\alpha_0(p_\alpha) \sim -\frac{1}{(4\pi)^{\frac{d_\alpha-1}{2}}(d_\alpha+1) r_0(d_\alpha)^{1/2}}\nu_\alpha^{\frac{d_\alpha-1}{2}}\;\;\mbox{as}\;\; \nu_\alpha\ra \infty.
\label{eq: asy0-2}
\end{equation}
For $i\in I_\alpha$, we  define the  symmetric $(d_\alpha-1)\times(d_\alpha-1)$--matrix
\begin{equation}
H_i^\alpha :=\begin{cases}
\one_{d-1}, & i=0\\
H_\beta, & i=\beta\in B_{2,d_\alpha}.
\end{cases}
\label{eq: hess9}
\end{equation}
Putting together all of the above, we  deduce that for $(i_1,i_2)\in S$  and $\nu_1,\nu_2\ra \infty$ we have
\begin{subequations}
\begin{equation}
Y^2_{i_2}(p_2) \Hess(Y^1_{i_1}) \sim\begin{cases}
0, & i_2\in B_{2,d_2}\;\mbox{or}\;\;i_1=c_1\\
A\nu_1^{\frac{d_1+3}{2}} H_{i_1}^1 , & i_2=c_2\;\;\mbox{and}\;\;i_1\in I_1\\
 B_{i_1} \nu_2^{\frac{d_2-1}{2}} \nu_1^{\frac{d_1+3}{2}} H_{i_1}^1, & i_2=0\;\;\mbox{and}\;\;i_1\in I_1,
\end{cases}
\label{eq: hess-asy2}
\end{equation}
\begin{equation}
Y^1_{i_1}(p_1) \Hess(Y^2_{i_2}) \sim\begin{cases}
0, & i_1\in B_{2,d_1}\;\mbox{or}\;\;i_2=c_2\\
C\nu_2^{\frac{d_2+3}{2}} H_{i_2}^1 , & i_1=c_1\;\;\mbox{and}\;\;i_2\in I_2\\
 D_{i_2} \nu_1^{\frac{d_1-1}{2}} \nu_2^{\frac{d_2+3}{2}} H_{i_2}^2, & i_1=0\;\;\mbox{and}\;\;i_2\in I_2,
\end{cases}
\label{eq: hess-asy3}
\end{equation}
\end{subequations}
where $A,C$ are nonzero constants that depend only on $d_1$ and $d_2$,  $B_{i_1}$ is a nonzero constant that depends on $i_1\in I_1$, $,d_1,d_2$, and $D_{i_2}$ is a nonzero constant that depends on $i_2\in I_2$, $d_1,d_2$.

Similarly,  using the  estimate (\ref{eq: adj-length1})  we deduce that for  $(r_1,r_2)\in R$ and $\nu_1,\nu_2\ra \infty$, there exists a nonzero constant $E=E(d_1,d_2)$ such that
\begin{equation}
\Hess(Y_{r_1,r_2}) \sim E \nu_1^{\frac{d_1+1}{2}}\nu_2^{\frac{d_2+1}{2}} \widehat{\Delta}_{r_1,r_2}.
\label{eq: hess-asy31}
\end{equation}
Using (\ref{eq: prod-jac}),  (\ref{eq: ev-sph}) and (\ref{eq: adj-length1})   we deduce that the Jacobian  $J(\nu_1,\nu_2)$ of  $\eA^\dag_{(p_1,p_2)}$  satisfies the  asymptotic estimate
\[
J(\nu_1,\nu_2)=J(\eA^\dag_{p_1})\cdot J(\eA^\dag_{p_2})\cdot |\ev_{p_1}|^{d_2-1}\cdot|\ev_{p_1}|^{d_1-1}\sim   CF\nu_1^{\frac{(d_1^2-1)}{2}} \nu_2^{\frac{(d_2^2-1)}{2}} (\nu_1\nu_2)^{\frac{(d_1-1)(d_2-1)}{2}},
\]
where $F$ is a positive constant. Assume now that
\[
\nu_1\sim t^{2\kappa_1},\;\;\nu_2\sim t^{2\kappa_2},\;\;t\ra\infty,
\]
i.e., $\nu_1,\nu_2$ go to infinity   in such   a fashion that
\[
\frac{\nu_1}{\nu_2^{r}} \ra 1,\;\;r:=\frac{\kappa_1}{\kappa_2}.
\]
The assumption $r\geq 1$ implies that
\begin{equation}
\kappa_1\geq \kappa_2 >0.
\label{eq: order}
\end{equation}
We  have
 \[
 \nu_\alpha^{(d_\alpha+3)/2}\sim t^{p_\alpha},\;\; p_\alpha:=3\kappa_\alpha d_\alpha+3\kappa_\alpha,\;\;\alpha=1,2,
 \]
 \[
\nu_2^{\frac{d_2-1}{2}} \nu_1^{\frac{d_1+3}{2}}\sim C t^{\omega_{11}},\;\;\omega_{11}={\kappa_1 d_1+\kappa_2 d_2 +3\kappa_1-\kappa_2},
\]
\[
\nu_1^{\frac{d_2-1}{2}} \nu_2^{\frac{d_1+3}{2}}\sim Ct^{\omega_{22}},\;\;\omega_{22}={\kappa_1 d_1+\kappa_2 d_2 +3\kappa_2-\kappa_1},
\]
\[
 \nu_1^{\frac{d_1+1}{2}}\nu_2^{\frac{d_2+1}{2}}\sim C t^{\omega_{12}},\;\;\omega_{12}={\kappa_1d_1+\kappa_2d_2+\kappa_1+\kappa_2},
 \]
\[
 \nu_1^{\frac{(d_1^2-1)}{2}} \nu_2^{\frac{(d_2^2-1)}{2}} (\nu_1\nu_2)^{\frac{(d_1-1)(d_2-1)}{2}} \sim C t^q,
 \]
 where
 \[
 q= {\kappa_1(d_1^2-1)+\kappa_2(d_2^2-1)+(\kappa_1+\kappa_2)(d_1-1)(d_2-1)}
 \]
 \[
 =\kappa_1d_1^2+\kappa_2d_2^2+(\kappa_1+\kappa_2) (d_1d_2-d_1-d_2).
 \]
From (\ref{eq: order}) we deduce
 \[
 \omega_{22}-\omega_{11}= -4(\kappa_1-\kappa_2){\leq} 0,\;\;\omega_{12}-\omega_{11}= -2(\kappa_1-\kappa_2)  {\leq} 0,
 \]
 \[
p_2-\omega_{11}= -(d_1+3)\kappa_1+4\kappa_2 <  -4(\kappa_1-\kappa_2),
 \]
 \[
 p_1-\omega_{11} = -\kappa_2(d_2-1) < 0 .
 \]
 Using (\ref{eq: hess8}), (\ref{eq: hess9}), (\ref{eq: hess-asy2}), (\ref{eq: hess-asy3}) and (\ref{eq: hess-asy31}) we deduce that
 \[
 H(\bv)\sim t^{\omega_{11}}\underbrace{\left[
 \begin{array}{cc}
 \bigl(\bsB(\bv) +o(1)\,\bigr) & t^{\omega_{12}-\omega_{11}}E\Bigl( \sum_{r_1,r_2} v_{r_1,r_2}\Delta_{r_1,r_2} +o(1)\Bigr)\\
 & \\
E t^{\omega_{12}-\omega_{11}}\Bigl(\sum_{r_1,r_2} v_{r_1,r_2}\Delta_{r_1,r_2} +o(1)\Bigr) & D_0v_{0,0}t^{\omega_{22}-\omega_{11}} \bigl(\,\one_{d_2-1}+ o(1)\,\bigr)
 \end{array}
 \right]}_{=:A(t)},
 \]
 where $o(1)$ denotes a quantity that converges to $0$ as $t\ra \infty$, \emph{uniformly} with respect to $\bv\in S(\bsL)$, and
 \[
 \bsB(\bv)=\sum_{i_1} B_{i_1}v_{i_1,0}H^1_{i_1},
 \]
 where $B_{i_1}$ are defined as in (\ref{eq: hess-asy2}).  We set
 \[
 S(t):=\left[
 \begin{array}{cc}
 \one_{d_1-1} & 0\\
 0 & t^{-2(\kappa_1-\kappa_2)}
 \end{array}
 \right].
 \]
 Observe that
 \[
 A(t)= S(t)  \cdot \left[
 \begin{array}{cc}
\bigl(\bsB(\bv)+o(1)\,\bigr) & E\Bigl( \sum_{r_1,r_2} v_{r_1,r_2}\Delta_{r_1,r_2} +o(1)\Bigr)\\
 & \\
E\Bigl(\sum_{r_1,r_2} v_{r_1,r_2}\Delta_{r_1,r_2} +o(1)\Bigr) & D_0 v_{0,0} \bigl(\,\one_{d_2-1}+ o(1)\,\bigr)
 \end{array}
 \right]\cdot  S(t).
 \]
 We deduce that
 \[
 \det H(\bv) \sim  t^{(d_1+d_2-2)\omega_{11}} |\det S(t)|^2 \cdot \det  \underbrace{\left[
 \begin{array}{cc}
 \bsB(\bv) & E\sum_{r_1,r_2} v_{r_1,r_2}\Delta_{r_1,r_2}\\
 & \\
E\sum_{r_1,r_2} v_{r_1,r_2}\Delta_{r_1,r_2}  & D_0v_{0,0} \one_{d_2-1}
 \end{array}
 \right]}_{=:H_\infty(\bv)}
 \]
 \[
 =t^{(d_1+d_2-2)\omega_{11}-4d_2(k_1-k_2)}\det H_\infty(\bv).
 \]
 Let us point out that  $\det H_\infty(\bv)$ is not identically zero. To see this, it suffices to choose $\bv$ such that $v_{0,0}=1$ and all the other coordinates $v_{i_1,i_2}$ are trivial. In this case (\ref{eq: hess9}) and (\ref{eq: hess-asy2}) imply that
 \[
 H_\infty(\bv)=\left[\begin{array}{cc}
 B_0\one_{d_1-1} & 0\\
 0 & D_0 \one_{d_2-1}
 \end{array}
 \right].
 \]
  It follows that
 \[
 \mu(S^{d_1-1}\times S^{d_2-1}, \bsV_{t^{2\kappa_1}}\otimes \bsV_{t^{2\kappa_2}})\sim C  t^{(d_1+d_2)\omega_{11}-4d_2(k_1-k_2)-q}.
 \]
 An elementary computation shows that
 \[
 (d_1+d_2-2)\omega_{11}-4d_2(k_1-k_2)-q= 2d_1\kappa_1+2d_2\kappa_2-6\kappa_1+2\kappa_2.
 \]
 On the other hand,
 \[
 \dim  \bsV_{t^{2\kappa_1}}\otimes \bsV_{t^{2\kappa_2}}= \dim  \bsV_{t^{2\kappa_1}}\times \dim\bsV_{t^{2\kappa_2}}\stackrel{(\ref{eq: asy-dim})}{\sim} K_{d_1,d_2} t^{2\kappa_1 d_1+2\kappa_2d_2-2\kappa_1-2\kappa_2}.
 \]
The desired conclusion follows by observing that
\[
\frac{2d_1\kappa1+2d_2\kappa_2-6\kappa_1+2\kappa_2}{2\kappa_1d_1+2\kappa_2d_2-2\kappa_1-2\kappa_2}=\frac{2(d_1-3)r +2d_2+2}{2(d_1-1)r+ 2d_2-2}=\varpi(d_1,d_2,r),\;\;r=\frac{\kappa_1}{\kappa_2}.
\]

\end{proof}

  \section{Random polynomials on $S^1\times S^{d-1}$, $d\geq 3$}
  \setcounter{equation}{0}

 For any $m\in \bZ$ define $X_m:S^1\ra \bR$ by
\begin{equation}
\Phi_m(\theta)=\begin{cases}
(2\pi)^{-1/2}, &m=0\\
\pi^{-1/2}\cos(m\theta), & m<0\\
\pi^{-1/2}\sin(m\theta), & m>0.
\end{cases}
\label{eq: phim}
\end{equation}
The collection $(\Phi_m)_{m\in\bZ}$ is an orthonormal Hilbert basis of $L^2(S^1,d\theta)$.   For any positive integer $\nu>0$,  we set
\[
\bsT_\nu:= {\rm span}\,\bigl\{ \Phi_m;\;|m|\leq \nu\,\bigr\}.
\]
In other words, $\bsT_\nu$ is the space of trigonometric polynomials  of degree $\leq \nu$.

\begin{lemma}  Let $g_0$ denote the natural metric on $S^1$ of length $2\pi$. Then the triple $(S^1,g_0,\bsT_\nu)$ is special. Moreover
\begin{equation}
|\ev_0|^2 =\frac{1}{\pi}\left(\nu+\frac{1}{2}\right).
\label{eq: ev-circle}
\end{equation}
\label{lemma: special-circle}
\end{lemma}

\begin{proof} As base point we choose $p_0=0$ and  the frame is $\underline{\bsf}=\{\pa_\theta\}$.  We denote by $\bsK_0$ the space  of trigonometric  polynomials that have $0$ as a critical point.    Also, for any trigonometric polynomial $\bt\in\bsT_\nu$ we denote by   $\Hess (\bt)$ the Hessian of $\bt$ at $0$, i.e., the $1\times 1$ matrix $\Hess(\bt):=\pa^2_\theta\bt(0) \one$. Note that
\[
\eA^\dag_0 \pa_\theta=-\pi^{-1/2}\sum_{m>0}  m\Phi_m.
\]
In particular,
\begin{equation}
J_\nu:=|\eA^\dag_0\pa_\theta|=\pi^{-1/2}\left(\sum_{m=1}^\nu m^2\,\right)^{1/2}\sim (3\pi)^{-1/2}\nu^{3/2}\;\;\mbox{as}\;\;\nu\ra \infty.
\label{eq: adj-circle}
\end{equation}
We set
\begin{equation}
\bp_\nu:=\frac{1}{|\eA^\dag_0\pa_\theta|} \eA^\dag_0\pa_\theta,
\label{eq: pnu}
\end{equation}
A simple computation shows that
\begin{equation}
\pa_\theta\bp_\nu(0)=J_\nu.
\label{eq: adj-circle1}
\end{equation}
Next observe that
\[
\Hess(\Phi_m)= \begin{cases}
0, &m\geq 0\\
-m^2 \Phi_m(0)\one, &m<0.
\end{cases}
\]
Thus, if
\[
\bt=\sum_{|m|\leq \nu}t_m \Phi_m\in \bsT_\nu,
\]
 then
 \[
\Hess(\bt)= -\pi^{-1/2} \Bigl(\sum_{m<0} m^2t_m\Bigr)\one.
\]
We now introduce
\[
\ba=\ba_\nu=-\pi^{-1/2}\sum_{m<0}m^2 \Phi_m,\;\;\be=\be_{\nu}=\frac{1}{|\ba_{\nu}|}\ba_{\nu},
\]
so that
\begin{equation}
\Hess(\bt)= (\bt,\ba)\one= |\ba_\nu|(\bt,\be_\nu)\one.
\label{eq: hess-circle}
\end{equation}
We set
\[
r_{\nu}:=|\ba_\nu|^2=\frac{1}{\pi}\sum_{m=1}^\nu m^4 =\frac{1}{5\pi}B_5(\nu+1)\sim \frac{1}{5\pi}\nu^5\;\;\mbox{as}\;\;\nu\ra \infty,
\]
  and we observe that
\begin{equation}
\det \Hess(\be_\nu)=r_\nu^{1/2}=(5\pi)^{-1/2}B_5(\nu+1)^{1/2}\sim (5\pi)^{-1/2}\nu^{5/2}\;\;\mbox{as}\;\;\nu\ra \infty.
\label{eq: hes-circle1}
\end{equation}
We set  $c(\nu)=\be_\nu(0)$, and we observe that
\begin{equation}
c(\nu)=-\frac{1}{(\pi r_\nu)^{1/2}}\sum_{m<0}m^2 \Phi_m(0)\sim -\frac{5^{1/2}}{3\pi^{1/2}}\nu^{1/2}\;\;\mbox{as}\;\;\nu\ra \infty.
\label{eq: enu0}
\end{equation}
 We see  that  we can choose as core the quadruple $(p_0,\pa_\theta,\bsL_\nu, Y_0)$, where $\bsL_\nu$ denotes the $1$-dimensional  space spanned by $\be_\nu$.
  The equality (\ref{eq: ev-circle}) follows from the identity $\ev_0=\sum_{m\leq 0} \Phi_m(0) \Phi_m$.
\end{proof}

\begin{corollary}  There exists a universal  positive constant $K$ such that
\[
\mu(S^1,\bsT_\nu)\sim 2\nu \sqrt{\frac{3}{5}}\sim\sqrt{\frac{3}{5}} \dim \bsT_\nu\;\;\mbox{as}\;\;\nu\ra \infty.
\]
This agrees with our previous  estimate (\ref{eq: trig-poly1}).
\label{cor: av-circle}
\end{corollary}

\begin{proof} We use Proposition  \ref{prop: special}(a), and we have
\[
\mu(S^1, \bsT_\nu)= \frac{  {\rm vol}\,(S^1)\Gamma(1)}{2\pi |\eA^\dag\pa_\theta|}\cdot 2 |\Hess(\,\be(\nu)\,)|
\]
(use (\ref{eq: adj-circle}) and (\ref{eq: hes-circle1}) )
\[
\sim 2\sqrt{\frac{3}{5}}\nu\;\;\mbox{as}\;\;\nu\ra \infty.
\]
\end{proof}

\begin{remark}  Let us mention  that, according to J. Dunnage, \cite{Du}, the expected number of real zeros of a  random trigonometric polynomial  in
\[
\bsT_\nu^0=\spa\bigl\{ \Phi_m;\;\;0<|m|\leq \nu\,\bigr\}.
\]
equipped with the $L^2$-metric is $\sim \frac{2}{\sqrt{3}}\nu$ as $\nu \ra \infty$.   The number of zeros of such a polynomial is a random variable $\zeta_\nu$, and its  asymptotic  behavior as $\nu\ra \infty$   has been recently   investigated in great detail  by A. Granville and I. Wigman,  \cite{GW}.  Let us observe  the operator $\pa_\theta$ induces a linear isomorphism
\[
\pa_\theta:\bsT_\nu^0\ra \bsT_\nu^0.
\]
From this point of view we see that  the expected number of critical points  of a random trig polynomial   in $\bsT_\nu^0$ equipped with the  $L^2$-metric is equal to the  expected number of critical  zeros of a   random  trig polynomial in $\bsT_\nu^0$, \emph{equipped with the Sobolev norm} $\|u\|_{H^{-1}}$.  In Section \ref{s: 7}  we will  describe the asymptotic behavior  as $\nu\ra \infty$ of the variance of the  number of critical  points of a random trigonometric polynomial in $\bsT_\nu^0$.
 \qed
 \label{rem: cov}
 \end{remark}

 \begin{ex}[Approximation regimes with large upper complexity] Suppose 
 \[
 \vfi: \{0,1,2,\dotsc,\}\ra \{0,1,2,\dotsc\}
 \]
  is a bijection such that $\vfi(0)=0$.  Define
 \[
 \bsT_\nu^\vfi:={\rm span}\,\bigl\{\Phi_{\pm\vfi(m)};\;\; 0\leq m\leq \nu\,\bigr\}.
 \]
 We denote  by $\mu(\bsT_\nu^\vfi)$ the expected number of critical points of a random trigonometric    polynomial in $\bsT_\nu^\vfi$. A simple modification of the arguments used in the proofs of Lemma \ref{lemma: special-circle} shows that
 \begin{equation}
 \mu(\bsT_\nu^\vfi)=2\left(\frac{\sum_{k=1}^\nu \vfi(k)^4}{\sum_{k=1}^\nu \vfi(k)^2}\right)^{1/2}.
 \label{eq: vfi}
 \end{equation}
 We  want to construct\footnote{This construction was worked out during a lively conversation with my colleague Richard Hind.} a permutation $\vfi$ such that
 \[
 \limsup_{\nu\ra\infty}\frac{\log \mu(\bsT_\nu^\vfi)}{\log\dim\bsT_\nu^\vfi}=\infty.
 \]
 To do this   we fix a    very fast increasing sequence of positive integers $(\ell_n)_{n\geq 0}$ such that
 \[
 \ell_0=0,\;\;\frac{\ell_{n+1}}{\ell_n}=2^n,\;\;\forall n\geq 1.
 \]
 For $n\geq 0$ we    set
 \[
 S_n:=\{\ell_n+1,\ell_n+2,\dotsc, \ell_{n+1}\,\bigr\}.
 \]
 We consider the  bijection $\phi: \{0,1,2,\dotsc,\}\ra \{0,1,2,\dotsc\}$  uniquely determined by the following requirements.

 \begin{itemize}

 \item $\phi(0)=0$

 \item $\phi(S_n)=S_n$,  $\forall n\geq 0$.

 \item  The restriction of $\vfi$ to $S_n$ is strictly decreasing so that
 \[
 \phi(\ell_n+1)= \ell_{n+1},\;\;\phi(\ell_n+2)=\ell_{n+1}-1\;\;\mbox{etc}.
 \]
 \end{itemize} We set
 \[
 \nu_n:=\ell_n+1,\;\;\bsW_n:= \bsT_{\nu_n}^\phi.
 \]
 Note that the collection $(\bsW_n)_{n\geq 0}$ is an approximation regime in the sense defined in the introduction, and
 \[
 \dim\bsW_n= 2\ell_n+3.
 \]
 We claim  that
 \begin{equation}
 \lim_{n\ra \infty}\frac{\log(\mu(\bsW_n)}{\log\dim\bsW_n)}= \infty.
 \label{eq: claim}
 \end{equation}
 Indeed, for any positive integer $k$, we have
 \[
 \sum_{m=1}^{\nu_n}\phi(m)^k=\sum_{m=1}^{\ell_n}m^k +\ell_{n+1}^k= P_{k+1}(\ell_n) +\ell_n^{k\cdot 2^n},
 \]
 where, according to (\ref{tag: ber}), 
 \[
 P_k(x)=\frac{1}{k+1}\bigl(\, B_{k+1}(x)-B_{k+1}\,\bigr)
 \]
 is a universal polynomial of degree $k+1$. Using (\ref{eq: vfi}) we deduce
 \[
 \frac{1}{4}\mu(\bsW_n)^2 =\frac{P_5(\ell_n)+\ell_n^{2^{n+2}}}{P_3(\ell_n)+\ell_n^{2^{n+1}}}\sim \ell_n^{2^{n+1}}\;\;\mbox{as}\;\;n\ra \infty.
 \]
 Hence,
 \[
 \log \mu(\bsW_n)\sim 2^{n}\log \ell_n,
 \]
 which  proves the claim (\ref{eq: claim}).\qed

 \label{ex: claim}
 \end{ex}
 
 \begin{remark} Let us observe that  for any  positive $\nu$, the space
 \[
\bsT_\nu^L:=\underbrace{ \bsT_\nu\otimes \cdots \bsT_\nu}_L\subset  C^\infty(\, \underbrace{S^1\times\cdots\times S^1}\,)
\]
contains  the space $\bsV(\eM_\nu^L)$ of Section \ref{s: torus} as a codimension one  subspace. The orhogonal complement of $\bsV(\eM_\nu^L)$ in $\bsT_\nu^L$  is the $1$-dimensional space spanned by the constant functions.\qed
\end{remark}

  \begin{theorem} For any $d\geq 3$ there  exists a universal positive constant $K=K_d$ such that
  \[
  \mu(S^1\times S^{d-1}, \bsT_\rho\otimes\bsV_\nu(d))\sim K_d (\dim \bsT_\rho\otimes\bsV_\nu(d))\;\;\mbox{as}\;\;\rho,\nu\ra \infty.
  \]
  \label{th: circle-sph}
  \end{theorem}

  \begin{proof} We will again rely on Proposition \ref{prop: special}.  We consider the core $(0,\pa_\theta, \bsL_\rho, \bw_\rho)$ of $(S_1, g_{S^1}, \bsT_\rho)$ described in Lemma \ref{lemma: special-circle} and the core  $(p_0, \underline{\bsf}, \bsL_\nu,\bw_\nu)$  of $(S^{d-1}, g_{S^{d-1}},\bsV_\nu)$ described in Example \ref{ex: prod-sph}.  We form the core $(p, \underline{\bsf}', \bsL_{\rho,\nu},\bw)$ of $(S^1\times S^{d-1}, g_{S^1}+ g_{S^{d-1}}, \bsT_\rho\otimes\bsV_\nu)$  following the prescriptions   in the proof of Proposition \ref{prop: special}. We have
  \[
 \bsL_\rho= {\rm span}\,\{\be_\rho\},\;\;\be_\rho=\frac{1}{|\ba_\rho|},\;\;\ba_\rho=-\pi^{-1/2}\sum_{m=-1}^{-\rho} m^2 \Phi_m(\theta),\;\; \bw_\rho(\theta)=(2\pi)^{-1/2}.
 \]
 As in the proof  of  Theorem \ref{th: prod-sph}, we choose  a basis
 \[
(Y_j)_{j\in J},\;\; J_\alpha= \{\ast\}\sqcup I\sqcup I^*\sqcup R_{d-1},\;\;I =\{0\}\cup   B_{2,d_\alpha},
\]
adapted to the core  $(p_0, \underline{\bsf}, \bsL_\nu,\bw_\nu)$. We have
\[
Y_{\ast_d}=\bsi_{d-1}^{-\frac{1}{2}}.
\]
 For $i\in I=\{0\}\cup   B_{2,d_\alpha}$, we have
\[
  Y_i= \begin{cases}
  \be_0(\nu), & i=0\\
  \be_\beta(\nu), &i=\beta\in B_{2,d},
  \end{cases}
 \]
 where the functions $\{\be_0(\nu)$, $\be_\beta(\nu)$ are  defined by (\ref{eq: good-basis-sph}). For $r\in R_{d-1}=\{1,\dotsc, d-1\}$ we have
  \[
  Y_r=U_r:=\frac{1}{|\eA^\dag_{p_0}\pa_{x_r}|}\eA^\dag_{p_0}\pa_{x_r},\;\;\eA^\dag_{p_0}\pa_{x_r}\stackrel{(\ref{eq: adj-sph1})}{=} \bsi^{-1/2}_{d-3}C_{1,0,d-1}\sum_{n=1}^\nu C_{n,1,d}P_{n,d}'(1)  Z_{n,1,r}.
  \]
  We can now write down an orthonormal basis of $\bsL_{\rho,\nu}$,
  \[
 \bigl\{A_{i,1}:= \bw_\rho Y_i;\;\;i\in I\,\bigr\}  \cup\bigl\{Z:=\be_\rho Y_\ast \}\cup\{A_{i,2}:=\be_\rho Y_i;\;\;i\in I\}\cup \{B_r:=\bp_\rho Y_r;\;\;r\in R_{d-1}\},
 \]
 where $\bp_\rho$ is given by  (\ref{eq: pnu}). For any function $v\in\bsV:=\bT_\rho\otimes\bsV_\nu$ we denote by $\Hess(\bv)$ its Hessian at $(0,p_0)$. We have
 \[
 \Hess(A_{i,1})=\left[
 \begin{array}{cc}
 0 & 0 \\
 &\\
 0 & (2\pi)^{-1/2} \Hess(Y_i)
 \end{array}
 \right],\;\; \Hess(Z)=\left[
 \begin{array}{cc}
 \bsi^{-1/2}_{d-1}\Hess(\be_\rho) & 0\\
 0 & 0\\
 \end{array}
 \right],
 \]
 \[
 \Hess(A_{i,2})=\left[
 \begin{array}{cc}
 \Hess(\be_\rho) Y_i(p_0) & 0\\
 &\\
 0 & \be_\rho(0) \Hess (Y_i)
 \end{array}
 \right],
 \]
 \[
 \Hess(p_\rho Y_r)=\left[
 \begin{array}{cc}
 0 &  \bp_\rho'(0) \pa_{x_r}Y_r(p_0)\Delta_r\\
 &\\
  \bp_\rho'(0) \pa_{x_r}Y_r(p_0)\Delta_r^\dag & 0
  \end{array}
  \right],
  \]
  where $\Delta_r$ denotes the $1\times (d-1)$ matrix with a single nonzero entry equal to $1$ in the $r$-th   position, and $\Delta_r^\dag$ denotes its transpose. In the sequel,  the symbols $C_0,C_1,\dotsc, $ will indicate positive constants that depend only on $d$.

  Using(\ref{eq: adj-circle1}),  (\ref{eq: hes-circle1}) and (\ref{eq: enu0})   we deduce that as $\rho\ra \infty$, we have
 \[
\bp_\rho'(0)\sim C_0\rho^{3/2},\;\; \Hess(\be_\rho)\sim C_1\rho^5,\;\;\be_\rho(0)\sim -C_2\rho^{1/2}.
 \]
 Using (\ref{eq: asib}), (\ref{eq: asy0}) and (\ref{eq: hess-special1}) we deduce that as $\nu\ra \infty$, we have
 \[
 \Hess(Y_0)\sim C_3 \nu^{\frac{d+3}{2}}\one_{d-1},\;\;\Hess(Y_\beta)\sim C_4 \nu^{\frac{d+3}{2}}H_\beta\;\;\beta\in B_{2,d}.
 \]
For $i\in I$ we set
\[
H_i:=\begin{cases}
\one_{d-1}, &i=0,\\
H_\beta, & i=\beta\in B_{2,d}.
\end{cases}
\]
We have
 \[
 Y_\beta(p_0)=0,\;\;\forall\beta\in B_{2,d},
 \]
 while (\ref{eq: asy0-2}) implies  that as $\nu\ra \infty$, we have
 \[
 Y_0(p_0)\sim-C_5\nu^{\frac{d-1}{2}}.
 \]
 Finally, using (\ref{eq: adj-length}) and (\ref{eq: par}), we deduce that as $\nu\ra \infty$ we have
 \[
 \pa_{x_r}Y_r(p_0)\sim C_6\nu^{\frac{d+1}{2}}.
 \]
  Putting together all of the above, we deduce that if
  \[
  \bv=\sum_{i\in I}v_{i,1}A_{i,1}+\sum_{i\in I}v_{i,2} A_{i,1}+ z Z+\sum_{r\in R} v_r B_r\in \bsL_{\rho,\nu},
  \]
  then, as $\rho,\nu\ra \infty$, we have
  \[
  \Hess(\bv)\sim \left[
  \begin{array}{cc}
  -C_7v_{0,2}\rho^5\nu^{\frac{d-1}{2}}+C_8z\rho^5 & C_9\sum_{r\in R}\rho^{3/2}\nu^{\frac{d+1}{2}}v_r\Delta_r\\
  &\\\
  C_9\sum_{r\in R}\rho^{3/2}\nu^{\frac{d+1}{2}}v_r\Delta_r^\dag &\nu^{\frac{d+3}{2}}\bigl(\,\bsB_0(\bv)+\rho^{1/2}\bsB_1(\bv)\,\bigr)
  \end{array}
  \right],
  \]
  where
  \[
  \bsB_0(\bv):=C_{10}v_{0,1}\one_{d-1}+\sum_{i\in B_{2,d}} C_{11}v_{i,1} H_i,\;\;\bsB_1(\bv):=C_{10}''v_{0,2}\one_{d-1}-\sum_{i\in B_{2,d}} C_{11}'v_{i,2} H_i.
  \]
  Factoring out $\nu^{\frac{(d-1)}{2})}$ and then $\rho^{5/2}$ from the first row and the first column, we deduce that
  \[
  \det \Hess(\bv) \sim\nu^{\frac{d(d-1)}{2}}\rho^5\det  \left[
  \begin{array}{cc}
  -C_7v_{0,2}+C_8z & C_9\sum_{r\in R}\rho^{-1}\nu v_r\Delta_r\\
  &\\\
  C_9\sum_{r\in R}\rho^{-1} \nu v_r\Delta_r^\dag &\nu^2\bigl(\,\bsB_0(\bv)+\rho^{1/2}\bsB_1(\bv)\,\bigr)
  \end{array}
  \right]
  \]
  (factor out $\nu$ from  the last $(d-1)$ rows and the last $(d-1)$ columns)
  \[
  =\nu^{\frac{d(d-1)}{2}+2(d-1)}\rho^5 \det  \left[
  \begin{array}{cc}
  -C_7v_{0,2}+C_8z & C_9\sum_{r\in R}\rho^{-1}v_r\Delta_r\\
  &\\\
  C_9\sum_{r\in R}\rho^{-1} v_r\Delta_r^\dag &\bsB_0(\bv)+\rho^{1/2}\bsB_1(\bv)
  \end{array}
  \right]
\]
(factor out $\rho^{1/4}$ from the last $(d-1)$ rows and the last $(d-1)$ columns)
 \[
= \nu^{\frac{(d+4)(d-1)}{2}}\rho^{5 +\frac{d-1}{2}} \det  \left[
  \begin{array}{cc}
  -C_7v_{0,2}+C_8z & C_9\sum_{r\in R}\rho^{-5/4}v_r\Delta_r\\
  &\\
  C_9\sum_{r\in R}\rho^{-5/4} v_r\Delta_r^\dag &\bsB_1(\bv)+\rho^{-1/2}\bsB_0(\bv)
  \end{array}
  \right]
  \]
  \[
  \sim \nu^{\frac{(d+4)(d-1)}{2}}\rho^{5 +\frac{d-1}{2}} \det  \left[
  \begin{array}{cc}
  -C_7v_{0,2}+C_8z & 0\\
  &\\\
  0 &\bsB_1(\bv)
  \end{array}
  \right].
  \]
  To compute the Jacobian $J_{\rho,\nu}$ of the adjunction map $\eA^\dag_{0,p_0}: T_{(0,p_0)}(\, S^1\times S^{d-1}\,)\ra \bsT_\rho\otimes\bsV_\nu$ we use  Proposition \ref{prop: special}(a).    We will denote by $K_0,K_1,\dotsc $ positive constants that depend only on $d$.

  According to (\ref{eq: ev-circle}) and (\ref{eq: adj-circle}),  the  Jacobian  $J_\rho$ of   $\eA^\dag_0: T_0S^1\ra\bsT_\rho$, and the evaluation  functional $\ev^\rho_0:\bsT_\rho\ra \bR$ satisfy the  $\rho\ra \infty$    asymptotics
  \[
  |\ev^\rho|\sim K_0\rho^{1/2},\;J_\rho\sim K_1\rho^{3/2}.
  \]
  Using  (\ref{eq: adj-length1}) and (\ref{eq: ev-sph})   we  deduce that    the  Jacobian  $J_\nu$ of   $\eA^\dag_{p_0}: T_{p_0}S^{d-1}\ra\bsT_\rho$ and the evaluation  functional $\ev^\nu=\ev^\nu_{p_0}:\bsV_\nu\ra \bR$ satisfy the  $\nu\ra \infty$    asymptotics
  \[
 J_\nu\sim K_2\nu^{\frac{(d-1)(d+1)}{2}},\;\;|\ev^\nu|\sim K_3\nu^{\frac{d-1}{2}}.
 \]
  Using (\ref{eq: prod-jac}) ,we conclude that
  \[
  J_{\rho,\bu}=J_\rho\cdot J_\nu \cdot |\ev^\rho|^{d-1}\cdot |\ev^{\nu}|\sim  K_4 \rho^{\frac{3}{2}+\frac{(d-1)}{2}}\nu^{\frac{(d-1)(d+2)}{2}}.
  \]
  We conclude that as $\rho,\nu\ra \infty$, we have
  \[
  \mu(S^1\times S^{d-1}, \bsT_\rho\otimes \bsV_\nu)\sim K_5 \rho \nu^{d-1}\sim K_6 \dim(\bsT_\rho\otimes \bsV_\nu).
  \]
  \end{proof}

\section{The variance of the number of critical points of a random trigonometric polynomial}
\label{s: 7}
\setcounter{equation}{0}

The statistics of the zero set of a random trigonometric polynomial      is equivalent  with the statistics of the zero set of the  gaussian field
\[
\eta_\nu(t) =\frac{1}{\sqrt{\pi\nu}}\sum_{m=1}^\nu (a_m \cos mt +b_m \sin mt),
\]
where  $a_m, b_m$ are  independent   normally distributed random variables with mean $0$ and  variance $1$. This is a stationary  gaussian process with  covariance function
\[
\si_\nu(t)=\frac{1}{\pi\nu} \sum_{m=1}^\nu \cos mt.
\]
The statistics of the critical set of a    random sample function of the above process    is identical to the statistics of the  zero set of a random sample  function of the stationary gaussian process
\[
\xi_\nu(t)=\frac{d\eta_\nu((t)}{dt} =\frac{1}{\sqrt{\pi\nu}}\sum_{m=1}^\nu (-ma_m \sin mt +m b_m \cos mt).
 \]
Equivalently,  consider the   gaussian process
\[
\Phi_\nu(t)=\frac{1}{\sqrt{\pi\nu^3}}\sum_{m=1}^n\left(  mc_m \cos\left(\frac{t}{\nu}\right) + md_m\sin\left(\frac{t}{\nu}\right)\right),
\]
where $c_m,d_m$ are independent   random variables with   identical standard normal distribution, and the random variable
\[
Z_\nu:=\mbox{the number of zeros of $\Phi_\nu(t)$ in the interval $[-\pi\nu,\pi\nu]$}.
\]
Note that   the expectation of $Z_\nu$ is precisely the expected number of critical points of a  random trigonometric polynomial in the  space
\[
\spa\bigl\{ \cos( m \theta),\; \sin(m \theta);\;\; m=1,\dotsc \nu\,\bigr\},
\]
equipped with the inner product
\[
(u,v)=\int_0^{2\pi} u(\theta)v(\theta) d\theta.
\]
The  covariance function of $\Phi_\nu$ is
\[
R_\nu(t)=\frac{1}{\pi\nu^3}\sum_{m=1}^\nu m^2\cos\left(\frac{mt}{\nu}\right).
\]
The Rice formula, \cite[Eq. (10.3.1)]{CrLe},  implies  that the expectation of $Z_\nu$ is
\[
\bsE(Z_\nu)= 2\nu\left(\frac{\lambda_2(\nu)}{\lambda_0(\nu)}\right)^{\frac{1}{2}},
\]
where
\[
\lambda_0(\nu)=R_\nu(0)=\frac{1}{\pi\nu^3}\sum_{m=1}^\nu m^2\;\;\mbox{and}\;\;\lambda_2(\nu)=-R_\nu''(0)=\frac{1}{\pi\nu^5}\sum_{m=1}^\nu m^4.
\]
This is in perfect agreement with our earlier computations.   We let $\bsE(\zeta)$, and respectively $\var(\zeta)$, denote the expectation, and respectively the variance, of  a random variable $\zeta$. The following is the main result of this section.

\begin{theorem}  Set
\[
\bla_0:=\lim_{\nu\ra \infty}\lambda_0(\nu)=\frac{1}{3},\;\;\bla_0:=\lim_{\nu\ra \infty}\lambda_2(\nu):=\frac{1}{5},
\]
Then for any $t\in\bR$ the limit  $\lim_{\nu\ra \infty} R_\nu(t)$ exists, it is equal to 
\[
\Ri(t):=\frac{1}{t^3}\int_0^t \tau^2\cos \tau d\tau=\int_0^1  \lambda^2 \cos (\lambda t) d\lambda,\;\;\forall t\in\bR,
\]
and 
\begin{equation}
\lim_{\nu\ra \infty}\frac{1}{\nu} \var(Z_\nu)= \delta_\infty:=\frac{2}{\pi}\int_{-\infty}^\infty\left( f_\infty(t)- \frac{\bla_2}{\bla_0}\right) dt + 2\sqrt{\frac{\bla_2}{\bla_0}},
\label{eq: varinfi}
\end{equation}
where
\[
f_\infty(t)= \frac{(\bla_0^2-R_\infty^2)\bla_2-\bla_0(R_\infty')^2}{(\lambda_0^2-R_\infty^2)^{\frac{3}{2}}}\left(\sqrt{1-\rho_\infty^2}+ \rho_\infty\arcsin\rho_\infty\right),
\]
and
\[
\rho_\infty=\frac{\Ri''(\bla_0^2-\Ri^2)+(\Ri')^2R_\infty}{(\bla_0^2-\Ri^2)\bla_2-\bla_0(\Ri')^2}.
\]
Moreover, the constant $\delta_\infty$ is positive.\footnote{Numerical experiments indicate that $\delta_\infty\approx 0.35$.}
\label{th: var}
\end{theorem}

\begin{proof} We follow a strategy inspired from \cite{GW}.  The variance of $Z_\nu$ can be computed using the results in \cite[\S 10.6]{CrLe}.  We introduce the    gaussian field
\[
\bPsi(t_1,t_2):=\left[
\begin{array}{c}
\Phi_\nu(t_1)\\
\Phi_\nu(t_2)\\
\Phi_\nu'(t_1)\\
\Phi_\nu'(t_2)
\end{array}
\right].
\]
Its covariance  matrix depends only on $t=t_2-t_1$. We have (compare with \cite[Eq. (17)]{GW})
\[
\bXi(t)=\left[\begin{array}{cccc}
\lambda_0  &  R_\nu(t) &  0 & R'_\nu(t)\\
R_\nu(t) & \lambda_0 & -R'_\nu(t) & 0\\
0 &-R_\nu'(t) &  \lambda_2 & -r_\nu''(t) \\

R'_\nu(t) & 0 & -R_\nu''(t)& \lambda_2
\end{array}
\right]=:\left[
\begin{array}{cc} A & B\\
B^\dag & C
\end{array}
\right].
\]
As explained  in \cite{Qua}, to apply  \cite[\S 10.6]{CrLe} we only need that $\bXi(t)$ is nondegerate.   This is established in the next result  whose proof can be found in Appendix \ref{s: d}.

\begin{lemma} The matrix $\bXi(t)$ is nonsingular if and only if $t\not \in 2\pi\nu\bZ$.\qed
\label{lemma: bxi}
\end{lemma}
For  any vector
\[
\bx:=\left[
\begin{array}{c}
x_1\\
x_2\\
x_3\\
x_4
\end{array}
\right]\in\bR^4
\]
we set
\[
p_{t_1,t_2}(\bx):=\frac{1}{4\pi^2 (\det\bXi)^{1/2}} e^{-\frac{1}{2}(\bXi^{-1}\bx,\bx)}.
\]
Then,  the results in \cite[\S 10.6]{CL} show that
\begin{equation}
\bsE(Z_\nu^2)-\bsE(\eZ_\nu)=\int_{\bsI_\nu\times\bsI_\nu}\left(\int_{\bR^2}|y_1y_2| \cdot  p_{t_1,t_2}(0,0,y_1,y_2) |dy_1dy_2|\right) |dt_1dt_2|.
\label{eq: cov}
\end{equation}
As in \cite{GW} we have
\[
\bXi(t)^{-1}=\left[ 
\begin{array}{cc}
\ast & \ast \\
\ast & \Omega^{-1}
\end{array}
\right], \Omega= C-B^\dag A^{-1} B.
\]
More explicitly,
\[
\Omega= C-B^\dag A^{-1} B
\]
\[
\begin{split}
&=\left[
\begin{array}{cc}
\lambda_2 & -R_\nu''(t)\\
-R_\nu''(t) &\lambda_2
\end{array}
\right] \\
&- \frac{1}{\lambda_0^2-R_\nu(t)^2}\left[
\begin{array}{cc}
0  &-R'_\nu(t)\\
R_\nu'(t) & 0
\end{array}
\right]\cdot \left[
\begin{array}{cc}
\lambda_0 & -R_\nu(t)\\
-R_\nu(t) &\lambda_0
\end{array}
\right]\cdot \left[
\begin{array}{cc}
0  &R'_\nu(t)\\
-R_\nu'(t) & 0
\end{array}
\right]
\end{split}
\]
\[
=\left[
\begin{array}{cc}
\lambda_2 & -R_\nu''\\
-R_\nu'' &\lambda_2
\end{array}
\right] - \frac{(R_\nu')^2}{\lambda_0^2-R_\nu^2}\left[
\begin{array}{cc}
\lambda_0 & R_\nu \\
R_\nu  & \lambda_0 
\end{array}
\right]=\mu\left[
\begin{array}{cc}
1 &-\rho\\
-\rho & 1
\end{array}
\right],
\]
where
\[
\mu=\mu_\nu=\frac{(\lambda_0^2-R_\nu^2)\lambda_2-\lambda_0(R_\nu')^2}{\lambda_0^2-R_\nu^2}, \;\;\rho=\rho_\nu=\frac{R_\nu''(\lambda_0^2-R_\nu^2)+(R'_\nu)^2R_\nu}{(\lambda_0^2-R_\nu^2)\lambda_2-\lambda_0(R_\nu')^2}.
\]
We want to emphasize, that in the above equalities the constants $\lambda_0$ and $\lambda_2$ do depend on $\nu$, although we have not indicated this in our notation. 
\begin{remark} The nondegeneracy of $\bXi$ implies that $\mu_\nu(t)\neq 0$ and $|\rho_\nu(t)|<1$, for all $t\not\in2\pi\nu\bZ$. \qed
\label{rem: bxi}
\end{remark}
We obtain as in \cite[Eq. (24)]{GW}
\[
\det\bXi= \det A\cdot \det \Omega=\mu^2(\lambda_0^2-R_\nu^2)(1-\rho^2),\;\;\Omega^{-1}= \frac{1}{\mu(1-\rho^2)}\left[
\begin{array}{cc}
1 &\rho\\
\rho & 1
\end{array}
\right].
\]
We can now rewrite the equality (\ref{eq: cov}) as  ($t=t_2-t_1$)
\[
\bsE(Z_\nu^2)-\bsE(Z_\nu) = \int_{\bsI_\nu\times\bsI_\nu}\left(\int_{\bR^2} |y_1y_2|e^{-\frac{y_1^2+2\rho y_1y_2+y_2^2}{2\mu(1-\rho^2)}} \frac{|dy_1dy_2|}{4\pi^2}\right) \frac{|dt_1dt_2|}{\mu\sqrt{(\lambda_0^2-R_\nu^2)(1-\rho^2)}}
\]
\[
=\int_{\bsI_\nu\times\bsI_\nu}\left(\int_{\bR^2} |y_1y_2|e^{-\frac{x_1^2+2\rho x_1x_2+x_2^2}{2(1-\rho^2)}} \frac{|dy_1dy_2|}{4\pi^2}\right) \frac{\mu |dt_1dt_2|}{\sqrt{(\lambda_0^2-R_\nu^2)(1-\rho^2)}}.
\]
From \cite[Eq. (A.1)]{BD} we deduce that
\[
\int_{\bR^2} |y_1y_2|e^{-\frac{x_1^2+2\rho x_1x_2+x_2^2}{2(1-\rho^2)}} \frac{|dy_1dy_2|}{4\pi^2}=\frac{1-\rho^2}{\pi^2}\left(1+\frac{\rho}{\sqrt{1-\rho^2}}\arcsin\rho\right).
\]
Hence
\[
\bsE(Z_\nu^2)-\bsE(Z_\nu)=\frac{1}{\pi^2}\int_{\bsI_\nu\times\bsI_\nu}\frac{\mu}{(\lambda_0^2-R_\nu^2)^{\frac{1}{2}}} \left(\sqrt{1-\rho^2}+{\rho}\arcsin\rho\right)|dt_1dt_2|
\]
\begin{equation}
=\frac{1}{\pi^2}\int_{\bsI_\nu\times\bsI_\nu}\underbrace{\frac{(\lambda_0^2-R_\nu^2)\lambda_2-\lambda_0(R_\nu')^2}{(\lambda_0^2-R_\nu^2)^{\frac{3}{2}}} \left(\sqrt{1-\rho^2}+{\rho}\arcsin\rho\right)}_{=:f_\nu(t)}  |dt_1dt_2|.
\label{eq: fnu}
\end{equation}
The function $f_\nu(t)=f_\nu(t_2-t_1)$ is doubly periodic  with periods $2\pi\nu$, $2\pi\nu$ and we conclude that
\begin{equation}
\bsE([Z_\nu]_2):=\bsE(Z_\nu^2)-\bsE(Z_\nu)=\frac{2\nu}{\pi}\int_{-\pi\nu}^{\pi\nu} f_\nu(t) dt.
\label{eq: cov2}
\end{equation}
We conclude that
\begin{equation}
\begin{split}
\var(Z_\nu) & = \bsE([Z_\nu]_2)+\bsE(Z_\nu)-\bigl(\,\bsE(Z_\nu)\,\bigr)^2= \bsE([Z_\nu]_2)-\bigl[\,\bsE(Z_\nu)\,\bigr]_2\\
& = \frac{2\nu}{\pi}\int_{-\pi\nu}^{\pi\nu} \left(f_\nu(t)-\frac{\lambda_2}{\lambda_0}\right) dt + 2\nu\sqrt{\frac{\lambda_2}{\lambda_0}}
\end{split}
\label{eq: cov21}
\end{equation}
To  complete the proof of  Theorem \ref{th: var}  we need to investigate the integrand in (\ref{eq: cov2}). This requires a detailed understanding of   the behavior of $R_\nu$ as $\nu \ra \infty$.     It is useful  to consider more general  sums of the form
\[
A_{\nu,r}(t)=\frac{1}{\nu^{r+1}}\sum_{m=1}^\nu m^r \cos \frac{mt}{\nu},\;\; B_{\nu,r}(t)_=\frac{1}{\nu^{r+1}}\sum_{m=1}^\nu m^r \sin \frac{mt}{\nu}, r\geq 1.
\]
Note that   if we set $z:=\cos \frac{mt}{\nu}+\ii \sin \frac{mt}{\nu}$. We have
\[
A_{\nu,r}(t)+\ii B_{\nu,r}(t)=\frac{1}{\nu^{r+1}}\underbrace{\sum_{m=1}^\nu m^rz^m}_{=: C_{\nu,r}(t)}.
\]
Observe that
\begin{equation}
R_\nu^{(k)}(t)=-\frac{1}{\pi\nu^{k+3}}\re\left(\frac{1}{\ii^{k+2}} C_{\nu,k+2}(t)\,\right).
\label{eq: cov-func}
\end{equation}
We set
\[
A_{\infty,r}(t)=\frac{1}{t^{r+1}}\int_0^t\tau^r\cos\tau d\tau,\;\;B_{\infty,r}(t):=\frac{1}{t^{r+1}}\int_0^t\tau^r\sin\tau d\tau.
\]
Observe that $R_\nu=A_{\nu,2}$ and  $\Ri=A_{\infty,2}$. We have the following result.
\begin{lemma} 
\begin{equation}
|A_{\nu,r}(t)-A_{\infty,r}(t)|+|B_{\nu,r}(t)-B_{\infty,r}(t)| = O\left(\frac{\max(1,t)}{\nu}\right),\;\;\forall t\geq 0,
\label{eq: est0}
\end{equation}
where, above and in the sequel, the constant implied  by the $O$-symbol is independent of  $t$ and $\nu$. In particular
\begin{equation}
\lim_{\nu\ra \infty}\frac{1}{\nu^{r+1}} C_{\nu,r}(t)= C_r(t):=\frac{1}{t^{r+1}}\int_0^t\tau^r e^{\ii\tau} d\tau,\;\;\forall t\geq 0.
\label{eq: C}
\end{equation}
\label{lemma: cov}
\end{lemma}
\begin{proof}We have
\[
A_\nu(t)= \frac{1}{\nu^{r+1}}\sum_{m=1}^\nu m^r \cos\left(\frac{mt}{\nu}\right)= \frac{1}{t^{r+1}}\underbrace{\sum_{m=1}^\nu \left\{\left(\frac{mt}{\nu}\right)^r \cos\left(\frac{mt}{\nu}\right)\right\}\cdot\left( \frac{t}{\nu}\right)}_{=:S_\nu(t)}.
\]
The term $S_\nu(t)$ is a Riemann sum corresponding to the integral
\[
\int_0^t f(\tau) d\tau,\;\; f(\tau):=\tau^r\cos\tau,
\]
and the  subdivision
\[
0<\frac{t}{\nu}<\cdots <\frac{(\nu-1)t}{\nu}<t.
\]
of the interval $[0,t]$. A simple application of the mean value theorem implies that there exist points 
\[
\theta_m\in \left[\frac{(m-1)t}{\nu},\frac{mt}{\nu}\right]
\]
such that
\[
\int_0^t f(\tau) d\tau =\sum_{m=1}^\nu f(\theta_m)\frac{t}{m}.
\]
We deduce that
\[
\int_0^t f(\tau) d\tau -S_\nu(t)=\frac{t}{\nu}\sum_{m=1}^\nu\left(\,f(\theta_m)-f\left(\,\frac{mt}{\nu}\right)\,\right)
\]
Now set
\[
M(t):=\sup_{0\leq \tau\leq t} |f'(\tau)|.
\]
Observe that
\[
M(t)=\begin{cases}
O(t^{r-1}),& 0\leq t\leq 1\\
O(t^r), & r>1.
\end{cases}
\]
We deduce
\[
\left| S_\nu(t)-\int_0^t f(\tau) d\tau\right|\leq M(t)\cdot \frac{t^2}{\nu}.
\]
This, proves the $A$-part  of (\ref{eq: est0}). The $B$-part is completely similar.
\end{proof}

We  need to refine  the estimates (\ref{eq: est0}). Recall that $[m]_r:=m(m-1)\cdots (m-r+1)$,  $r\geq 1$. We  will express $C_{\nu,r}(t)$ in terms of the sums
\[
D_{\nu,r}(t):=\sum_{m=1}^\nu [m]_r z^m= z^r\frac{d^r}{dz^r}\left(\sum_{k=1}^\nu z^k\right)=z^r\frac{d^r}{dz^r}\left(\frac{z-z^{\nu+1}}{1-z}\right)=z^r\frac{d^r}{dz^r}\left(\frac{1-z^{\nu+1}}{1-z}\right).
\]
Using the classical formula
\[
m^r=\sum_{k=1}^r S(r,k)[m]_k,
\]
where $S(r,k)$ are  the Stirling numbers of the  second kind, we deduce,
\begin{equation}
C_{\nu,r}(\zeta)=\sum_{k=1}^r S(r,k) D_{\nu,k}(\zeta) =\sum_{k=1}^r S(r,k)z^r\frac{d^k}{dz^k}\left(\frac{1-z^{\nu+1}}{1-z}\right).
\label{eq: cov3}
\end{equation}

\begin{lemma} Set $\theta:=\frac{t}{2\nu}$, and $f(\theta)=\frac{\sin\theta}{\theta}$. Then
\begin{equation}
\frac{t^{r+1}}{\nu^{r+1}}D_{\nu, r}(t)=\ii^r r!\left(\,\frac{2\sin\left(\frac{(\nu+1)t}{2\nu}\right)}{f(\theta)^{r+1}}\cdot   e^{\frac{\ii(\nu+r)t}{2\nu} }-e^{\ii t}\sum_{j=1}^r\ii^{1-j}\binom{\nu+1}{j}t^j\cdot \left(\frac{e^{\ii\theta}}{f(\theta)}\right)^{r+1-j}\right)
\label{eq: sharp}
\end{equation}
\label{lemma: sharp}
\end{lemma}

\begin{proof} We have
\[
D_{\nu, r}(t)=z^r\sum_{j=0}^r \binom{r}{j}\frac{d^j}{dz^j}(1-z^{\nu+1})\frac{d^{r-j}}{dz^{r-j}} (1-z)^{-1}
\]
\[
=r!\frac{z^r(1-z^{\nu+1})}{(1-z)^{r+1}}-\sum_{j=1}^r\binom{r}{j}[\nu+1]_j(r-j)!\frac{z^{\nu+1+r-j} } {(1-z)^{1+r-j}}
\]
\[
=r!\frac{z^r(1-z^{\nu+1})}{(1-z)^{r+1}}-z^\nu r!\sum_{j=1}^r \binom{\nu+1}{j}\left(\frac{z} {1-z}\right)^{r+1-j}.
\]
Using the identity
\[
1-e^{\ii\alpha}=\frac{2}{\ii}\sin\left(\frac{\alpha}{2}\right)e^{\frac{\ii\alpha}{2}}
\]
we deduce
\[
\frac{z}{1-z}=\frac{\ii e^{\frac{\ii t}{2\nu}}}{2\sin\left(\frac{t}{2\nu}\right)},
\]
and
\[
D_{\nu, r}(t)=\ii^r\frac{r!}{2^r} \frac{ \sin(\frac{(\nu+1)t}{2\nu}) e^{\frac{\ii(\nu+1+2r)t}{2\nu} }}{ \sin^{r+1}(\frac{t}{2\nu})e^{\frac{(r+1)\ii t}{2\nu} } }-e^{\ii t}r!\sum_{j=1}^r\ii^{r+1-j}\binom{\nu+1}{j}\frac{e^{\ii(r+1-j)\theta}}{(2\sin\theta)^{r+1-j}}
\]
\[
=\ii^rr!\left(\frac{2\sin\left(\frac{(\nu+1)t}{2\nu}\right)}{\left(2\sin\theta\right)^{r+1}}\cdot   e^{\frac{\ii(\nu+r)t}{2\nu} }-e^{\ii t}\sum_{j=1}^r\ii^{1-j}\binom{\nu+1}{j}\left(\frac{e^{\ii\theta}}{2\sin\theta}\right)^{1+r-j}\,\right).
\]
Multiplying  both  sides of the above equality by $\left(\frac{t}{\nu}\right)^{r+1}$ we get (\ref{eq: sharp}).
\end{proof}
Lemma \ref{lemma: sharp} coupled with the fact that the function $f(\theta)$ is bounded on $[0,\frac{\pi}{2}]$  yield the following estimate.
\begin{equation}
\frac{t^{r+1}}{\nu^{r+1}} D_{\nu,r}(t)=O(1),\;\;\forall \nu,\;\;0\leq t\leq \pi\nu. 
\label{eq: est1}
\end{equation}
Using (\ref{eq: est1}  )and the identity $S(r,1)=1$ in (\ref{eq: cov3})  we deduce that there exists $K=K_r>0$ such that for any $\nu>0$ and any $t\in [0,\pi\nu]$ we have
\[
\left|\frac{t^{r+1}}{\nu^{r+1}}\Bigl(\,C_{\nu,r}(t)-D_{\nu,r}(t)\,\Bigr)\,\right|\leq K_r\sum_{j=0}^{r-1}\left|\frac{t^{r+1}}{\nu^{r+1}} D_{\nu,j}(t)\right|\leq K_r\sum_{j=0}^{r-1}\left(\frac{t}{\nu}\right)^{r-j}\leq K_r\frac{t}{\nu},
\]
so that
\begin{equation}
\left|\frac{1}{\nu^{r+1}}\Bigl(\,C_{\nu,r}(t)-D_{\nu,r}(t)\,\Bigr)\,\right|\leq K_r\frac{1}{\nu t^r}
\label{eq: est2}
\end{equation}
Using Lemma \ref{lemma: sharp} we deduce
\begin{equation}
\lim_{\nu\ra \infty} \frac{t^{r+1}}{\nu^{r+1}}\re D_{\nu, r}(t)=I_r(t):=\ii^rr!\left( 2 \sin\left(\frac{t}{2}\right)\cdot   e^{ \frac{\ii t}{2} }-e^{\ii t}\sum_{j=1}^r\ii^{1-j}\frac{t^j}{j!}\right).
\label{eq:  ir}
\end{equation}
uniformly  for $t$ on compacts.  The estimate  (\ref{eq: est2}) implies that
\[
\lim_{\nu\ra \infty} \frac{t^{r+1}}{\nu^{r+1}}\re D_{\nu, r}(t)= I_r(t).
\]
We have the following crucial estimate whose proof can be found in Appendix \ref{s: d}. 
\begin{lemma} For every $r\geq 0$ there exists  $C_r>0$ such that  for any $\nu>0$ we have
\[
\left|\frac{1}{\nu^{r+1}}  D_{\nu, r}(t)- \frac{1}{t^{r+1}} I_r(t)\,\right|\leq \frac{C_r}{\nu}\frac{t^{r+1}-1}{t^r(t-1)},\;\;\forall 0<t\leq \pi\nu.\proofend
\]
\label{lemma: sharp2}
\end{lemma}
Using Lemma \ref{lemma: sharp2}  in (\ref{eq: est2}) we deduce
\begin{subequations}

\begin{equation}
\left|\frac{1}{\nu^{r+1}}  C_{\nu, r}(t)- \frac{1}{t^{r+1}} I_r(t)\,\right|\leq \frac{C_r}{\nu}\frac{t^{r+1}-1}{t^r(t-1)},\;\;\forall 0<t\leq \pi\nu.
\label{eq: est3}
\end{equation}
\begin{equation}
I_r(t) =t^{r+1} C_r(t)=\int_0^t\tau^r e^{\ii\tau} d\tau.
\label{eq: est3b}
\end{equation}
\end{subequations}
Using (\ref{eq: est0}) and (\ref{eq: est3}) we deduce that for any nonnegative integer $r$ there exists a positive constant $K=K_r>0$ such that
\begin{equation}
\left| C_{\nu,r}(t)-C_r(t)\right|\leq \frac{K_r}{\nu}\frac{t^{r+1}-1}{t^{r}(t-1)}.
\label{eq: est4}
\end{equation}
Coupling the above estimates with (\ref{eq: est0}) we deduce
\begin{equation}
C_{\nu,r}(t)= C_r(t) + O\left(\frac{1}{\nu}\right),\;\;\forall 0\leq t \leq \nu,
\label{eq: est5}
\end{equation}
where the constant implied by the symbol $O$ depends on $k$, but it is independent of $\nu$.  The last equality coupled with  (\ref{eq: cov-func}) implies that
\begin{equation}
R_\nu^{(k)}(t)=R_\infty^{(k)}(t)+O\left(\frac{1}{\nu}\right),\;\;\forall 0\leq t\leq \nu.
\label{eq: cov-func2}
\end{equation}
We deduce that, for any $t>0$ we have 
\[
\lim_{\nu\ra \infty}f_\nu(t)= f_\infty(t)
\]
where  $f_\nu$ is the function   defined in (\ref{eq: fnu}), while
\begin{equation}
f_\infty(t)=\frac{(\bar{\lambda}_0^2-R_\infty^2)\bar{\lambda}_2-\bar{\lambda}_0(R_\infty')^2}{(\bar{\lambda}_0^2-R_\infty^2)^{\frac{3}{2}}} \left(\sqrt{1-\rho_\infty^2}+{\rho_\infty}\arcsin\rho_\infty\right),
\label{eq: finfty}
\end{equation}
where
\[
\bar{\lambda}_0=\lambda_0(\infty)=\lim_{\nu\ra\infty}\lambda_0(\nu)=\frac{1}{3},\;\;\bar{\lambda}_2=\lim_{\nu\ra\infty}\lambda_2(\nu)=\frac{1}{5},
\]
\[
\rho_\infty(t)=\lim_{\nu\ra\infty}\rho_\nu(t)=\frac{R_\infty''(\bar{\lambda}_0^2-R_\infty^2)+(R'_\infty)^2R_\infty}{(\bar{\lambda}_0^2-R_\infty^2)\bar{\lambda}_2-\bar{\lambda}_0(R_\infty')^2}.
\]
We have the following result whose proof can be found in Appendix \ref{s: d}.
\begin{lemma}  
\begin{subequations}
\begin{equation}
|R_\infty(t)|<R_\infty(0), |R''_\infty(t)|< |\Ri''(0)|,\;\;\forall t>0,
\label{eq: diffa1}
\end{equation}
\begin{equation}
\Ri(t),\;\;\Ri'(t),\;\;\Ri''(t)=O\left(\frac{1}{t}\right) \;\;\mbox{as}\;\;t\ra \infty.
\label{eq: diffb1}
\end{equation}
\begin{equation}
(\bar{\lambda}_0^2-R_\infty^2)\bar{\lambda}_2-\bar{\lambda}_0(R_\infty')^2>0,\;\;\forall t>0.
\label{eq: diffc1}
\end{equation}
\begin{equation}
R_\infty(t) =\frac{1}{3}-\frac{1}{10}t^2 +\frac{1}{168}t^4+O(t^6),\;\; \Ri'(t)=-\frac{1}{5}t+ \frac{1}{42}t^3+ O(t^5),\;\;\mbox{as}\;\;t\ra 0.
\label{eq: diffd1}
\end{equation}
\end{subequations}
\qed
\label{lemma: cov-func}
\end{lemma}
We set
\[
\delta(t):=\max(t,1),\;\;t\geq 0.
\]
We find it convenient to introduce new  functions
\[
G_\nu(t):= \frac{1}{R_\nu(0)}R_\nu(t)=\frac{1}{\lambda_0(\nu)}R_\nu(t),\;\;H_\nu(t)= \frac{1}{R''_\nu(0)} R_\nu''(t)=-\frac{1}{\lambda_2(\nu)} R_\nu''(t).
\]
Using these notations  we  can rewrite (\ref{eq: diffc1}) as
\begin{equation}
\underbrace{(1-G_\infty^2)-\frac{\bla_0}{\bla_2}(G_\infty')^2}_{=:\eta(t)}>0,\;\;\forall t>0.
\label{eq: diffc2}
\end{equation}
The equalities (\ref{eq: diffd1}) imply that
\begin{equation}
\eta(t)=\frac{3}{4375}t^4+ O(t^6),\;\;\forall |t|\ll 1.
\label{eq: eta4}
\end{equation}
Then
\[
f_\nu(t)=\frac{\lambda_2(\nu)}{\lambda_0(\nu)}\,\times \eC_\nu(t)\times\left(\sqrt{1-\rho_\nu^2}+{\rho_\nu}\arcsin\rho_\nu\right),
\]
where
\medskip
\[
\eC_\nu(t):=\frac{(1-G_\nu(t)^2)-\frac{\lambda_0(\nu)}{\lambda_2(\nu))} (G_\nu'(t))^2}{(1-G_\nu(t)^2)^{3/2}},
\]
and
\medskip
\[
\rho_\nu=\frac{R_\nu''(\lambda_0(\nu)^2-R_\nu^2)+(R'_\nu)^2R_\nu}{(\lambda_0^2(\nu)-R_\nu^2)\lambda_2(\nu)-{\lambda_0}(R_\nu')^2}=\frac{-H_\nu(1-G_\nu^2)+\frac{\lambda_0(\nu)}{\lambda_2(\nu)}(G'_\nu)^2 G_\nu}{(1-G_\nu^2)-\frac{\lambda_0(\nu)}{\lambda_2(\nu)}(G_\nu')^2}.
\]
\begin{lemma} Let $\kappa \in (0,1)$. Then  
\[
\eC_\nu(t)=O(t),\;\;\forall 0\leq t\leq \nu^{-\kappa},
\]
 where the constant implied  by $O$-symbol  is independent of $\nu$ and $t\in [0,\nu^{-\kappa})$, but it could depend on $\kappa$.
 \label{lemma: small}
 \end{lemma}

\begin{proof}  Observe  that for $t\in [0,\nu^{-\kappa}]$  we  have
\[
G_\nu(t)= 1-\frac{\lambda_2(\nu)}{2\lambda_0(\nu)} t^2 + O(t^4),\;\;G'_\nu(t)= -\frac{\lambda_2(\nu)}{\lambda_0(\nu)} t + O(t^3)
\]
so that
\[
\bigl(\,1-G_\nu(t)^2\,\bigr)^{-3/2}= \left(\frac{\lambda_2(\nu)}{\lambda_0(\nu)}t\right)^{-3}\times (1+O(t)\,),
\]
and
\[
(1-G_\nu(t)^2)-\frac{\lambda_0(\nu}{\lambda_2(\nu)}(G_\nu')=O(t^4)
\]
so that  $\eC_\nu(t)=O(t)$.
\end{proof}

\begin{lemma} Let $\kappa\in (0,1)$. Then
\begin{equation}
\eC_\nu=\eC_\infty\times\left( 1+O\left(\frac{(G'_\infty)^2}{\nu\eta(t)}+ \frac{1}{\nu\gamma(t)\delta(t)}+\frac{1}{\nu\delta(t)\eta(t)}+\frac{1}{\nu^2\eta(t)}\right)\,\right),
\label{eq: est22.4}
\end{equation}
and
\begin{equation}
\eC_\nu(t)=\eC_\infty(t)+ O\left(\frac{(G'_\infty)^2}{\nu\gamma(t)^{3/2}}+\frac{\eta(t)}{\nu\delta(t)\gamma(t)^{3/2}}+\frac{1}{\nu\delta(t)\gamma(t)^{3/2}}+\frac{1}{\nu^2\gamma(t)^{3/2}}\,\right),
\label{eq: est22.5}
\end{equation}
where
\[
\eta(t):=(1-G_\nu^2)-\frac{\bla_0}{\bla_2}(G_\nu')^2\;\;\mbox{and}\;\; \gamma(t)=1-G_\infty(t)^2.
\]
\label{lemma: large}
\end{lemma}

\begin{proof} Observe that
\begin{equation}
G_\nu^2= \left(G_\infty+O\left(\frac{1}{\nu}\right)\,\right)^2\stackrel{(\ref{eq: diffb1})}{=}  G^2_\infty+O\left(\frac{1}{\nu\delta(t)}\right),
\label{eq:  est20}
\end{equation}
so that
\[
1-G_\infty(t)^2(t)= \bigl(1-G_\infty(t)^2\,\bigr)\left(1 +O\left(\frac{1}{\nu\delta(t)\gamma(t)}\right)\,\right).
\]
For $t>\nu^{-\kappa}$ we have
\[
\frac{1}{\nu\delta(t)\gamma(t)}=O\left(\frac{1}{\nu^{1-\kappa}}\right)=o(1)\;\;\mbox{uniformly in $t>\nu^{-\kappa}$ as $\nu\ra \infty$}.
\]
Hence
\begin{equation}
\left(1-G_\nu^2\right)^{-3/2}=\left(1-G_\infty^2\right)^{-3/2}\left(1+ O\left(\frac{1}{\nu\delta(t)\gamma(t)}\right)\,\right),\;\;\gamma(t):=1-G^2_\infty(t).
\label{eq: est21}
\end{equation}
Next observe that
\begin{equation}
\lambda_0(\nu) =\frac{B_3(\nu+1)}{3\nu^3}= \frac{1}{3}+\nu^{-1}+O(\nu^{-2}),\;\;\lambda_2(\nu) =\frac{B_5(\nu+1)}{5\nu^5}= \frac{1}{5}+\nu^{-1}+O(\nu^{-2})
\label{eq: bernq}
\end{equation}
and
\[
\frac{\lambda_0(\nu)}{\lambda_2(\nu)}=\frac{\frac{1}{3}+\nu^{-1}+O(\nu^{-2})}{\frac{1}{5}+\nu^{-1}+O(\nu^{-2})}=\frac{5}{3} -\frac{10}{3}\nu^{-1}+O(\nu^{-2}).
\]
Using  (\ref{eq: est20}) and the above  estimate we  deduce 
\begin{equation}
\begin{split}
(1-G_\nu(t)^2)-\frac{\lambda_0(\nu)}{\lambda_2(\nu))} (G_\nu')^2= \underbrace{(1-G_\infty^2)-\frac{\bla_0}{\bla_2}(G_\infty')^2}_{\eta(t)}+\frac{10}{3\nu}(G_\infty')^2+ O\left(\frac{1}{\nu\delta(t)}\right)+O\left(\frac{1}{\nu^2}\right)\\
= \left(\,(1-G_\infty^2)-\frac{\bla_0}{\bla_2}(G_\infty')^2\right)\cdot\left(1+O\left(\frac{(G'_\infty)^2}{\nu\eta(t)}+\frac{1}{\nu\delta(t)\eta(t)}+\frac{1}{\nu^2\eta(t)}\right)\,\right) .
\end{split}
\label{eq: nu-large}
\end{equation}
Using (\ref{eq: est21})   we deduce (\ref{eq: est22.4}). The estimate (\ref{eq: est22.5}) follows   (\ref{eq: est22.4}) by invoking the definitions  of $\gamma(t)$ and $\eta(t)$.
\end{proof}

\begin{lemma} Let $\kappa\in (0,\frac{1}{4})$. Then
\begin{equation}
\rho_\nu=\rho_\infty+O\left(\frac{1}{\nu\eta(t)}+\frac{(G'_\infty)^2}{\nu\eta(t)\delta(t)}+\frac{1}{\nu\delta(t)^2}+\frac{1}{\nu\delta(t)^2\eta(t)}+\frac{1}{\nu^2\eta(t)\delta(t)}\,\right),\;\;\forall t>\nu^{-\kappa},
\label{eq: est23}
\end{equation}
where the constant implied  by $O$-symbol  is independent of $\nu$ and $t>\nu^{-\kappa}$, but it could depend on $\kappa$.
\label{lemma: rho-larg}
\end{lemma} 

\begin{proof} The estimates (\ref{eq: diffb1}), (\ref{eq: diffd1})  and (\ref{eq: eta4}) imply that for $t>\nu^{-\kappa}$ we have 
\[
\frac{(G'_\infty)^2}{\nu\eta(t)}+\frac{1}{\nu\delta(t)\eta(t)}+\frac{1}{\nu^2\eta(t)}=O(\frac{1}{\nu^{1-4\kappa}})=o(1),
\]
uniformly in $t>\nu^{-\kappa}$. We conclude from (\ref{eq: nu-large}) that
\begin{equation}
\begin{split}
\Bigl(\,(1-G_\nu(t)^2)-\frac{\lambda_0(\nu)}{\lambda_2(\nu))} (G_\nu')^2\,\Bigr)^{-1}&= \left(\,\bigl(1-G_\infty^2\bigr)-\frac{\bla_0}{\bla_2}(G_\infty')^2 \,\right)^{-1}\\
&\times \left(1+O\left(\frac{(G'_\infty)^2}{\nu\eta(t)}+\frac{1}{\nu\delta(t)\eta(t)}+\frac{1}{\nu^2\eta(t)}\right)\,\right) .
\end{split}
\label{eq: est22}
\end{equation}
Since  $H_\nu=H_\infty+O(\nu^{-1})$ we deduce
\[
-H_\nu(1-G_\nu)^2+\frac{\lambda_0(\nu)}{\lambda_2(\nu)}(G_\nu')^2G_\nu=  -H_\infty(1-G_\infty)^2+\frac{\bla_0}{\bla_2}(G_\infty')^2G_\infty+O\left(\frac{1}{\nu}\right)
\]
Recalling that 
\[
\rho_\nu=\frac{-H_\nu(1-G_\nu^2)+\frac{\lambda_0(\nu)}{\lambda_2(\nu)}(G'_\nu)^2 G_\nu}{(1-G_\nu^2)-\frac{\lambda_0(\nu)}{\lambda_2(\nu)}(G_\nu')^2}
\]
and $-H_\infty(1-G_\infty)^2+\frac{\bla_0}{\bla_2}(G_\infty')^2G_\infty=O(\delta^{-1})$ we see that  (\ref{eq: est23}) follows from ({eq: est22}).
\end{proof}

Consider the function
\[
A(u)= \sqrt{1-u^2}+u \arcsin u,\;\;|u|\leq 1.
\]
Observe that
\begin{subequations}
\begin{equation}
\frac{dA}{du}=O(1),\;\;\forall |u|\leq 1,
\label{eq: asina}
\end{equation}
\begin{equation}
\frac{dA}{du}=\arcsin u = O(u) \;\;\mbox{as $u\ra 0$}.
\label{eq: asinb}
\end{equation}
\end{subequations}
Now fix an exponent $\kappa\in (0,\frac{1}{4})$.     We discuss separately two cases.

\medskip

\noindent {\bf 1.} If $\nu^k<t<\pi\nu$.  Then in this range we have
\[
\frac{1}{\delta},\;\;\rho_\infty=O(t^{-1}),\;\; \rho_\nu=\rho_\infty +o(1),\;\;\frac{1}{\eta(t)}=O(1),
\]
and  using (\ref{eq: asinb}) and (\ref{eq: est23}) we  deduce
\begin{equation}
A(\rho_\nu)= A(\rho_\infty)+ A'(\rho_\infty)(\rho_\nu-\rho_\infty)+ O\bigl(\,(\rho_\nu-\rho_\infty)^2\,\bigr)= A(\rho_\infty) +O\left(\frac{1}{t\nu}+\frac{1}{\nu^2}\,\right).
\label{eq: est24}
\end{equation}

\noindent {\bf 2.} $\nu^{-\kappa}< t<\nu^{\kappa}$.  The equality (\ref{eq: eta4}) shows that in this range we have
\[
\frac{1}{\eta(t)}=O(\nu^{-4\kappa}),\;\;\frac{1}{\delta(t)}=O(1)
\]
so that and (\ref{eq: est23}) implies that
\[
\rho_\nu-\rho_\infty= O\left(\frac{1}{\nu^{1-4k}}\right).
\]
Using (\ref{eq: asina})  we  deduce
\begin{equation}
A(\rho_\nu)-A(\rho_\infty)=  O\left(\frac{1}{\nu^{1-4k}}+\frac{1}{\nu^{1-4\kappa}\delta(t)^2}+\frac{1}{\nu^{2-8\kappa}\delta(t)}\right).
\label{eq:  est25}
\end{equation}
Set
\[
\Delta_\nu:=\eC_\nu A(\rho_\nu)-\eC_\infty A(\rho_\infty),\;\; q_\nu:=\left(\frac{\lambda_2(\nu)}{\lambda_0(\nu)}\right),\;\;q_\infty=\left(\frac{\bla_2}{\bla_0}\right).
\]
Then using (\ref{eq: bernq}) we deduce that
\[
q_\nu-q_\infty= \frac{6}{5}\nu^{-1}+O(\nu^{-2}),\;\;q_\nu=q_\infty\left( 1+2\nu^{-1}+O(\nu^{-2})\,\right).
\]
Then
\[ 
\Bigl(\,f_\nu(t)-q_\nu\,\Bigr)- \Bigl(\,f_\infty(t)-q_\infty\,\Bigr)=
\]
\[
=q_\nu\bigl(\,\eC_\infty A_\infty(\rho_\infty)-1 +\Delta_\nu\,\bigr)-q_\infty\bigl(\,\eC_\infty A(\rho_\infty)-1\bigr)
\]
\[
=\bigl(q_\nu-q_\infty)\bigl(\,\eC_\infty A(\rho_\infty)-1\bigr)+ q_\infty\Delta_\nu
\]
To prove (\ref{eq: varinfi})  we need to prove the following equality.
\begin{equation}
\bigl(q_\nu-q_\infty)\int_0^{\pi\nu} \bigl(\,\eC_\infty(t) A(\rho_\infty(t))-1\bigr)dt,\;\; q_\infty\int_0^{\pi\nu} \Delta_\nu(t) dt =o(1)\;\;\mbox{as $\nu\ra\infty$}.
\label{eq: finally}
\end{equation}
We can dispense easily of the first integral above since  $\eC_\infty(t) A(\rho_\infty(t))-1$ is absolutely integrable on $[0,\infty)$ and $q_\nu-q_\infty=O(\nu^{-1})$.

The second integral requires a bit of work. More precisely, we will show the following result.

\begin{lemma}If $0< \kappa <\frac{1}{5}$, then 
\begin{equation}
\int_0^{\nu^{-\kappa}} \Delta_\nu(t) dt,\;\;\int_{\nu^{-\kappa}}^{\nu^\kappa} \Delta_\nu(t) dt,\;\;\int_{\nu^\kappa}^{\pi\nu} \Delta_\nu(t) dt=o(1)\;\;\mbox{as $\nu\ra\infty$}.
\label{eq: final1}
\end{equation}
\label{lemma: final1}
\end{lemma}

\begin{proof} We will discuss each of the three cases separately.

\medskip

\noindent {\bf 1.} $0<t <\nu^{-\kappa}$.  The easiest way to prove that $\int_0^{\nu^{-\kappa}} \Delta_\nu(t) dt\ra 0$ is to show that
\[
\eC_\nu(t)A(\rho_\nu(t))=O(1),\;\;0<t<\nu^{-k}.
\]
This follows using Lemma \ref{lemma: small} and observing  that  the  function $A(u)$ is bounded. 

\medskip

\noindent {\bf 2.} $\nu^{-\kappa}<t<\nu^\kappa$.  In this range we have
\[
\frac{1}{\eta}=O(\nu^{-4\kappa}),\;\; \frac{1}{\gamma(t)}= O(\nu^{-2\kappa}),\;\;\frac{1}{\delta(t)}= O(1),\;\;G'_\infty(t)=O(1).
\]
Using  (\ref{eq: est22.5}), (\ref{eq:  est25})  we deduce
\[
\eC_\nu(t)=\eC_\infty(t)+O\left(\frac{1}{\nu^{1-3\kappa}}\right),\;\;A(\rho_\nu)=A(\rho_\infty)+O\left(\frac{1}{\nu^{1-4\kappa}}\right).
\]
Hence
\[
\Delta_\nu =O\left(\frac{1}{\nu^{1-4\kappa}}\right)\;\;\mbox{and}\;\;\int_{\nu^{-\kappa}}^{\nu^\kappa}\Delta_\nu(t) dt= O\left(\frac{1}{\nu^{1-5\kappa}}\right)=o(1).
\]

\medskip

\noindent {\bf 3.} $\nu^\kappa <t<\pi\nu$.  For these values of $t$ we have
\[
\eta(t),\frac{1}{\eta(t)} =O(1),\;\;\frac{1}{\delta(t)}=O(t^{-1}).
\]
Using  (\ref{eq: est22.5}) and (\ref{eq:  est24})   we deduce
\[
\eC_\nu(t)=\eC_\infty(t)+ O\left(\frac{1}{\nu t}+\frac{1}{\nu^2}\,\right),\;\; A(\rho_\nu)= A(\rho_\infty)+  O\left(\frac{1}{\nu t}+\frac{1}{\nu^2}\,\right)
\]
so  that
\[
\Delta_\nu(t)= O\left(\frac{1}{\nu t}+\frac{1}{\nu^2}\,\right),\;\;\int_{\nu^{-\kappa}}^{\pi\nu}\Delta_\nu(t)\,dt= O\left(\frac{1}{\nu}+\frac{\log\nu}{\nu}\right)=o(1).
\]
\end{proof}

The fact that $\delta_\infty$  defined  as in (\ref{eq: varinfi}) is positive follows  by arguing exactly as  in \cite[\S 3.2]{GW}.  This completes the proof of  Theorem \ref{th: var}.
\end{proof}

\begin{remark} The proof of Lemma \ref{lemma: final1} shows that  for any  $\ve>0$ we have
\begin{equation}
\var(Z_\nu)= \nu \delta_\infty+ O(\nu^\ve)\;\;\mbox{as $\nu\ra \infty$}.
\label{eq: varinfi2}
\end{equation}
Numerical experiments suggest that $\delta_\infty\approx 0.35$.\qed
\end{remark}

\appendix

\section{Some elementary integrals}
\label{s: a}
\setcounter{equation}{0}

Suppose $\bsW$ is an oriented Euclidean   vector space       equipped with an orthogonal decomposition
\[
\bsW=\bsW_0\oplus \bsW_1,\;\;\dim \bsW_i=n_i,\;\;i=0,1,\;\;\dim\bsW=n=n_0+n_1.
\]
For any $\bw \in \bsW$ we denote by $\bw_i$ its orthogonal projection on $\bsW_i$, $i=0,1$, so that $\bw=\bw_0+\bw_1$. We set $r_i(\bw):=|\bw_i|$, $i=0,1$.

\begin{lemma} Let $\vfi_i: \bsW_i\ra \bR$, $i=0,1$, be locally integrable  functions,  such that $\vfi_0$ is positively homogeneous of degree $k_0\geq 0$, and set $\vfi(\bw_0,\bw_1):=\vfi_0(\bw_0)\vfi_1(\bw_1)$.
Then
\begin{equation}
\int_{S(\bsW)}  \vfi(\bw)\,|dS(\bw)|=\int_{S(\bsW_0)} \vfi_0(\bw_0) \,|dS(\bw_0)|\times\int_{B(\bsW_1)}   \vfi_1(\bw_1)(1-r_1^2)^{\frac{k_0+n_0-2}{2}}\,|dV(\bw_1)|.
\label{eq: int0}
\end{equation}
In particular, if $\vfi_0=1$, then
\begin{equation}
\int_{S(\bsW)}  \vfi(\bw)\,|dS(\bw)|=\bsi_{n_0-1}\int_{B(\bsW_1)}   \vfi(\bw_1)(1-r_1^2)^{\frac{n_0-2}{2}}\,|dV(\bw_1)|.
\label{eq: int1}
\end{equation}
\label{lemma: int1}
\end{lemma}

\begin{proof} The key trick behind the equality (\ref{eq: int1}) is the \emph{co-area formula}.    Denote by $\pi$ the  orthogonal projection onto $\bsW_1$, i.e.,
\[
\pi:\bsW \ra \bsW_1,\;\;\bw_0+\bw_1\mapsto \bw_1.
\]
This induces a smooth map $\pi: S(\bsW)\ra B(\bsW_1)$. We denote by $\Sigma_{\bw_1}$ the fiber of this map over $\bw_1$. We observe that $\Sigma_{\bw_1}$ is the  sphere in $\bsW_0$  of radius $(1-r_1^2)^{1/2}$ and center $0$. Denote by $J_\pi: S(\bsW)\ra \bR$ the relative Jacobian of the map $\pi$ defined as in \cite[\S 5.1.1]{KP}.

Fix a point $\bw=(\bw_0,\bw_1)\in \Sigma_{\bw_1}$.  Next, choose an orthonormal  basis $\be_1,\dotsc, \be_{n_1}$ of $\bsW_1$ such that $\be_1= \frac{1}{r_1}\bw_1$.

The orthogonal complement $N_\bw\Sigma_{\bw_1}$  of $T_\bw \Sigma_{\bw_1}$  in $T_\bw S(\bsW)$ consists of vectors     that  are orthogonal on $\bsW_0$ and on the unit vector $\bw=\bw_0+\bw_1$. We deduce  that the collection
\[
\bsf_1= -r_1^2w_0+r_0^2\bw_1,\;\;\bsf_2=\be_2,\dotsc,\bsf_{n_1}=\be_{n_1},
\]
is an orthogonal basis   of  $N_\bw \Sigma_{\bw_1}$.   Note that
\[
|\bsf_1|= r_1^4r_0^2+r_0^4r_1^2= r_1^2r_2^2.
\]
We obtain an orthonormal basis by replacing $\bsf_1$ with the vector
\[
\bar{\bsf}_1=\frac{1}{|\bsf_1|}\bsf_1= -\frac{r_1}{r_0}\bw_0+ \frac{r_0}{r_1}\bw_1.
\]
The orthogonal projection onto $\bsW_1$ of the  orthonormal basis $\bar{\bsf}_1,\bsf_2,\dotsc,\bsf_{n_1}$ is the orthogonal basis
\[
r_0\be_1,\be_2,\dotsc,\be_{n_1},
\]
whose determinant is $r_0$. This shows that
\[
J_\pi(\bw) = r_0(\bw)=(1-r_1(\bw)^2)^{1/2},\;\;\forall \bw\in S (\bsW).
\]
The coarea formula \cite[Thm. 5.3.9]{KP} implies that
\[
\int_{ S(\bsW)} \vfi(\bw)\,|dS(\bw)|=\int_{B(\bsW_1)}\left(\int_{\Sigma_{\bw_1}} \frac{1}{J_\pi(\bw_0,\bw_1)} \vfi(\bw_0,\bw_1)\,  |dS(\bw_1)|\,\right)
\]
\[
=\Bigl(\int_{S(\bsW_0)} \vfi_0(\bw_0)\,|dS(\bw_0)|\Bigr)\int_{B(\bsW_1)}   \vfi_1(\bw_1)r_0^{k_0+n_0-2}\,|dV(\bw_1)|
\]
\[
=\Bigl(\int_{S(\bsW_0)} \vfi_0(\bw_0)\,|dS(\bw_0)|\Bigr) \int_{B(\bsW_1)}   \vfi_1(\bw_1)(1-r_1^2)^{\frac{k_0+n_0-2}{2}}\,|dV(\bw_1)|.
\]
\end{proof}

Suppose that $\bsL$ is a  Euclidean vector space of dimension $\ell$, and $Q:\bsL\ra \bR$ is    a continuous, positively homogeneous function of degree $k>0$. For any  positive integer $n$ we set
\[
I_n(Q):=\int_{B(\bsL)} |Q(\bx)| (1-|x|^2)^{n/2}\,|dV(\bx)|, \;\; J_n(Q):=\int_{S(\bsL)} |Q(\bx)|\,|dS(\bx)|,
\]
where $S(\bsL)$ denotes the unit sphere in $\bsL$ centered at the origin, and $B(\bsL)$ denotes the unit ball in $\bsL$ centered at the origin.
\begin{lemma}
\begin{equation}
I_n(Q) =  \frac{ \Gamma(\frac{\ell+k}{2})\Gamma(\frac{n}{2}+1)}{2\Gamma(\frac{n+\ell+k}{2}+1)} J_n(Q).
\label{eq: int7}
\end{equation}
\label{lemma: int2}
\end{lemma}

\begin{proof}   We have
\[
I_n(Q)=J_n(Q)\int_0^1r^{\ell + k-1}(1-r^2)^{n/2} dr =\frac{1}{2} J_n(Q)\int_0^1 s^{\frac{\ell+k}{2}-1}(1-s)^{n/2} ds
\]
\[
=\frac{1}{2} J_n(Q) B\left(\,\frac{\ell+k}{2},\frac{n}{2}+1\right),
\]
where $B(p,q)$ denotes the Eulerian integral
\[
B(p,q)=\int_0^1 (1-s)^{p-1}s^{q-1} ds=\frac{\Gamma(p)\Gamma(q)}{\Gamma(p+q)}.
\]
\end{proof}

The integrals of homogeneous  functions over the unit sphere can be reduced to gaussian integrals of these polynomials.  More precisely we have the following result.

\begin{lemma}  Suppose  that $\bsW$ is an Euclidean space of dimension $N$ and $f: \bsW\ra \bR$ is a locally  integrable  positively homogeneous function of degree $\ell\geq 0$. Then
\begin{equation}
\int_{S(\bsW)} f(\bx)\,|dS(\bx)|= \frac{2}{\Gamma(\frac{N+\ell}{2})}\int_{\bsW} e^{-|\bx|^2} f(\bx)\,|dV(\bx)|,
\label{eq: gauss}
\end{equation}
or equivalently,
\begin{equation}
\frac{1}{{\rm area}\,S(\bsW)}\int_{S(\bsW)} f(\bx)\,|dS(\bx)|= \frac{\Gamma(\frac{N}{2})}{\Gamma(\frac{N+\ell}{2})}\int_{\bsW} f(\bx) \frac{e^{-|\bx|^2}}{ \pi^{ \frac{N}{2}} } \,|dV(\bx)|.
\label{eq: gauss1}
\end{equation}
\label{lemma: gauss}
\end{lemma}

\begin{proof}  We have
\[
\int_{\bsW} e^{-|\bx|^2} f(\bx)\,|dV(\bx)| =\int_0^\infty\left(\int_{S(\bsW)} f(\bx)\,|dS(\bx)|\right) e^{-r^2}r^{N+\ell-1} dr
\]
\[
= \left(\int_{S(\bsW)} f(\bx)\,|dS(\bx)|\right)\int_0^\infty e^{-s} s^{\frac{N+\ell}{2}-1} \frac{ds}{2}=\frac{\Gamma(\frac{N+\ell}{2})}{2} \int_{S(\bsW)} f(\bx)\,|dS(\bx)|.
\]
\end{proof}

\begin{proposition} Suppose $a$ and $b$ are nonnegative real numbers such that $a>b$. Then
\[
\bsI(a,b):=\int_{\bR^3} e^{- \frac{1}{2(a^2-b^2)}(ax^2+ay^2-2bxy)-\frac{1}{2}z^2} |xy-z^2||dxdydz|
\]
\[
=\sqrt{2\pi(a^2-b^2)}\left(\int_0^{2\pi}\frac{2c^{3/2}}{(c+2)^{1/2}}d\theta-2\pi a+2\pi\right),
\]
where $c(\theta):= (a-b\cos 2\theta)$. 
\label{prop: iab}
\end{proposition}

\begin{proof} Let  $\{ \ii,\jj,\kk\}$ be  the canonical  orthonormal basis of $\bR^3$.  Define a new orthonormal basis $\be_1,\be_2,\be_3$ of $\bR^3$ by setting
\[
\be_1=\frac{1}{\sqrt{2}}(\ii+\jj);\;\;\be_2=\frac{1}{\sqrt{2}}(\ii-\jj),\;\;\be_3=\kk.
\]
If  we let $(u,v,w)$ denote the  coordinates with respect to this new orthonormal frame, then from the equality
\[
u\be_1+v\be_2+w\be_3=x\ii+y\jj+z\kk
\]
we deduce
\[
z=w,\;\;x=\frac{1}{\sqrt{2}}(u+v),\;\;y=\frac{1}{\sqrt{2}}(u-v)
\]
\[
xy=\frac{1}{2}(u^2-v^2),\;\;  ax^2+ay^2-2bxy= (a-b)u^2+(a+b)v^2.
\]
We  deduce that 
\[
\bsI=\frac{1}{2}\int_{\bR^3} e^{-\frac{1}{2(a^2-b^2)}((a-b)u^2+(a+b)v^2)-\frac{1}{2}w^2}|u^2-v^2-2w^2|\,|dudvdw|
\]
We now make the change in variables
\[
u= \sqrt{2(a+b)} u,\;\;v=\sqrt{2(a-b)}v,\;\;w=\sqrt{2} w,
\]
to deduce
\[
\bsI=2\sqrt{2(a^2-b^2)}\underbrace{\int_{\bR^3}e^{-(u^2+v^2+w^2)}|(a+b)u^2+(a-b)v^2-2w^2|\,|dudvdw|}_{=:\bsI_1}.
\]
We now change to  cylindrical coordinates,
\[
w=w,\;\;u=r\cos\theta,\;\;v=r\sin\theta,
\]
so that
\[
u^2+v^2+w^2=r^2+w^2,\;\; 
\]
\[
(a+b)u^2+(a-b)v^2-2w^2= r^2\bigl(  (a+b)\cos^2\theta+(a-b)\sin^2\theta\,\bigr) -2w^2
\]
\[
=r^2(a-b\cos 2\theta) -2w^2.
\]
We have
\[
\bsI_1=\int_0^{2\pi}d\theta\int_{-\infty}^\infty dw\int_0^\infty e^{-r^2}\bigl| r^2(a-b\cos 2\theta)-2w^2\,\bigr| rdr
\]
\[
=\frac{1}{2}\int_0^{2\pi}d\theta\int_{-\infty}^\infty e^{-w^2}dw\int_0^\infty e^{-s}\bigl| \,s(a-b\cos 2\theta)-2w^2\,\bigr| ds
\]
At this point  we observe that for any $c,d>0$ we have
\[
\int_0^\infty e^{-x} |cx-d| dx=2ce^{-\frac{d}{c}} +d-c.
\]
Hence, if we set $c=c(\theta)=(a-b\cos 2\theta)$ we deduce
\[
\bsI_1=\frac{1}{2}\int_0^{2\pi}d\theta\underbrace{\int_{-\infty}^\infty e^{-w^2} \bigl(\, 2ce^{-\frac{2w^2}{c} }+2w^2-c\,\bigr) dw}_{=:\bsJ(c)}.
\]
We have
\[
\bsJ(c)=2c\int_\bR e^{-\frac{c+2}{c}w^2} dw +2\int_{\bR} e^{-w^2} w^2 dw -c\int_\bR e^{-w^2} dw
\]
\[
=\frac{2c^{3/2}}{(c+2)^{1/2}}\pi^{1/2} +2\Gamma(3/2) -c\pi^{1/2}=\pi^{1/2}\left(\frac{2c^{3/2}}{(c+2)^{1/2}}-c+1\right).
\]
We deduce that
\[
\bsI= \sqrt{2(a^2-b)^2}\int_0^{2\pi}\bsJ(c) d\theta=\sqrt{2\pi(a^2-b)^2}\int_0^{2\pi} \left(\frac{2c^{3/2}}{(c+2)^{1/2}}-c+1\right)d\theta.
\]
The conclusion of the proposition follows by observing that $\int_0^{2\pi} c(\theta) d\theta=a$.
\end{proof}

\section{Basic facts about spherical harmonics}
\label{s: b}
\setcounter{equation}{0}

We survey here a few classical facts about spherical harmonics that we needed in the main body of the paper. For proofs and more details we refer to our main source,  \cite{Mu}.

We denote by  $\eH_{n,d}$ the space of homogeneous, harmonic polynomials   of degree $n$ in $d$ variables.   We regard  such polynomials as functions on $\bR^d$, and    we denote by $\eY_{n,d}$ the subspace of $C^\infty(S^{d-1})$   spanned by the restrictions of these polynomials to the unit sphere. We have
\[
\dim\eH_{n,d}=\dim \eY_{n,d}= M(n,d)=\binom{ d+n-1}{n}-\binom{d+n-3}{n-2}
\]
\[
=\frac{2n+d-2}{n+d-2}\binom{n+d-2}{d-2}\sim 2\frac{n^{d-2}}{(d-2)!}\;\;\mbox{as}\;\;n\ra \infty.
\]
Observe that
\begin{equation}
M(0,d)=1,\;\; M(1,d)= d,\;\; M(2,d)= \binom{d+1}{2} -1.
\label{eq: leg0}
\end{equation}
The space $\eY_{n,d}$ is the  eigenspace of the Laplace operator on $S^{d-1}$ corresponding to the eigenvalue $\lambda_n(d)=n(n+d-2)$.

We want to   describe an inductive  construction of an orthonormal  basis of $\eY_{n,d}$.  We start with the case  $d=2$.    For any $m\in \bZ$,  we set
\[
\vfi_m(\theta)=\begin{cases}
\cos(m\theta), &m\leq 0\\
\sin(m\theta), & m>0.
\end{cases},\;\; t_m=\|\vfi_m\|_{L^2}=\begin{cases}
(2\pi)^{1/2}, &m=0\\
\pi^{1/2}, &m>0.
\end{cases},\;\;\Phi_m=\frac{1}{t_m}\vfi_m.
\]
Then $\eB_{0,2}=\{\Phi_0\}$ is an orthonormal basis of $\eY_{0,2}$, while $\eB_{n,2}= \{\Phi_{-n},\Phi_n\}$ is an orthonormal basis of $\eY_{n,2}$, $n>0$.

Assuming now that we have produced orthonormal bases $\eB_{n,d-1}$ of  all the spaces $\eY_{n,d-1}$,  we indicate how to produce  orthonormal     bases  in the  harmonic spaces $\eY_{n,d}$. This requires the introduction of the Legendre polynomials and their associated functions.

The \emph{Legendre polynomial  $P_{n,d}(t)$ of degree $n$ and order $d$}  is given by the Rodriguez formula
\begin{equation}
P_{n,d}(t)=(-1)^n R_n(d)(1-t^2)^{-\frac{d-3}{2}} \left(\dt\right)^n(1-t^2)^{n+\frac{d-3}{2}},
\label{eq: leg1}
\end{equation}
where $R_n(d)$ is the Rodriguez  constant
\[
R_n(d)= 2^{-n}\frac{\Gamma(\frac{d-1}{2})}{\Gamma(n+\frac{d-1}{2})}=2^{-n}\frac{1}{\left[n+\frac{d-3}{2}\right]_n},
\]
where we recall that $[x]_k:=x(x-1)\cdots (x-k+1)$. Equivalently, they can be defined recursively  via the relations
\[
P_{0,d}(t)=1,\;\;P_{1,d}(t)= t,
\]
\[(n+d-2)P_{n+1,d}(t)-(2n+d-2)tP_{n,d}(t)+ nP_{n-1,d}(t)=0,\;\;n>0.
\]
In particular, this shows that
\[
P_{2,d}(t)= \frac{1}{d-1}\bigl(\,  dt^2 -1\,\bigr).
\]
The Legendre polynomials are normalized by the equality
\[
P_{n,d}(1)=1,\;\;\forall d\geq 2,\;\;n\geq 0.
\]
More generally, for any $n>0$, $d\geq 2$, and  any $0<j\leq n$, we have
\[
P_{n,d}^{(j)}(1)=(-1)^nR_n(d)\binom{n+j}{j}\left\{\, \frac{D_t^n(1-t)^{n+\frac{d-3}{2}}}{(1-t)^{\frac{d-3}{2}}} \cdot\frac{D_t^j(1+t)^{n+\frac{d-3}{2}}}{(1+t)^{\frac{d-3}{2}}}\,\right\}_{t=1},\;\;D_t:=\dt,
\]
\[
=2^{n-j}R_n(d)\binom{n+j}{j}\left[n+\frac{d-3}{2}\right]_n\cdot \left[n+\frac{d-3}{2}\right]_j,
\]
which implies
\begin{equation}
P_{n,d}^{(j)}(1)= 2^{-j}\binom{n+j}{j}\left[n+\frac{d-3}{2}\right]_j.
\label{eq: leg3}
\end{equation}
 For any $d\geq 3$, $n\geq 0$ and $0\leq j\leq n$, we define the \emph{normalized associated Legendre functions}
 \[
 \hP_{n,d}^j(t):= C_{n,j,d}(1-t^2)^{\frac{j}{2}} P_{n,d}^{(j)}(t),
 \]
 where
 \begin{equation}
 C_{n,j,d}:=\frac{[n+d-3]_{d-3}}{\Gamma(\frac{d-1}{2})}\left(\frac{(2n+d-2)}{2^{d-2}[n+d+j-3]_{2j+d-3}}\right)^{1/2}.
 \label{eq: leg2}
 \end{equation}
When $d=3$,  the above formul{\ae} take the form
\begin{equation}
\hP_{n,j,3}(t)= \sqrt{\frac{(n+\frac{1}{2})(n-j)!}{(n+j)!}}(1-t^2)^{\frac{j}{2}} P_{n,3}^{(j)}(t).
\label{eq: leq23}
\end{equation}
For any $0\leq j\leq n$, and any $d>2$ we define  a linear map
\[
\eT_{n,j,d}:\eY_{j,d-1}\ra \eY_{n,d},\;\; Y\mapsto \eT_{n,j,d}[Y],
\]
\[
\eT_{n,j,d}[Y] (\bx)=  \hP^j_{n,d}(x_d)\,\cdot\, Y\left(\frac{1}{\|\bx'\|}\bx'\right),\;\;\forall \bx\in S^{d-1}, \bx'=(x_1,\dotsc,x_{d-1})\neq 0.
\]
Note that for $\bx=(\bx',x_d)\in S^{d-1}$ we have
\[
\|\bx'\|=\bigl(\,1-x_d^2\,\bigr)^{1/2}\;\;\mbox{and}\;\;\hP^j_{n,d}(x_d)= C_{n,j, d} (1-x_d^2)^{j/2}P_{n,d}^{(j)}(x_d)= C_{n,j,d} \|\bx'\|^j P_{n,d}^{(j)}(x_d),
\]
so that
\[
\eT_{n,j,d}[Y] (\bx)= C_{n,j,d}  P_{n,d}^{(j)}(x_d)\widetilde{Y}(\bx'),\;\;\forall \bx=(\bx',x_d)\in S^{d-1},
\]
where $\widetilde{Y}$ denotes the extension of $Y$ as a homogeneous polynomial of degree $j$ in $(d-1)$-variables. The  sets $\eT_{n,j,d} [\eB_{j,d-1}]$,  $0\leq j\leq n$ are  disjoint, and their union is an orthonormal basis of $\eY_{n,d}$ that we denote by $\eB_{n,d}$.

The space $\eY_{0,d}$ consists only of constant functions and $\eB_{0,d}=\{\, \bsi_{d-1}^{-\frac{1}{2}}\,\}$. The orthonormal basis $\eB_{1,d}$ of $\eY_{1,d}$  obtained via the above inductive process is
\begin{equation}
\eB_{1,d}= \bigl\{\, C_0 x_i,\;\;1\leq i\leq d\,\bigr\}=\bigl\{\bsi_{d-2}^{-\frac{1}{2}}C_{1,0,d}x_i;\;\;1\leq i\leq d\,\bigr\}.
\label{eq: leg4}
\end{equation}
The orthonormal basis  $\eB_{2,d}$ of $\eY_{2,d}$ is
\begin{equation}
C_1(dx_i^2-r^2), \;\;1\leq i <d,\;\;C_2x_ix_j,\;\;1\leq i <j\leq d,
\label{eq: harmdeg2}
\end{equation}
where $r^2= x_1^2+\cdots +x_d^2$, and the positive constants $C_0, C_1, C_2$ are found  from the   equalities
\[
C_0^2\int_{S^{d-1}} x_1^2\,|dS(\bx)|=C_1^2\int_{S^{d-1}} (d^2 x_1^4 -2dx_1^2+1)\, |dS(\bx)|=C_2^2\int_{S^{d-1}} x_1^2x_2^2 |dS(\bx)|=1,
\]
aided by the classical identities,  \cite[Lemma 9.3.10]{N0},
\begin{equation}
\int_{S^{d-1}}x_1^{2h_1}\cdots x_d^{2h_d}\, |dS(\bx)|=\frac{2\Gamma(\frac{2h_1+1}{2})\cdots \Gamma(\frac{2h_d+1}{2})}{\Gamma(\frac{2h+d}{2})},\;\;h=h_1+\cdots +h_d.
\label{eq: bclass}
\end{equation}

\section{Invariant integrals over the space of symmetric matrices}
\label{s: c}
\setcounter{equation}{0}

In the main body of the paper we  encountered many integrals     of the form
\[
\int_{\Sym_N}|\det  A|\, |d \bgamma(A)|,
\]
where $\Sym_N$ is  the  space of  symmetric $N\times N$ matrices, and $\bgamma$ is a Gaussian  probability measure on $\Sym_N$.  In this appendix, we   want  show that in certain cases  we can  reduce this integral  to an integral over a space of  much lower dimension  using a    basic trick in random matrix theory. We set
\[
D_N:=\dim \Sym_N=\binom{N+1}{2},\;\;\Sym_N^0:=\bigl\{  \,   S\in \Sym_N;\;\;\tr S=0\,\bigr\}.
\]
Note first that we have  a canonical  $O(N)$-invariant   metric $g_*$  on $\Sym_N$ with norm $|-|_*$ given by
\[
|A|_*:=\bigl(\,\tr A^2\,\bigr)^{1/2}.
\]
Using the canonical basis  of $\bR^k$ we can describe    each $A\in \Sym_N$ as a linear combination
\[
A= \sum_{i\leq j} a_{ij} H_{ij},
\]
where $H_{ij}$ is the symmetric  $k\times k$ matrix  whose $(i,j)$ and $(j,i)$ entries are $1$, while  the remaining entries are $0$.    With respect to the coordinates $(a_{ij})$ we have
\[
g_*= \sum_ida_i^2+2\sum_{i< j} da_{ij}^2.
\]
The collection  $(H_{ij})_{1\leq i\leq j\leq N}$  is an orthonormal basis with respect to the  metric $g_*$. The volume density $|d\eV|_*$  determined by  the metric $g_*$ has the description
\[
|d\eV_*|=2^{\frac{D_N-N}{2}}\left|\prod_{i\leq j}d a_{ij}\right|.
\]
 Via the metric  on $\bR^N$ we can identify $\Sym_N$  with the vector space of   homogeneous polynomials of degree $2$ in $N$-variables. More precisely, to such a polynomial $P$  we associate  the matrix $\Hess(P)$,  the Hessian of $P$ at the origin.  The subspace  $\Sym_N^0$   corresponds to the space  $\h_{2,N}$ of homogeneous,  harmonic polynomials of degree $2$ on $\bR^N$.

The     orthogonal  group  $O(N)$ acts by conjugation  on $\Sym_N$,  and $\Sym_N$ decomposes  into irreducible    components
\[
\Sym_N= \bR\lan \one_N\ran \oplus \Sym_N^0,
\]
where $\bR\lan \one_N\ran$ denotes the one-dimensional space spanned by the identity matrix $\one_N$.

We fix an  $O(N)$-invariant metric on $\Sym_N$. The irreducibility of $\Sym_N^0$ implies that  such a  metric is uniquely determined by two constants $a, b>0$  so that the  collection
\[
a\one_N,\;\; b\Hess(Y),\;\;Y\in \eB_{2,N}
\]
is an orthonormal basis.  We denote  by $|-|_{a,b}$ the norm of this metric.  We want to express $|A|_{a,b}$ in terms of $\tr A^2$ and $(\tr A)^2$. 

Note first that
\[
|\one_N|_{a,b}^2=\frac{1}{a^2} =\frac{1}{Na^2}|\one_N|_*^2.
\]
The irreducibility of $\Sym_k^0$ implies that there exists a universal constant $R=R_N>0$ such that for any homogeneous  harmonic polynomial $P$ of degree $2$ in $N$ variables we have
\[
|\Hess(P)|_*^2=R^2\int_{S^{N-1}}P(\bx)^2\,|dS(\bx)|.
\]
If we take  $P=x_1x_2$, we deduce
\[
2= R^2\int_{S^{N-1}}x_1^2x_2^2\,|dS(\bx)|,
\]
and using (\ref{eq: bclass}), we deduce
\begin{equation}
R^2=\frac{\Gamma(\frac{N+4}{2})}{\Gamma(3/2)^2\Gamma(1/2)^{N-2}}=\frac{4\Gamma(\frac{N+4}{2})}{\pi^{\frac{N}{2}}}.
\label{eq: cab1}
\end{equation}
We see that for any $P\in \eH_{2,N}$
\[
|Y|_{a,b}^2=\frac{1}{b^2}\|Y\|_{L^2(S^{N-1})}^2=\frac{1}{b^2R^2}|\Hess(Y)|_*^2.
\]
In particular, we deduce that
\[
|-|_*=|-|_{a_*,b_*},\;\;a_*=\frac{1}{N},\;\;b_*=\frac{1}{R}.
\]
In general, if $A\in \Sym_N$, then we  have  a decomposition
\[
A=\frac{1}{N}(\tr A)\one_N  +\Bigl(\, A-\frac{1}{N}(\tr A)\one_N\,\Bigr)
\]
 that is orthogonal with respect to both $|-|_*$ and $|-|_{a,b}$. We deduce 
 \[
 |A|_{a,b}^2=\left|\frac{1}{N}(\tr A)\one_N\right|_{a,b}^2+\left|(A-\frac{1}{N}(\tr A) \one_N)\right|_{a,b}^2
 \]
 \[
 =\frac{1}{N^2a^2}(\tr A)^2 +\frac{1}{b^2R^2}\tr \bigl(\, A^2- \frac{2}{N}(\tr A) A+ \frac{1}{N^2}(\tr A)^2\one_N\,\bigr)=
 \]
 \begin{equation}
 =\underbrace{\frac{1}{N}\left(\frac{1}{Na^2}-\frac{1}{b^2R^2}\right)}_{=:\beta} (\tr A)^2 +\underbrace{\frac{1}{b^2R^2}}_{=:\alpha}\tr A^2.
 \label{eq: cab}
 \end{equation}
 Note that the quantities $\alpha,\beta$ depend on $a, b$ and the dimension $N$.

 If $|d\eV_*|$ denotes the volume density determined by the metric $|-|_*$ and $|d\eV_{a,b}|$ denotes the volume density   associated to the metric $|-|_{a,b}$, then we have
 \begin{equation}
 |d\eV_{a,b}|=C_N(a,b)|d\eV_*|,\;\;C_N(a,b):= \frac{1}{ aN^{1/2} (bR)^{D_N-1}}.
 \label{eq: cgamma}
 \end{equation}
 Suppose now that $f:\Sym_N\ra \bR$ is a continuous $O(N)$-invariant function that is homogeneous of degree $\ell >0$. We want  to  find a simpler expression for the integral
\[
J_{a,b}(f):=\int_{\Sym_N}  e^{-|A|_{a,b}^2} f(A)\, |d \eV_{a,b}(A)|.
\]
This can be reduced to a situation frequently encountered in random matrix theory.  We have
\[
J_{a, b}(f)=C_N(a,b)\int_{\Sym_N} e^{-|A|^2_{a,b}} f(A) |d\eV_*(A)|= C_N(a,b)\int_{\Sym_N} e^{-\alpha\tr A^2-\beta(\tr A)^2} f(A)\,|d\eV_*(A)|
\]
$A=\frac{1}{\sqrt{2\alpha}}B$
\[
=\frac{\pi^{\frac{D_N}{2}}C_N(a,b)}{\alpha^{\frac{D_N}{2}}}\int_{\Sym_N} \underbrace{e^{-\frac{\beta}{2\alpha}(\tr B)^2} f\left(\frac{1}{\sqrt{2\alpha}}B\right)}_{=:\Phi_{\alpha,\beta}(B)}\, \frac{e^{-\frac{1}{2}\tr B^2}}{(2\pi)^{\frac{D_N}{2}}}\,|d\eV_*|
\]
Observe that the  function $B\mapsto \Phi_{\alpha,\beta}(B)$ is also $O(N)$ invariant.    We denote by  $\eD_N\subset \Sym_N$ the subspace consisting of diagonal matrices.        We identify $\eD_N$ with $\bR^N$ in the obvious fashion. Using \cite[Prop. 4.1.1]{AGZ} or \cite[Thm. 2.50]{DG} we deduce that
\[
\int_{\Sym_N}  \Phi_{\alpha,\beta}(B) \frac{e^{-\frac{1}{2}\tr B^2}}{(2\pi)^{\frac{D_N}{2}}}\,|d\eV_*|=\frac{1}{Z_N}\int_{\eD_N}\Phi_{\alpha,\beta}(B)\,|\Delta(B)|\,\cdot\, \frac{e^{-\frac{1}{2}\tr B^2}}{(2\pi)^{\frac{D_N}{2}}}|dV(B)|,
\]
where

\begin{itemize}

\item $\Delta(\bx)$  is  the   discriminant $\Delta(x_1,\dotsc, x_N)=\prod_{1\leq i< j\leq N}(x_i-x_j)$

\item   The constant $Z_N$ is given by the integral
\[
Z_N= \int_{\eD_N} |\Delta(B)|\,\cdot\, \frac{e^{-\frac{1}{2}\tr B^2}}{(2\pi)^{\frac{D_N}{2}}}|dV(B)|=(2\pi)^{- \frac{D_N-N}{2} }\int_{\bR^N} |\Delta(\bx)| \frac{e^{-\frac{|\bx|^2}{2} }}{(2\pi)^{\frac{N}{2}}}\,|dV(\bx)|.
\]
\[
=(2\pi)^{- \frac{D_N-N}{2} }\prod_{j=0}^{N-1}\frac{\Gamma(1+\frac{j}{2})}{\Gamma(\frac{3}{2})}.
\]
\end{itemize}
Putting together all of the above, we deduce
\begin{equation}
J_{a,b}(f)=\frac{\pi^{\frac{D_N}{2}}C_N(a,b)}{Z_N\alpha^{\frac{D_N}{2}}} \int_{\bR^N} e^{-\frac{1}{2}\bx|^2-\frac{\beta}{2\alpha}(x_1+\cdots+x_N)^2}f\left(\frac{\bx}{\sqrt{2\alpha}}\right)  |\Delta(\bx)|\, |dV(\bx)|.
\label{eq: c1}
\end{equation}
In particular, we have
\begin{multline}
\int_{\Sym_N}e^{-|A|^2_{a,b}} |\det (A)|\,|d\eV_{a,b}(A)| \\
=\frac{\pi^{\frac{D_N}{2}}C_N(a,b)}{Z_N(2\alpha)^{\frac{N}{2}}\alpha^{\frac{D_N}{2}}} \int_{\bR^N} e^{-\frac{|\bx|^2}{2}-\frac{\beta}{2\alpha}(\sum_{i=1}^N x_i)^2}\prod_{i=1}^N|x_i|\cdot  |\Delta(\bx)|\, |dV(\bx)|,
\label{eq: c2}
\end{multline}
where $\alpha,\beta$ are defined by (\ref{eq: cab}) and $C_N(a,b)$ by (\ref{eq: cgamma}).

Let us point out that, up to a universal multiplicative constant, the  measure $ e^{-\frac{1}{2}\tr A^2}|d\eV|_*(A)$ is the probability distribution of  the real gaussian ensemble, \cite{AGZ, DG}.  As explained in \cite[Chap.3]{DG}, the multidimensional integral  (\ref{eq: c2}) can be reduced to computations of $1$-dimensional integrals  in the special case when $\beta=0$, i.e., $ka^2=b^2R^2$. As explained in \cite[\S 1.5]{For}, \cite{Fy}, the case  $\beta<0$  can be reduced to computations of $1$-point correlations of the Gaussian ensemble   of   $(k+1)\times (k+1)$-matrices. In turn, these can be reduced to computations of $1$-dimensional integrals \cite[\S 4.4]{DG}, \cite[Chap. 6]{For}, \cite[Chap. 7]{Me}.

\section{Some elementary  estimates}
\label{s: d}
\setcounter{equation}{0}

\noindent {\bf Proof of Lemma \ref{lemma: bxi}.} Consider the complex valued random process
\[
\eF_\nu(t):= \frac{1}{\sqrt{\pi\nu^3}} \sum_{m=1}^\nu m z_m   e^{\frac{\ii mt}{\nu}},\;\;z_m=c_m-\ii d_m.
\]
The covariance function of this process is
\[
\eR_\nu(t)=\bsE\bigl( \eF_\nu(t)\overline{\eF}_\nu(0)\,\bigr)=\frac{1}{\pi\nu^3}\sum_{m=1}^\nu m^2 e^{\frac{\ii mt}{\nu}}.
\]
Observe that $\re \eF_\nu=\Phi_\nu$. Note that the spectral measure of the process $\eF_\nu$ is
\[
d\si_\nu=\frac{1}{\pi\nu^3}\sum_{m=1}^\nu m^2\delta_{\frac{m}{\nu}},
\]
where $\delta_{t_0}$ denotes the Dirac measure on $\bR$ concentrated at $t_0$. We form the covariance matrix of the  gaussian vector valued random variable 
\[
\left[
\begin{array}{c}
\eF_\nu(0)\\
\eF_\nu(t)\\
\eF_\nu'(0)\\
\eF_\nu'(t)
\end{array}
\right] \Lra \eX_\nu(t)=\left[\begin{array}{cccc}
\eR_\nu(0)  &  \overline{\eR}_\nu(t) &   \overline{\eR}_\nu'(0) & \overline{\eR}'_\nu(t)\\
\eR_\nu(t) & \eR_\nu(0) & -\ii \eR'_\nu(t) & -\ii\overline{\eR}'_\nu(0)\\
 \eR'_\nu(0) &\ii \overline{\eR}_\nu'(t) & - \eR_\nu''(0) & -\overline{\eR}_\nu'(t) \\
 \eR'_\nu(t) & \ii\eR_\nu'(0) & -\eR_\nu''(t)& -\eR_\nu''(0)
\end{array}
\right].
\]
Observe that $\re\eX(t)=\re\bXi(t)$. If we let
\[
\vec{z}=\left[
\begin{array}{c}
u_0\\
v_0\\
u_1\\
v_1
\end{array}
\right]\in\bC^4
\]
Then, as in \cite[Eq. (10.6.1)]{CL} we have
\[
\lan\eX_\nu\vec{z},\vec{z}\ran =\frac{1}{\pi\nu^3}\sum_{m=1}^\nu m^2\left| \left(u_0+v_0e^{\frac{\ii m t}{\nu}}\right)+\frac{\ii m}{\nu}\left(u_1+ v_1e^{\frac{\ii mt}{\nu}}\right)\right|^2
\]
We see that 
\begin{equation}
\lan\eX_\nu\vec{z},\vec{z}\ran =0\,  \Llra \,  \left(u_0+v_0e^{\frac{\ii m t}{\nu}}\right)+\frac{\ii m}{\nu}\left(u_1+ v_1e^{\frac{\ii mt}{\nu}}\right) =0,\;\;\forall m=1,\dotsc, \nu.
\label{eq: ex}
\end{equation}
We see that if the linear system (\ref{eq: ex}) has a nontrivial solution $\vec{z}$ then the  complex $4\times 4$ matrix
\[
A_\nu(t):=\left[
\begin{array}{cccc}
1&  \zeta & 1  &  \zeta \\
1&  \zeta^2 & 2  &  2 \zeta^2 \\
1&  \zeta^3 & 3  &  3\zeta^3 \\
1&  \zeta^4 & 4  &  4 \zeta^4 
\end{array}
\right],\;\;\zeta= e^{\frac{\ii t}{\nu}},
\]
must be  singular, i.e.,  $\det A_\nu(t)=0$. We  have 
\[
\det A_\nu(t)=\det\left[
\begin{array}{cccc}
1&  \zeta & 1  &  \zeta \\
0&  \zeta^2-\zeta & 1  &  2 \zeta^2 -z \\
0&  \zeta^3-\zeta & 2  &  3\zeta^3 -z\\
0&  \zeta^4-z & 3  &  4 \zeta^4 -z
\end{array}
\right]= \zeta^2\det\left[
\begin{array}{cccc}
1&           1                     & 1   &  1 \\
0&  \zeta-  1 & 1   &  2 \zeta -1 \\
0&  \zeta^2-1 & 2   &  3\zeta^2 -1\\
0&  \zeta^3-1 & 3   &  4 \zeta^3 -1
\end{array}
\right]
\]
\[
=\zeta^2\det\left[
\begin{array}{ccc}
\zeta-  1 & 1   &  2 \zeta -1 \\
  \zeta^2-1 & 2   &  3\zeta^2 -1\\
  \zeta^3-1 & 3   &  4 \zeta^3 -1
  \end{array}
  \right]= \zeta^2 \det\left[
\begin{array}{ccc}
\zeta-  1 & 1   &   \zeta  \\
  \zeta^2-1 & 2   &  2 \zeta^2 \\
  \zeta^3-1 & 3   &  3 \zeta^3 
  \end{array}
  \right]
  \]
  \[
  = \zeta^3 \det\left[
\begin{array}{ccc}
\zeta-  1 & 1   &   1  \\
  \zeta^2-1 & 2   &  2 \zeta \\
  \zeta^3-1 & 3   &  3 \zeta^2 
  \end{array}
  \right]= \zeta \det\left[
\begin{array}{ccc}
\zeta-  1 & 1   &   0  \\
  \zeta^2-1 & 2   &  2 \zeta^2-2 \\
  \zeta^3-1 & 3   &  3 \zeta^3 -3
  \end{array}
  \right]
  \]
  \[
  =\zeta^3 \det\left[
\begin{array}{ccc}
\zeta-  1                                  & 1   &   0  \\
  \zeta^2-2\zeta+1 & 0   &  2 \zeta^2-2 \\
  \zeta^3-3\zeta +2 & 0   &  3 \zeta^3 -3
  \end{array}
  \right]=\zeta^3 \det\left[
\begin{array}{ccc}
\zeta-  1                                  & 1   &   0  \\
 (\zeta-1)^2 & 0   & 2  (\zeta-1)(\zeta+1)\\
  \zeta^3-3\zeta +2 & 0   &  3 \zeta^3 -3
  \end{array}
  \right]
  \]
  \[
  =\zeta^3(\zeta-1)\det\left[
\begin{array}{ccc}
\zeta-1                                   & 1   &   0  \\
 (\zeta-1) & 0   & 2 (\zeta+1)\\
  (\zeta-1)^2(\zeta+1)& 0   &  3(\zeta-1)(\zeta^2+\zeta+1)
  \end{array}
  \right]
  \]
  \[
  =\zeta^3(\zeta-1)^3 \det\left[
\begin{array}{ccc}
1                                 & 1   &   0  \\
 1& 0   & 2 (\zeta+1)\\
  \zeta+1)& 0   &  3(\zeta^2+\zeta+1)
  \end{array}
  \right]= -\zeta^3(\zeta-1)^3 \left(\, 3(\zeta^2+\zeta+1)- 2(\zeta^2+2\zeta+1)\,\right)
  \]
  \[
  =-\zeta^3(\zeta-1)^3(\zeta^2-\zeta-1).
  \]
  Since $|\zeta|=1$, we  that $\det A_\nu(t)=0$ if and only if $t\in 2\pi\nu\bZ$.
\qed

\noindent {\bf Proof of Lemma \ref{lemma: sharp2}.}  Recall that $\theta:=\frac{t}{2\nu}$,   $f(\theta):=\frac{\sin\theta}{\theta}$.    By (\ref{eq: sharp}) we have 

\[
\frac{t^{r+1}}{\ii^r\nu^{r+1}}D_{\nu, r}(t)=r!\left(\,\frac{2\sin\left(\frac{(\nu+1)t}{2\nu}\right)}{f(\theta)^{r+1}}\cdot   e^{\frac{\ii(\nu+r)t}{2\nu} }-e^{\ii t}\sum_{j=1}^r\ii^{1-j}\frac{\binom{\nu+1}{j}}{\nu^j}t^j\cdot \left(\frac{e^{\ii\theta}}{f(\theta)}\right)^{r+1-j}\right).
\]
Using (\ref{eq: ir}) we deduce that
\begin{equation}
\begin{split}
\left|\frac{t^{r+1}}{\ii^r\nu^{r+1}} D_{\nu, r}(t))- \frac{1}{\ii^r}I_r(t)\,\right| & \leq 2r!t^{r+1}\left|\frac{ \sin\left(\frac{(\nu+1)t}{2\nu}\right)}{f(\theta)^{r+1}}e^{\ii r\theta}-\sin\left(\frac{t}{2}\right)\right|\\
&+r!\sum_{j=1}^{r} t^j\left|\frac{\binom{\nu+1}{j}}{\nu^j}\left(\frac{e^{\ii\theta}}{f(\theta)}\right)^{r+1-j} -\frac{1}{j!}\right|.
\end{split}
\label{eq: d1}
\end{equation}
In the sequel  we will use  Landau's symbol $O$.  These implied  constants  will be independent of $\nu$.   Also  we will denote by the same symbol $C_r$ constants independent  of $\nu$ put possibly dependent on $r$. Throughout we assume $0< t\leq \pi\nu$.  Then $0<\theta<\frac{\pi}{2}$ and for $0\leq j\leq r$ we have
\[
e^{\ii\theta}=1+O(\theta),\;\;\frac{\binom{\nu+1}{j}}{\nu^j}=1+O\left(\frac{1}{\nu}\right),\;\;\sin\left(\frac{(\nu+1)t}{2\nu}\right)=\sin\left(\frac{t}{2}\right) +O(\theta),
\]
\[
\frac{e^{\ii\theta}}{f(\theta)}= \frac{\theta(\cos\theta+\ii\sin\theta)}{\sin\theta}=1+O(\theta).
\]
Hence 
\begin{equation}
\left|\frac{ \sin\left(\frac{(\nu+1)t}{2\nu}\right)}{f(\theta)^{r+1}}e^{\ii r\theta}-\sin\left(\frac{t}{2}\right)\right|=O(\theta),
\label{eq: est30}
\end{equation}
while for any $1\leq j\leq r$ we have
\begin{equation}
\left|\frac{\binom{\nu+1}{j}}{\nu^j}\left(\frac{e^{\ii\theta}}{f(\theta)}\right)^{r+1-j} -\frac{1}{j!}\right|= O\left(\frac{1}{\nu}+\theta\right).
\label{eq: est31}
\end{equation}
Using  (\ref{eq: est30}) and (\ref{eq: est31})   in (\ref{eq: d1}) we deduce that 
\[
\left|\frac{t^{r+1}}{\ii^r\nu^{r+1}} D_{\nu, r}(t))- \frac{1}{\ii^r}I_r(t)\,\right| =O\left( \theta+\left(\theta+\frac{1}{\nu}\right)\sum_{j=1}^{r-1} t^j\right)=O\left(\frac{1}{\nu}\sum_{j=1}^{r+1}t^j\right).
\]
Hence
\[
\left|\, \frac{1}{\ii^r\nu^{r+1}} D_{\nu, r}(t))- \frac{1}{\ii^r t^{r+1}}I_r(t)\,\right|=O\left(\frac{1}{\nu t^r}\frac{(1-t^{r+1})}{ (1-t)}\,\right).
\]
\qed

\noindent {\bf Proof of Lemma \ref{lemma: cov-func}.}      We have
\[
R_\infty^{(k)}(t) =\pm \frac{1}{t^{k+3}}\int_0^t \tau^{k+2}  u(\tau)  d\tau,\;\;u(\tau)=\begin{cases}
\sin\tau, & k\in 1+2\bZ\\
\cos\tau, & k\in 2\bZ.
\end{cases}.
\]
Note that
\[
|\Ri^{(k)}(0)|=\frac{1}{k+3}=\frac{1}{t^{k+3}}\int_0^t \tau^{k+2} d\tau.
\]
The inequality (\ref{eq: diffa1}) now follows  from the inequality $|u(\tau)|\leq 1$, $\forall \tau$.

For any positive integer $r$ we denote by $\jj_r$ the $r$-th jet at $0$ of  a one-variable function.  We can rewrite (\ref{eq: ir}) as follows:
\[
\frac{1}{\ii^r} I_r(t)=r!\left(2\sin\left(\frac{t}{2}\right)e^{\frac{\ii t}{2}}-\ii e^{\ii t} \sum_{j=1}^r\frac{(-\ii t)^j}{j!}\right)=r!\left(2\sin\left(\frac{t}{2}\right)e^{\frac{\ii t}{2}}-\ii e^{\ii t}\cdot \jj_r\bigl( e^{-\ii t} -1\bigr)\,\right)
\]
\[
=r!\left(2\sin\left(\frac{t}{2}\right)e^{\frac{\ii t}{2}}+\ii e^{\ii t}-\ii e^{\ii t}\cdot \jj_r\bigl( \, e^{-\ii t}\,\bigr)\,\right)
\]
\[
=r!\left( -\ii\bigl(e^{\frac{\ii t}{2}}-e^{-\frac{\ii t}{2}}\bigr)e^{\frac{\ii t}{2}} +\ii e^{\ii t}-\ii e^{\ii t}\cdot \jj_r\bigl( \, e^{-\ii t}\,\bigr)\,\right)=\ii r!\left(1-e^{\ii t}\cdot \jj_r\bigl( \, e^{-\ii t}\,\bigr)\,\right)
\]
Hence
\[
\re\left(\frac{1}{\ii^r}I_r(t)\,\right)=\im\left(\, e^{\ii t}\cdot \jj_r\bigl( e^{-\ii t}\bigr)\,\right)\;\;\mbox{and}\;\;\frac{1}{t^{r+1}}I_r(t)=O(t^{-1})\;\;\mbox{$t\ra\infty$}.
\]
This proves (\ref{eq: diffb1}).

The  spectral measure 
\[
d\si_\nu=\frac{1}{\pi\nu^3} \sum_{m=1}^\nu m^2\delta_{\frac{m}{\nu}}
\]
 of the process $\eF_\nu$ converges  weakly as $\nu\ra \infty$ to the measure  
\[
d\si_\infty=\frac{1}{\pi}\chi_{[0,1]} t^2 dt,
\]
where $\chi_{[0,1]}$ denotes the characteristic function of  $[0,1]$.    Indeed, an argument identical to the one used in the proof of Lemma \ref{lemma: cov} shows that  for every continous bounded function $f:\bR\ra \bR$ we have
\[
\lim_{\nu\ra \infty}\int_{\bR} f(t) d\si_\nu(t) =  \int_{\bR} f(t) d\si_\infty.
\]
The complex  valued  stationary Gaussian  process  $\eF_\infty$ on $\bR$ with spectral measure $d\si_\infty$  has covariance function 
\[
\eR_\infty= \frac{1}{\pi}\int_0^1 t^2 e^{\ii t} dt.
\]
Note that $\re \eR_\infty= \Ri$.  The results in  \cite[\S 10.6]{CL} show that the covariance matrix
\[
\eX_\infty=\left[\begin{array}{cccc}
\eR_\infty(0)  &  \overline{\eR}_\infty(t) &   \overline{\eR}_\infty'(0) & \overline{\eR}'_\infty(t)\\
\eR_\infty(t) & \eR_\infty(0) & -\ii \eR'_\infty(t) & -\ii\overline{\eR}'_\infty(0)\\
 \eR'_\infty(0) &\ii \overline{\eR}_\infty'(t) & - \eR_\nu''(0) & -\overline{\eR}_\infty'(t) \\
 \eR'_\infty(t) & \ii\eR_\infty'(0) & -\eR_\infty''(t)& -\eR_\infty''(0)
\end{array}
\right]=\lim_{\nu\ra \infty}\eX_\nu,
\]
is nondegenerate. The equality $\det\re\eX_\infty(t)\neq 0$, $\forall t\in\bR$ implies as in Remark \ref{rem: bxi} that $\mu_\infty(t)\neq 0$, $|\rho_\infty(t)|<1$, $\forall t\in\bR$, where
\[
\mu_\infty=\frac{(\bla_0^2-R_\infty^2)\bla_2-\bla_0(R_\infty')^2}{\lambda_0^2-R_\infty^2}, \;\;\rho_\infty=\frac{\Ri''(\bla_0^2-\Ri^2)+(\Ri')^2R_\infty}{(\bla_0^2-\Ri^2)\bla_2-\bla_0(\Ri')^2}.
\]
This proves (\ref{eq: diffc1}). The  equality (\ref{eq: diffd1}) follows from the Taylor expansion of $\Ri$.\qed

\end{document}